\documentclass[11pt]{article}
\usepackage{amsmath,amsthm,amssymb,dsfont,mathrsfs}
\usepackage[usenames,dvipsnames]{xcolor}
\usepackage{upgreek}
\usepackage{enumerate}
\usepackage{graphicx}
\usepackage{cite}
\usepackage{url}
\usepackage{oands}
\usepackage{tikz}
\usepackage{changepage}
\usepackage{bbm}
\usepackage{mathtools}
\usepackage[margin=1in]{geometry}
\usepackage[pagewise,mathlines]{lineno}
\usepackage{appendix}
\usepackage{stmaryrd} 
\usepackage{multicol}
\usepackage{microtype}
\usepackage[colorinlistoftodos]{todonotes}
\usepackage{subfigure}

\definecolor{darkblue}{rgb}{0.0, 0.2, 0.6}
\usepackage[pdftitle={},
  pdfauthor={},
colorlinks=true,linkcolor=RoyalBlue,urlcolor=MidnightBlue,citecolor=darkblue,bookmarks=true,bookmarksopen=true,bookmarksopenlevel=2,unicode=true,linktocpage]{hyperref}
\usepackage[capitalise,noabbrev]{cleveref}

\newcommand{\ssb}{\begin{adjustwidth}{2.5em}{0pt}}
\newcommand{\sse}{\end{adjustwidth}}

\usepackage[utf8]{inputenc} 
\usepackage[T1]{fontenc} 

\usepackage
[final]
{changes}
\makeatletter
\@namedef{Changes@AuthorColor}{magenta}
\colorlet{Changes@Color}{magenta}
\makeatother
\usepackage{enumitem}
\usepackage{float}
\usepackage[labelfont={bf,sf}, textfont=sf]{caption}
\usepackage{chngcntr}

\newcommand\N{\mathbbm{N}}
\newcommand\Z{\mathbbm{Z}}

\newcommand\R{\mathbbm{R}}
\newcommand\C{\mathbbm{C}}
\newcommand\Ub{\mathbbm{U}}

\newcommand\Pb{\mathbbm{P}}
\newcommand\Eb{\mathbbm{E}}

\newcommand{\Mcal}{\mathcal{M}}

\newcommand{\Pcal}{\mathcal{P}}
\newcommand{\Ecal}{\mathcal{E}}
\newcommand\Ebf{\mathbf{E}}
\newcommand\Pbf{\mathbf{P}}
\newcommand\Phat{\widehat{\mathcal{P}}}
\newcommand\Ehat{\widehat{\mathcal{E}}}

\newcommand{\Xcal}{\mathcal{X}}
\newcommand{\Xbf}{\mathbf{X}}
\newcommand{\Lcal}{\mathcal{L}}
\newcommand{\Xhat}{\widehat{\mathcal{X}}}
\newcommand{\Xbfhat}{\widehat{\mathbf{X}}}
\newcommand{\xihat}{\widehat{\xi}}
\newcommand{\Thetahat}{\widehat{\Theta}}

\newcommand\Fhat{\widehat{F}}
\newcommand\psihat{\widehat{\psi}}
\newcommand\qhat{\widehat{q}}
\newcommand\Ghat{\widehat{G}}

\newcommand\Kcal{\mathcal{K}}

\newcommand{\Fcal}{\mathcal{F}}
\newcommand{\Gcal}{\mathcal{G}}
\newcommand{\Gscr}{\mathscr{G}}

\newcommand{\Ical}{\mathcal{I}}
\newcommand\Jcal{\mathcal{J}}
\newcommand\Jhat{\widehat{\Jcal}}
\newcommand\Hhat{\widehat{H}}

\newcommand\diag{\text{diag}}

\newcommand{\qbf}{\mathbf{q}}

\newcommand\Tcal{\mathcal{T}}

\theoremstyle{plain}
\newtheorem{Thm}{Theorem}[section]
\newtheorem*{Thm*}{Theorem}
\newtheorem{Prop}[Thm]{Proposition}
\newtheorem*{Prop*}{Proposition}

\newtheorem*{Def*}{Definition}
\newtheorem{Lem}[Thm]{Lemma}
\newtheorem*{Cor*}{Corollary}

\theoremstyle{definition}

\newtheorem*{Ex*}{Example}
\newtheorem{Rk}[Thm]{Remark}

\begin{document}

\vglue30pt

\begin{center}
    \Large\bf Multitype self-similar growth-fragmentation processes
\end{center}

\bigskip

\centerline{by}

\medskip

\centerline{William Da Silva\footnote{\scriptsize University of Vienna, Austria, {\tt william.da.silva@univie.ac.at}} and Juan Carlos Pardo\footnote{\scriptsize Centro de Investigación en Matemáticas A.C. Calle Jalisco s/n. 36240 Guanajuato, México, {\tt jcpardo@cimat.mx}}}

\bigskip

\bigskip

{\leftskip=2truecm \rightskip=2truecm \baselineskip=15pt \small


\noindent{\slshape\bfseries Summary.}  In this paper, we are interested in multitype self-similar growth-fragmentation processes. More precisely, we investigate a multitype version of the self-similar growth-fragmentation processes introduced by Bertoin, therefore extending the signed case (considered by the first author in a previous work) to finitely many types. Our main result in this direction describes the law of the spine in the multitype setting.  In order to do so, we introduce two genealogical martingales, in the same spirit as in the positive case, which allow us not only to obtain  the law of  the spine but also to  study the  limit of the empirical measure of fragments. 
We stress that our arguments only rely on the structure of the underlying Markov additive processes (MAPs), and hence is more general than the treatment of the signed case. Our methods also require new results on exponential functionals for MAPs and a multitype version of the tail estimates in multiplicative cascades which are interesting in their own right.  

\medskip


 

} 

\bigskip
\bigskip

\section{Introduction}
Self-similar growth-fragmentation processes first appeared in \cite{Ber-GF} to describe the evolution of a cloud of atoms which may grow and dislocate in a binary way. More precisely, these atoms are assumed to have a specific one-dimensional trait of interest, which we can think of as its \emph{mass} or \emph{size}. Initially, the cloud starts from one particle (the common ancestor of all future particles) whose size is a positive quantity evolving in time in a Markovian way. This size will have jumps, and at each negative jump $y<0$, we wish to add to the cloud a new particle whose size at birth will be given by $-y$, at the time when the jump occurs. This creates children of the original ancestor in such a way that the divisions are conservative, that is summing the size of the child and the size of the parent just after division exactly gives the size of the parent before division. Then, the newborn particles evolve independently of the parent, and independently of one another, in the same Markovian way as the parent. We proceed similarly creating the offspring of those particles, thereby introducing the grandchildren, great grandchildren, and so on, of the original ancestor.

Such growth-fragmentation models have been given a striking geometric flavour, in the context of random planar maps. This originated from \cite{BCK} and \cite{BBCK}, where a remarkable class of self-similar growth-fragmentations shows up in the scaling limit of perimeter processes (see \cite{Budd}) in Markovian explorations of Boltzmann planar maps. These growth-fragmentation processes are closely related to stable L\'evy processes with stability parameter  $\theta \in (\frac12, \frac32]$. Since then, the same growth-fragmentation processes were directly constructed in the continuum \cite{MSW} for $1<\theta\le \frac32$ by drawing a CLE exploration on a quantum disc. Moreover, the boundary case $\theta=\frac32$, corresponding to the random triangulations in \cite{BCK}, already appeared in \cite{LG-Rie} as the collection of perimeters obtained when slicing a Brownian disc at heights. The critical Cauchy case $\theta=1$, in turn, corresponds to slicing a Brownian half-plane excursion at heights \cite{AD}. This approach was recently extended to $\frac12<\theta<1$ \cite{DS} by considering other half-plane excursions.

Let us point out that in \cite{AD} (and subsequently in \cite{DS}), \emph{negative} mass is taken into account in the system, whereas the aforementioned construction of growth-fragmentation processes deals with positive mass only. In those examples, the sign depends on the time orientation of the excursions. In particular, slicing a half-plane Brownian excursion only yields the critical case $\theta=1$ in \cite{BBCK} provided one discards the negative cells. On a related note, the driving cell processes in the distinguished family of growth-fragmentations in \cite{BBCK} also have positive jumps (except for $\theta=\frac32$). This has a geometric meaning: Boltzmann planar maps correspond to the gasket of a loop $O(n)$--model, and the positive jumps occur when discovering a loop, which could then be explored. In the continuum, positive jumps also arise in \cite{MSW} when hitting a CLE loop for the first time. This prompted \cite{DS} to provide a framework for self-similar \emph{signed} growth-fragmentations.

Adding negative mass to the system presents some technical issues. The analysis of the positive case carried out in \cite{Ber-GF} and \cite{BBCK} relies heavily on the Lamperti representation \cite{Lam} for positive self-similar Markov processes, allowing for a large toolbox of L\'evy techniques. This breaks down if one is willing to deal with signed processes, in other words the effect of introducing a sign is to move from the class of Lévy processes to the one of Markov additive processes, see for instance \cite{CPR}, \cite{KKPW}, \cite{KP} and \cite{PR}. One of the aims of  this paper is to extend the framework to a general set of types. This has a counterpart in the \emph{pure} fragmentation setting, see for instance \cite{Ste}. In this case, we show that natural martingales arise, in connection to the additive martingales appearing in the context of multitype branching random walks (\cref{sec: multitype GF}). These martingales have the same form as in \cite{Ber-GF}, except that they are weighted by the types. The pair of genealogical martingales arises naturally as intrinsic martingales associated with that of multitype branching random walks. We use a multitype version of Biggins'  martingale convergence theorem due to Kyprianou and Sani \cite{KS} to deduce that one of the martingales converges to 0 almost surely. The other martingale is uniformly integrable and the tail probability behaviour  of its terminal value  decreases polynomially. In order to deduce the latter, we extend to the multitype setting the tail estimates in multiplicative cascades due to  Jalenkovi\'c and Olvera-Cravioto \cite{JO-C}.
Following the same lines as \cite[Theorem 4.2]{BBCK}, our main theorem in the multitype setting (\cref{thm:spine multitype}) describes the cell system under the change of measures with respect to these martingales (\cref{sec: spine}). We stress that, although the framework developed here includes the signed case which was already treated in \cite{DS}, our methods are completely different and moreover further results are included which are not treated in the signed case. Indeed, \cite{DS} hinges upon a change of driving cell process (which is specific to the signed case) to reduce to the positive case, whereas in this paper we directly work with Markov additive processes (MAPs). Using properties of MAPs, we then arrive at a pair of martingales indexed by continuous time and which are naturally related to the aforementioned genealogical martingales which allow us to describe explicitly the dynamics of multitype growth-fragmentations under the probability measures which are obtained by tilting the initial one with these intrinsic martingales. Such intrinsic martingales are also applied to  study the  limit of the empirical measure of fragments. Along the way, we provide additional results on exponential functionals for MAPs and the aforementioned multitype version of the tail estimates in multiplicative cascades which are interesting in their own right. 

\medskip
\noindent \textbf{Acknowledgments.} We thank Élie Aïdékon and Andreas Kyprianou for some interesting discussions. W.D.S. acknowledges the support of the Austrian Science Fund (FWF) grant P33083-N on ``Scaling limits in random conformal geometry''.  J.C.P.  acknowledges the support of CONACyT grant A1-S-33854.

\tableofcontents

\section{Self-similar Markov processes with types} \label{sec: ssmp}
We start by presenting some shared features of the Markov processes we will be interested in, revolving around the notion of self-similarity. We explain how to deal with \emph{types} for self-similar processes in the case when the set of types is finite. A key ingredient of our analysis is the Lamperti-Kiu representation, which gives a bijection between these self-similar  processes and a class of Markov additive processes. We refer to \cite{KP} for a detailed treatment of these questions.

\bigskip
\noindent \textbf{Markov additive processes.}
Let $\Ical$ be  a finite set. We also let $(\xi(t),\Theta(t), t\ge 0)$ be a regular Feller process in $\R\times \Ical$ with probabilities ${\tt P}_{x,\theta}$, $x\in\R$, $\theta\in \Ical$ and cemetery state $(-\infty, \dagger)$, on $(\Omega,\Fcal,\Pb)$, and denote by $(\Gcal_t)_{t\ge 0}$ the natural standard filtration associated with $(\xi,\Theta)$. We say that $(\xi,\Theta)$ is a \emph{Markov additive process} (MAP for short) if for every bounded measurable $f: \R\times \Ical \rightarrow \R$, $s,t\ge 0$ and $(x,\theta)\in\R\times \Ical$,
\[
\mathtt{E}_{x,\theta}\Big[f(\xi(t+s)-\xi(t),\Theta(t+s))\mathds{1}_{\{t+s<\varsigma\}} \Big| \Gcal_t\Big] = \mathds{1}_{\{t<\varsigma\}}\mathtt{E}_{0,\Theta(t)}\Big[f(\xi(s),\Theta(s))\mathds{1}_{\{s<\varsigma\}}\Big],
\]
where $\varsigma := \inf\{t>0, \, \Theta(t)= \dagger\}$. The process $\Theta$ is thus a Markov chain on $\Ical$ and is called the \emph{modulator} of $\xi$, whereas the latter is called the  \emph{ordinator}. The notation 
\[
\mathtt{P}_{\theta}(\cdot):=\mathbb{P}(\;\cdot\; |\,\xi(0)=0 \,\text{ and }\, \Theta(0)=\theta)\quad\textrm{ for }\quad \theta\in E,
\] will be in force throughout the paper.   

MAPs have received  a lot of attention recently,  since they have found a prominent role in classical applied probability  (see for instance \cite{Asm} and \cite{Iva}) and in the study of self-similar Markov processes (see for instance \cite{CPR, KKPW, KP}). It is important to note that the concept of MAPs also makes sense when  $\Theta$ is replaced by a more general Markov process,  see for instance \cite{Cin, KP}, but this case is out of the scope of this manuscript.  For a treatment of the general  case in the context of self-similar growth-fragmentation we refer to Da Silva and Pardo \cite{DP23}.

Informally, one should think of a MAP as a natural extension of a L\'evy process in the sense that $\Theta$ is an arbitrary well-behaved Markov chain and $((\xi(t), \Theta(t))_{t\ge 0}, \mathtt{P}_{x,\theta})$ is equal in law to  
$((\xi(t)+x, \Theta(t))_{t\ge 0}, \mathtt{P}_{\theta})$.  Indeed the ordinate process $\xi$ can be thought of as the concatenation of L\'evy processes  which depend on the current type in $\Ical$ given by $\Theta$, as  is stated in the following proposition, see \cite{Iva}, \cite{KKPW}, \cite{KP} or the survey \cite{PR}.
\begin{Prop} \label{prop:MAP structure}
The process $(\xi,\Theta)$ is a Markov additive process if and only if there exist independent sequences $(\xi_i^{(n)}, n\ge 0)_{i\in \Ical}$ and $(U_{i,j}^{(n)}, n\ge 0)_{i,j\in \Ical}$, all independent of $\Theta$, such that:
\begin{itemize}
    \item for $i\in \Ical$, $(\xi_i^{(n)}, n\ge 0)$ is a sequence of i.i.d. Lévy processes,
    \item for $i,j\in \Ical$, $(U_{i,j}^{(n)}, n\ge 0)$ are i.i.d. random variables,
    \item if $(T_n)_{n\ge 0}$ denotes the sequence of jump times of the chain $\Theta$ (with the convention $T_0=0$), then for all $n\ge 0$,
\begin{equation} \label{eq:piecewise}
\xi(t) = \left(\xi(T_n^-)+ U_{\Theta(T_n^-),\Theta(T_n)}^{(n)}\right)\mathds{1}_{\{n\ge 1\}} + \xi_{\Theta(T_n)}^{(n)}(t-T_n), \qquad T_n\le t<T_{n+1}.
\end{equation}
\end{itemize}
\end{Prop}

\noindent We now turn to defining the \emph{matrix exponent} of a MAP, which is the analogue of the Laplace exponent in the setting of L\'evy processes. Without loss of generality, we assume that $\Ical=\{1,\ldots, N\}$ where $N\in\N$, and that $\Theta$ is irreducible and ergodic. We write $Q=(q_{i,j})_{1\le i,j\le N}$ for its intensity matrix, and $\rho_i$, $i\in\Ical$, for the exponential time that $\Theta$ takes to jump from state $i$ to some other state. Also, we denote for all $i,j\in \Ical$, all on the same probability space, by $\xi_i$ a Lévy process distributed as the $\xi_i^{(n)}$'s, and by $U_{i,j}$ a random variable distributed as the $U^{(n)}_{i,j}$'s, with the convention $U_{i,i}=0$ and $U_{i,j}=0$ if $q_{i,j}=0$. Finally, we introduce the Laplace exponent $\psi_i$ of $\xi_i$ and the Laplace transform $G_{i,j}(z):= \Eb\left[ \mathrm{e}^{zU_{i,j}}\right]$ of $U_{i,j}$ (this defines a matrix $G(z)$ with entries $G_{i,j}(z)$). Then the matrix exponent $F$ of $(\xi,\Theta)$ is defined as 
\begin{equation} \label{eq: F matrix}
F(z) := \diag(\psi_1(z),\ldots,\psi_N(z))+Q\circ G(z),    
\end{equation}
where $\circ$ denotes pointwise multiplication of the entries. Then the following equality holds for all $i,j\in \Ical$, $z\in \C$, $t\ge 0$, whenever one side of the equality is defined:
\[
\mathtt{E}_{0,i}\left[ \mathrm{e}^{z\xi(t)} \mathds{1}_{\{\Theta(t)=j\}}\right] = (\mathrm{e}^{F(z)t})_{i,j}.
\]
Let us denote the stationary distribution of $\Theta$ by $\pi=(\pi_1, \ldots, \pi_N)$. Given the MAP $(\xi, \Theta)$ with probabilities $\mathtt{P}_{x, \theta}, x\in \mathbb{R}, \theta\in E$, we can introduce the dual process; that is the MAP with probabilities $\mathtt{P}^{\natural}_{x, \theta}, x\in \mathbb{R}, \theta\in E$ whose matrix exponent when it is defined, is given by 
\[
\mathtt{E}^\natural_{0,i}\left[ \mathrm{e}^{z\xi(t)} \mathds{1}_{\{\Theta(t)=j\}}\right] = (\mathrm{e}^{F^\natural(z)t})_{i,j},
\]
where
\[
F^\natural(z):= \diag(\psi_1(-z),\ldots,\psi_N(-z))+Q^\natural\circ G(-z)^{\tt t},
\]
with $A^{\tt t}$ being the transpose matrix of $A$ and $Q^\natural$ is the intensity matrix of the modulating Markov chain on $\mathcal{I}$ with entries given by 
\begin{equation}\label{dualQ}
q^\natural_{i,j}=\frac{\pi_j}{\pi_i} q_{j,i}, \qquad i,j \in \mathcal{I}.
\end{equation}
Observe that the latter can also be written as $Q^\natural=\Delta^{-1}_\pi Q^{\tt t} \Delta_\pi$, where $\Delta_\pi=\diag(\pi_1,\ldots,\pi_N)$ and hence, when it exists, 
\begin{equation} \label{eq: F^natural expression}
F^\natural(z)=\Delta^{-1}_\pi F(-z)^{\tt t} \Delta_\pi,
\end{equation}
showing that 
\[
\pi_i\mathtt{E}^\natural_{0,i}\left[ \mathrm{e}^{z\xi(t)} \mathds{1}_{\{\Theta(t)=j\}}\right] =\pi_j\mathtt{E}_{0,j}\left[ \mathrm{e}^{-z\xi(t)} \mathds{1}_{\{\Theta(t)=i\}}\right].
\]
The previous identity can be understood, at the level of processes, as changing time-directions.
\begin{Lem}\label{dlemma} We have that $(\xi(t-s)-\xi(t), \Theta((t-s)-), 0\le s\le t)$ under $\mathtt{P}_\pi=\sum_{i=1}^N \pi_i \mathtt{P}_{i}$ is equal in law to $(\xi(s), \Theta(s), 0\le s\le t)$ under $\mathtt{P}^\natural_\pi$.
\end{Lem}


\bigskip

\noindent \textbf{Spectral properties of MAPs 
}
We also state for future reference the following classical results (see \cite{Asm, Iva} or the survey \cite{PR}) about the leading eigenvalue of a MAP, often dubbed \emph{Perron-Frobenius} eigenvalue. We consider a MAP $(\xi,\Theta)$ on $\R\times \Ical$ with matrix exponent $F$.

\begin{Prop} \label{prop: Perron-Frobenius MAP}
Let $F$ denote the matrix exponent of some Markov additive process, and $z\in\R$ such that $F(z)$ is well-defined. Then the matrix $F(z)$ has a real simple eigenvalue $\chi(z)$, which is larger than the real parts of all its other eigenvalues. In addition, $\chi(z)$ is associated to a positive eigenfunction $w(z)$.
\end{Prop}

\noindent We say that $(\xi,\Theta)$ satisfies \emph{Cramér's condition} if there exists $\gamma_0>0$ and $\Upsilon\in(0,\gamma_0)$ such that $F$ is defined on $(0,\gamma_0)$ and $\chi(\Upsilon)=0$. The number $\Upsilon$ is then called a \emph{Cramér number}. The leading eigenvalue enables to identify the following \emph{Wald martingale}, which is a multitype version of the exponential martingale for Lévy processes. 
\begin{Prop}\label{prop: wald MAP}
Fix $\gamma$ such that $F(\gamma)$ is well-defined. With the notation of \cref{prop: Perron-Frobenius MAP}, let
\[
\mathcal{W}^{(\gamma)}(t):= \frac{w_{\Theta(t)}(\gamma)}{w_{\Theta(0)}(\gamma)} \mathrm{e}^{\gamma\xi(t)-t\chi(\gamma)}, \quad t\ge 0.
\]
Then $\mathcal{W}^{(\gamma)}$ is a martingale with respect to $(\Gcal_t)_{t\ge 0}$, and under any initial distribution of $(\xi(0),\Theta(0))$. Moreover, the law of $(\xi,\Theta)$ under the corresponding change of measure, that is
\[
\frac{\mathrm{d} \mathtt{P}^{(\gamma)}_{0,i}}{\mathrm{d} \mathtt{P}_{0,i}}\Bigg|_{\mathcal{G}_t}=\mathcal{W}^{(\gamma)}(t), \qquad t\ge 0,
\]
 is that of a Markov additive process with matrix exponent 
\[
F^{(\gamma)}(z) := \mathrm{diag}(w_i(\gamma), i\in\Ical)^{-1} (F(\gamma+z)-\chi(\gamma)\mathrm{Id})\mathrm{diag}(w_i(\gamma), i\in\Ical).
\]
In particular, the leading eigenvalue of $F^{(\gamma)}(z)$ is given by $\chi^{(\gamma)}(z):=\chi(\gamma+z)-\chi(\gamma)$.
\end{Prop}

\noindent The following property will also come in useful.
\begin{Prop} \label{prop: convexity leading eigenvalue}
We take the notation of \cref{prop: Perron-Frobenius MAP}. Let $D$ be an interval of $\R$ on which $F$ is defined. Then the leading eigenvalue $\chi$ is smooth and convex on $D$.
\end{Prop}

\noindent An important quantity associated to a MAP $(\xi,\Theta)$ on $\R\times \Ical$, particularly in view of the lifetime \eqref{eq: zeta lifetime} appearing in the next paragraph in relation to the Lamperti-Kiu transform, is the so-called \emph{exponential functional}, namely
\[
I(\xi) := \int_0^{\infty} \mathrm{e}^{\xi(s)}\mathrm{d}s.
\] 
This quantity has been studied in great detail, first for L\'evy processes (see, notably, \cite{BLR, BarSav, BY,CPY, KP, PR,PatSav} and references therein), and then more recently for MAPs (see in particular \cite{AliWoo, BehSid, KKPW,Ste}). We stress that the study of $I(\xi)$ usually involves the spectral properties of the MAP, and in particular the leading eigenvalue $\chi$. We state for future reference the following result, giving a finiteness criterion for the moments of the exponential functional of Markov additive processes. The case of Lévy processes is also fully understood \cite[Lemma 3]{Riv}. 

\begin{Prop} \label{prop: moments exponential functionals}
Assume that for some $\gamma>0$, $F$ is defined on $[0,\gamma]$, and $\chi(\gamma)<0$. Then $I(\xi)$ has finite moment of order $\gamma$, under any initial distribution of $(\xi,\Theta)$.
\end{Prop}
\begin{proof}
We only prove the statement for $\gamma> 1$ since the case $\gamma\le 1$ is contained in \cite[Proposition 3.6]{KKPW} (although \cite[Proposition 3.6]{KKPW} only states the result for Cramér numbers $\Upsilon<1$, their proof of finiteness of moments carries over to any $\gamma\le 1$ such that $\chi(\gamma)<0$). For the remainder of the proof, we fix $\gamma>1$ and we take $\varepsilon\in(0,\gamma)$. Jensen's inequality provides
\begin{align}
\left(\int_0^{\infty} \mathrm{e}^{\xi(s)} \mathrm{d}s\right)^{\gamma} 
&= \left(\int_0^{\infty} \mathrm{e}^{\xi(s)} (-\chi(\varepsilon))^{-1} \mathrm{e}^{-s\chi(\varepsilon)} \big(-\chi(\varepsilon)\mathrm{e}^{s\chi(\varepsilon)}\mathrm{d}s \big)\right)^{\gamma} \notag \\
&\le (-\chi(\varepsilon))^{1-\gamma} \int_0^{\infty} \mathrm{e}^{\gamma\xi(s)-(\gamma-1)s\chi(\varepsilon)} \mathrm{d}s. \label{eq: proof Cramer MAP}
\end{align}
Now write $C_{\varepsilon} = \max_{i,j\in\Ical} \frac{w_j(\gamma)}{w_i(\gamma)} \cdot (-\chi(\varepsilon))^{1-\gamma} >0$, and let $i\in\Ical$. Taking the $\texttt{P}_{0,i}$--expectation of \eqref{eq: proof Cramer MAP}, a rough estimate yields
\begin{align*}
\texttt{E}_{0,i}\Bigg[\left(\int_0^{\infty} \mathrm{e}^{\xi(s)} \mathrm{d}s\right)^{\gamma} \Bigg]
&\le C_{\varepsilon} \int_0^{\infty} \texttt{E}_{0,i}\Big[\mathcal{W}^{(\gamma)}(s)\mathrm{e}^{s\chi(\gamma)}\Big]\mathrm{e}^{-(\gamma-1)s\chi(\varepsilon)} \mathrm{d}s \\
&= C_{\varepsilon} \int_0^{\infty} \mathrm{e}^{s(\chi(\gamma)-(\gamma-1)\chi(\varepsilon))} \mathrm{d}s,
\end{align*}
by the martingale property of $\mathcal{W}^{(\gamma)}$ in \cref{prop: wald MAP}. Noting that $\chi(\gamma)<0$ and $\chi(\varepsilon)\to 0$ as $\varepsilon\to 0$, we get that $\chi(\gamma)-(\gamma-1)\chi(\varepsilon)<0$ for small enough $\varepsilon$, which completes the proof.
\end{proof}
\begin{Rk} \label{rk: moments exponential functional}
In particular, if $\Upsilon$ is a Cramér number for $(\xi,\Theta)$, then by convexity of $\chi$ the exponential functional of $\xi$ has finite moments of order $\gamma$ for all $\gamma<\Upsilon$. The case when $\Upsilon<1$ already appears in \cite[Proposition 3.6]{KKPW}, but for $\Upsilon>1$ the result does not seem to be contained in the existing literature. 
\end{Rk}
Our next result studies the tail behaviour of $I(\xi)$ under Cram\'er's condition. The case when $N=2$ has  been treated in \cite{AliWoo} using direct computations associated with the matrix exponent  $F$. The case of L\'evy process was studied in \cite{Riv}. Our approach uses the long-term behaviour of MAPs. We recall that $(w_i(z), i\in\Ical)$ is a positive eigenvector associated to the leading eigenvalue $\chi(z)$ of $F(z)$. 
\begin{Prop} \label{prop: tails exponential functionals}
Assume that $(\xi,\Theta)$ satisfies Cramér's condition, with Cramér number $\Upsilon>0$ and that $\xi$ is not concentrated on a lattice. Set for all $i\in\Ical$,
\[
J_i(\xi) := \int_0^{\infty} \frac{w_{\Theta(s)}(\Upsilon)}{w_i(\Upsilon)} \mathrm{e}^{\xi(s)} \mathrm{d}s.
\] 
Then for $i\in \mathcal{I}$, there exists a constant $C_i \ge 0$ such that
\[
{\mathtt{P}_{0,i}}(J_i(\xi)>t)\sim C_i t^{-\Upsilon}, \qquad \textrm{as}\quad t\to\infty. 
\]
In particular, for all $i\in \Ical$, there exist nonnegative constants $C_i^{(1)}$ and $C_i^{(2)}$ such that for $t$ large enough,
\begin{equation}\label{eq: tail probability I(xi)}
C^{(1)}_i t^{-\Upsilon} \le {\mathtt{P}_{0,i}}(I(\xi)>t) \le C^{(2)}_i t^{-\Upsilon}.
\end{equation}
\end{Prop}
We believe that under $\mathtt{P}_{0,i}$, the tail probability of $I(\xi)$ should also be equivalent to $C_i t^{-\Upsilon}$ for some constant $C_i$. Actually we can use the fact that $I(\xi)$ is the solution of a random affine equation and  then  Theorem 4.1 in \cite{Gol} (see  also  \cite{Kes}) but for technical reasons our arguments are better suited to deal with $J_i(\xi)$ and only provide the weaker form \eqref{eq: tail probability I(xi)}.
\begin{proof}
We only prove the first statement since \eqref{eq: tail probability I(xi)} follows easily from it by bounding $w_{\Theta(t)}$ by the maximal or minimal $w_j$. In this proof, we shall write $w_j$ instead of $w_j(\Upsilon)$. Fix $i\in\Ical$ and observe that for all $t\ge 0$,
\[
J_i(\xi)=\int_0^{t} \frac{w_{\Theta(s)}}{w_i} \mathrm{e}^{\xi(s)}  \mathrm{d}s + \frac{w_{\Theta(t)}}{w_i}\mathrm{e}^{\xi(t)}\int_{t}^{\infty} \frac{w_{\Theta(s)}}{w_{\Theta(t)}}\mathrm{e}^{\xi(s)-\xi(t)}  \mathrm{d}s.
\]
From the Markov additive property of $(\xi, \Theta)$, we deduce that, conditionally on $\Theta(t)$, 
\[
\int_{t}^\infty \frac{w_{\Theta(s)}}{w_{\Theta(t)}} \mathrm{e}^{\xi(s)-\xi(t)}  \mathrm{d}s,
\]
is a copy of  $J_{\Theta(t)}(\xi)$, under $\mathtt{P}_{0,\Theta(t)}$, which is further independent of $((\xi_s, \Theta_s), s\le t)$.  In other words, $J_i(\xi)$ is the solution of a multitype random affine equation of the form \eqref{eq: random affine equation}. Remark that the matrix $m$ of \eqref{eq: matrix m(q) affine} is precisely the matrix exponent $F$ of $(\xi,\Theta)$. Hence according to \cref{thm: tail behaviour affine} in the appendix, our result will be fulfilled as long as the following conditions are satisfied
\begin{itemize}
\item[$(i)$] $F(q)$ is well-defined on a domain containing $[\gamma,\Upsilon]$  for some $\gamma\ge 0$, 
\item[$(ii)$] $\chi(\Upsilon) =0$ and $0<\mathtt{E}_{0,i}\Big[\frac{w_{\Theta(t)}}{w_i} \xi(t) \mathrm{e}^{\Upsilon \xi(t)}\Big] <\infty$, 
\item[$(iii)$] $\mathtt{E}_{0,i}[J_i(\xi)^{\beta}]<\infty$ for all $0<\beta<\Upsilon$ and $i\in \Ical$,
\item[$(iv)$] $\mathtt{E}_{0,i}\left[\left(\int_0^{t} \frac{w_{\Theta(s)}}{w_i} \mathrm{e}^{\xi_{s}} \mathrm{d} s\right)^\Upsilon \right]<\infty$ for all $i\in\Ical$.
\end{itemize}
Condition $(i)$ and the first claim in condition $(ii)$ follow from Cramér's condition. To see that $\mathtt{E}_{0,i}\Big[\frac{w_{\Theta(t)}}{w_i} \xi(t) \mathrm{e}^{\Upsilon \xi(t)}\Big]>0$, we note that the function 
\[
h: q \mapsto \mathtt{E}_{0,i}\Big[\frac{w_{\Theta(t)}}{w_i} \mathrm{e}^{q \xi(t)}\Big],
\]
is strictly convex and satisfies $h(0) = h(\Upsilon) = 1$, whence $h'(\Upsilon) = \mathtt{E}_{0,i}\Big[\frac{w_{\Theta(t)}}{w_i} \xi(t) \mathrm{e}^{\Upsilon \xi(t)}\Big] >0$. Condition $(iii)$ is a consequence of \cref{prop: moments exponential functionals} (see also \cref{rk: moments exponential functional}). It remains to prove $(iv)$, which is equivalent to prove that 
\[\mathtt{E}_{0,i}\left[\left(\int_0^{t}  \mathrm{e}^{\xi_{s}} \mathrm{d} s\right)^\Upsilon \right]<\infty, \]
for all $i\in\Ical$. Let $i\in\Ical$. If $\Upsilon\ge 1$, we use Jensen's inequality to get
\[
\mathtt{E}_{0,i}\left[\left(\int_0^{t}  \mathrm{e}^{\xi_{s}} \mathrm{d} s\right)^\Upsilon \right]
\le
t^{\Upsilon} \mathtt{E}_{0,i}\left[\int_0^{t}  \mathrm{e}^{\Upsilon \xi_{s}} \frac{\mathrm{d} s}{t} \right]
\le
M_i t^{\Upsilon},
\]
with $M_i = \underset{j\in\Ical}{\max} \;  \frac{w_i}{w_j}$, since $\mathtt{E}_{0,i}\Big[\frac{w_{\Theta(s)}}{w_i} \mathrm{e}^{\Upsilon \xi(s)}\Big]=1$ for all $s\ge 0$. The arguments for $\Upsilon<1$ are contained in the proof of \cite[Proposition 3.6]{KKPW}, but we repeat them for completeness. In this case, one can bound
\[
\mathtt{E}_{0,i}\left[\left(\int_0^{t}  \mathrm{e}^{\xi_{s}} \mathrm{d} s\right)^\Upsilon \right]
\le
t^{\Upsilon} \mathtt{E}_{0,i}\left[\underset{0\le s\le t}{\sup} \mathrm{e}^{\Upsilon \xi_{s}} \right].
\]
Now we take $p>1$, and bound
\[
\mathtt{E}_{0,i}\left[\underset{0\le s\le t}{\sup} \mathrm{e}^{\Upsilon \xi_{s}} \right]
\le
M^{(p)}_i \mathtt{E}_{0,i}\left[\underset{0\le s\le t}{\sup} \Big(\frac{w_{\Theta(s)}(\Upsilon/p)}{w_i(\Upsilon/p)} \mathrm{e}^{\Upsilon \xi_{s}/p}\Big)^p \right],
\]
with $M^{(p)}_i := \underset{j\in\Ical}{\max} \;  \left(\frac{w_i(\Upsilon/p)}{w_j(\Upsilon/p)}\right)^p$ as before. By convexity of $\chi$, $\chi(\Upsilon/p)<0$, hence
\[
\mathtt{E}_{0,i}\left[\underset{0\le s\le t}{\sup} \mathrm{e}^{\Upsilon \xi_{s}} \right]
\le
M^{(p)}_i \mathtt{E}_{0,i}\left[\underset{0\le s\le t}{\sup} \Big(\mathcal{W}^{(\Upsilon/p)}(s)\Big)^p \right],
\]
where we recall that $\mathcal{W}^{(\Upsilon/p)}(s) := \frac{w_{\Theta(s)}(\Upsilon/p)}{w_i(\Upsilon/p)} \mathrm{e}^{\Upsilon \xi(s)/p - \chi(\Upsilon/p)s}$, $s\ge 0$, is the Wald martingale associated to $\Upsilon/p$. We use Doob's maximal inequality:
\[
\mathtt{E}_{0,i}\left[\underset{0\le s\le t}{\sup} \mathrm{e}^{\Upsilon \xi_{s}} \right]
\le
\left(\frac{p}{p-1}\right)^p M^{(p)}_i \mathtt{E}_{0,i}\Big[\Big(\mathcal{W}^{(\Upsilon/p)}(t)\Big)^p \Big]
\le 
\left(\frac{p}{p-1}\right)^p M^{(p)}_i m^{(p)}_i \mathrm{e}^{-p\chi(\Upsilon/p)} \mathtt{E}_{0,i}[ \mathrm{e}^{\Upsilon \xi_{t}} ],
\]
with $m_i^{(p)} :=  \underset{j\in\Ical}{\max} \;  \left(\frac{w_j(\Upsilon/p)}{w_i(\Upsilon/p)}\right)^p$. This proves condition $(iv)$ above. Therefore  \cref{thm: tail behaviour affine} yields the desired estimate.
\end{proof}

We stress that the representation of MAPs given in Proposition \ref{prop:MAP structure} is quite important in our arguments below.  The extension  of the aforementioned result, as well as other results regarding the matrix exponent and some related martingales,  to  the case when  $\Ical$ is countable seems feasible under some general assumptions on the modulator $\Theta$. However, up to our knowledge, this case has not been rigorously treated in the literature, so we will not explore it  to keep this paper at a reasonable size.

\bigskip

\noindent \textbf{Self-similar Markov processes with types in $\Ical$.} Similarly to the construction  of positive self-similar Markov processes through L\'evy processes, it is possible to build a more general family of self-similar Markov processes using MAPs.  More precisely, let $(\xi,\Theta)$ be a MAP on $\R\times \Ical$, and fix $\alpha\in\R$. We construct the following process $(X,J)$ with values in $\R_+\times\Ical$ with a possible cemetery state $\partial$ \textit{via} a Lamperti-type procedure. First, introduce
\begin{equation} \label{eq: Lamperti time change}
    \varphi(t):= \inf\left\{s>0, \; \int_0^s \mathrm{e}^{\alpha \xi(s)}\mathrm{d}s > t\right\}, \quad t\ge 0.
\end{equation}
Then, for $x>0$, let 
\begin{equation} \label{eq: Lamperti MAP}
X(t) := x \exp(\xi(\varphi(tx^{-\alpha}))) \quad \textrm{and}\quad J(t) := \Theta(\varphi(tx^{-\alpha})), \quad t\ge 0,
\end{equation}
with the convention that $(X(t),J(t))=\partial$ when $t\ge \zeta$ where
\begin{equation} \label{eq: zeta lifetime}
\zeta:=x^{\alpha}\int_0^{\infty} \mathrm{e}^{\alpha \xi(s)}\mathrm{d}s.
\end{equation}
We write $\Pb_{x,i}$, $x>0$, with $i\in\Ical$, for the law of $(X,J)$ started from $(x,i)$, and $\Pb_{i} = \Pb_{1,i}$.
It is plain from this construction that $(X,J)$ is a Markov process, and that $X$ is a \emph{self-similar} process, that is to say for any $c>0$ and for all $x>0$, $i\in\Ical$, 
\begin{equation} \label{eq: ss}
\Big( (cX(c^{-\alpha}s), s\ge 0), \Pb_{x,i}\Big) \overset{d}{=} \Big( (X(s), s\ge 0), \Pb_{cx,i}\Big).
\end{equation}
Conversely, if $(X,J)$ is a Markov process in $\R_+\times\Ical$, such that $X$ is self-similar with index $\alpha$ in the sense of \eqref{eq: ss}, then one can find a MAP such that \eqref{eq: Lamperti MAP} holds, with the time change given in \eqref{eq: Lamperti time change}. This construction is reminiscent of the Lamperti or Lamperti-Kiu representations \cite{Lam,CPR, KKPW, ACGZ} for positive or real-valued self-similar Markov processes respectively. In the latter case, the type $J$ is the sign, see \cite{CPR, KKPW}. We call this process $(X,J)$, or sometimes just $X$, a \emph{self-similar Markov process with types}. 

For our purposes, it is convenient to  study   self-similar Markov processes (ssMp) with types  when the starting point of $X$ tends to the origin. In other words,  we would like to  find the behaviour of  $\mathbbm{P}_{x,i}$ as $x$ tends to $0$. On the other hand, we see from the scaling property that the question boils down to that of the large times asymptotic behaviour. That is 
$\mathbb{P}_{x,i}(X_1\in \cdot , J_1\in \cdot)$ converges weakly, as $x\to 0$ to a probability measure here denoted as $\eta(\cdot, \cdot)$ if, and only if, for any $y>0$
\[
\mathbb{P}_{y,i}(t^{-1/\alpha}X_t\in \cdot , J_t\in \cdot)\to \eta(\cdot, \cdot)\qquad \textrm{as}\quad t\to \infty.
\] 
 Let us assume that $X$ is {\it conservative} that is, that its lifetime $\zeta$ is almost surely infinite. Our next result determines the existence of the weak limit of $\mathbb{P}_{x,i}(X_t\in \cdot , J_t\in \cdot)$, as $x\to 0$, under the assumption that $\xi$ is not concentrated on a lattice and $m:=\mathtt{E}_{0,\pi}[\xi_1]<\infty$, as a \textit{self-similar entrance law} for  the ssMp with types $(X,J)$.
 
   Following  Rivero \cite{Rivero16}, we say that a family  $(\eta_t,t>0)$ of $\sigma$--finite measures on $(0,\infty)\times \mathcal{I}$ is a self-similar entrance law for the semigroup $(P_t^{(X,J)}, t\ge 0)$ of $(X,J)$ if the following two properties are satisfied:
 \begin{itemize}
 \item[i)]  the identity between measures
 \[
 \eta_s P_t^{(X,J)}=\eta_{t+s},
 \]
 holds, that is 
 \[
\sum_{i\in \mathcal{I}} \int_{(0,\infty)} \eta_s({\rm d} x, i) \mathbb{E}_{x,i}\left[f(X_t, J_t)\right]=\sum_{i\in \mathcal{I}} \int_{(0,\infty)} \eta_{t+s}({\rm d} x, i) f(x,i),
 \]
where, for every $i \in \mathcal{I}$,  $f(\cdot, i)$ is a positive  and measurable function and $s>0$, $t\ge 0.$
\item[ii)]  for all $s>0$,
\[
\eta_s f =\eta_1H_{s^{1/\alpha}}f,
\]
where, for every $i \in \mathcal{I}$,  $f(\cdot, i)$ is a positive  and measurable function, and for $c>0$, $H_c$ denotes the dilation operator $H_cf(x,i)=f(cx, i)$.
 \end{itemize}

\begin{Thm}\label{entrancelaw} Assume that $X$ is a conservative ssMp with types and index $\alpha>0$. Moreover, suppose that the MAP $(\xi,\Theta)$ associated with $X$ through the Lamperti transform is such that $\xi$ is not concentrated on a lattice and has positive mean \begin{equation}\label{meanalh}
m=\mathtt{E}_{0,\pi}[\xi_1]<\infty.
\end{equation}
Then as $x\to 0+$, the probability measures $\mathbb{P}_{x,j}$ converge in the sense of finite dimensional marginals towards a probability measure which is determined by the  self-similar entrance law  $(\eta_t, t\ge 0)$, where, for $j\in \mathcal{I}$ such that  $f(\cdot, j)$ is positive and measurable, and $t>0$, 
\[
\eta_t f:=\frac{1}{\alpha m}\sum_{i\in \mathcal{I}}\pi_i\mathtt{E}^\natural_{0,i}\left[\frac{1}{I(\alpha\xi)} f\left(\left(\frac{t}{I(\alpha\xi)}\right)^{1/\alpha}, i\right)\right],
\]
with
\[
I(\alpha\xi):= \int_0^{\infty} \mathrm{e}^{\alpha \xi(s)}\mathrm{d}s.
\]
More precisely, for all measurable sets $A_1, B_1, \ldots, A_n, B_n$ and all times $0<t_1<\ldots< t_n$, as $x\to 0+$,
\begin{multline*}
\mathbb{P}_{x,j}(X_{t_1}\in A_1, J_{t_1}\in B_1,X_{t_2}\in A_2, J_{t_2}\in B_2, \ldots, X_{t_n}\in A_n, J_{t_n}\in B_n)= \\
\to \sum_{i\in B_1}\int_{A_1} \eta_{t_1}( dy, i)\mathbb{P}_{y,i}(X_{t_2}\in A_2, J_{t_2}\in B_2, \ldots, X_{t_n}\in A_n, J_{t_n}\in B_n).
\end{multline*}
\end{Thm}
\noindent It is important to note that the self-similar entrance law  $(\eta_t, t\ge 0)$ is actually a family of probability measures.

The proof of the previous result follows the same steps as  the construction of the law  of  positive  self-similar Markov processes issued from the origin given by  Bertoin and Yor \cite{BY2002} plus  Markov renewal theory.   We will explain in the Appendix how we can use all these ideas to deduce our result. Bertoin and Yor \cite{BY2002} 

It is important to note that the above representation
of the entrance law has not been developed in the literature even for real self-similar Markov processes, with the
exception of the special setting of a stable process conditioned to avoid the
origin (cf. Kyprianou et al. \cite{KKPW}). Nonetheless a similar expression is stated in Theorem 11.16 in \cite{KP} for general real self-similar Markov processes without proof.

%
%
\section{Multitype growth-fragmentation processes} \label{sec: multitype GF}
In this section, we present an extension of the growth-fragmentation framework \cite{Ber-GF} to particles with finitely  many types in $\Ical$. We point out that the approach here presented is completely different to the treatment in \cite{DS} of the signed case, which relies on a change of Eve cell to go back to the positive setting. We shall describe the martingales appearing in this context, and how they can be found in the roots of \emph{multitype cumulants}.

\subsection{Construction of the multitype growth-fragmentation cell system} 
\label{sec: construction GF multitype}
Following section \ref{sec: ssmp}, we will consider a \textit{càdlàg} self-similar Markov process with types $(X,J)$. For technical reasons, we further assume that $(X,J)$ is either absorbed at the cemetery state $\partial$ after a finite time $\zeta$, or that $X$ converges to $0$ under $\Pb_{x,i}$ for all $x\in\R_+^*, i\in\Ical$. We write $\Delta X(s) := (X(s)-X(s^-))\mathds{1}_{\{X(s)<X(s^-)\}}$ for the jump of $X$, when negative.

We now construct a cell system whose building block is the self-similar Markov process with types $(X,J)$. This cell system will start from a single particle whose size and type are given by the process $(X,J)$, that will split in a binary way whenever $X$ has a negative jump. This will create new particles with initial size given by the intensity of the jump, and which will then evolve as $(X,J)$ independently of the mother cell, and independently of one another, conditionally on their sizes at birth. This construction takes the viewpoint presented in \cite{Ber-GF}, but note that in our context we need to clarify the types of the offspring. To this end, we introduce some preliminary notation. Call $(\xi,\Theta)$ the MAP underlying $(X,J)$ \textit{via} the Lamperti-Kiu transform \eqref{eq: Lamperti MAP}. We assume throughout the paper that for all $i\in \Ical$, the Lévy measure $\Lambda_i$ of $\xi_i$ can be decomposed as the sum of Lévy measures
\begin{equation} \label{eq: decomposition Lévy}
\Lambda_{i}(\mathrm{d}x) := \sum_{k\in\Ical}\Lambda_{i}^{(k)}(\mathrm{d}x),
\end{equation}
satisfying the following integrability condition
\[
\int_{\mathbb{R}}(1\land |x|^2)\Lambda_{i}(\mathrm{d}x)<\infty.
\]
Likewise, for $i,j\in\Ical$, we give ourselves a decomposition of the law $\Lambda_{U_{i,j}}$ of $U_{i,j}$ as 
\begin{equation} \label{eq: decomposition Ui,j}
\Lambda_{U_{i,j}}(\mathrm{d}x) := \sum_{k\in\Ical}\Lambda_{U_{i,j}}^{(k)}(\mathrm{d}x).
\end{equation}
Equations \eqref{eq: decomposition Lévy} and \eqref{eq: decomposition Ui,j} can be interpreted as a thinning of $\xi_i$ and $U_{i,j}$ respectively: the jumps of $\xi_i$ and $U_{i,j}$ should be understood as the result of a superimposition of jumps coming with a type $k\in\Ical$. Through the Lamperti time change \eqref{eq: Lamperti time change}, we see that any jump $\Delta X(s)$ of $X$ now also comes with some type, that we denote $J_{\Delta}(s)$. 

We may now construct the cell system associated with $(X,J)$ and indexed by the tree $\Ub:=\bigcup_{i\ge 0} \N^i$, with $\N=\{1,2,\ldots\}$ and $\N^0:=\{\varnothing\}$ is the label of the \emph{Eve cell}. For $u:=(u_1,\ldots,u_i)\in \Ub$, we denote by $|u|=i$, the \emph{generation} of $u$. In this tree, the offspring of $u$ will be labelled by the lists $(u_1,\ldots,u_i,k)$, with $k\in \N$.

We then define the law $\Pcal_{x,i}$, $x>0$, $i\in\Ical$, of the cell system $((\Xcal_u(t),\Jcal_u(t)),u\in \Ub)$ driven by $X$ similarly to \cite{Ber-GF}. Let $b_{\varnothing}=0$ and take a copy $(\Xcal_{\varnothing}, \Jcal_{\varnothing})$ of $(X,J)$ started from $(x,i)$. We can rank the sequence of positive jumps and times $(x_1,\beta_1),(x_2,\beta_2),\ldots$ of $-\Xcal_{\varnothing}$ by descending lexicographical order of the $x_k$'s (this is possible because in any case X  is either absorbed at the cemetery state $\partial$  or converges to $0$). We write $j_1,j_2,\ldots$ for the corresponding types. Given this sequence $(x_k,j_k,\beta_k, k\in\N)$, we define the first generation $(\Xcal_k, \Jcal_k), k\in \N,$ of our cell system as independent processes with respective law $\Pb_{x_k, j_k}$, and we set $b_k = b_{\varnothing}+\beta_k$ for the \emph{birth time} of $k$ and $\zeta_k$ for its lifetime. Likewise, we define the $n$-th generation given generations $1,\ldots, n-1$. A cell $u'=(u_1,\ldots,u_{n-1})\in \N^{n-1}$ gives birth to the cell $u=(u_1,\ldots,u_{n-1},k)$, with lifetime $\zeta_u$, at time $b_u = b_{u'}+\beta_{k}$ where $\beta_{k}$ is the $k$-th jump of $\Xcal_{u'}$ (in terms of the same ranking as before). Moreover, conditionally on the jump sizes, types and times of $\Xcal_{u'}$, $(\Xcal_u,\Jcal_u)$ has law $\Pb_{y,j}$ and is independent of the other daughter cells at generation $n$, where $-y=\Delta \Xcal_{u'}(\beta_k)$ comes with type $j$. Note that division events are \emph{conservative} in the sense that the sum of the size of a particle born at time $t$ and of its mother cell at time $t$ exactly equals the size of the mother cell before dislocation.

Although by construction the cells are not labelled chronologically, this uniquely defines the law $\Pcal_{x,i}$ of the cell system driven by $(X,J)$ and started from $(x,i)$. The cell system $((\Xcal_u(t), \Jcal_u(t)),u\in \Ub)$ is meant to describe the evolution of a population of atoms $u$ with size $\Xcal_u(t)$ and type $\Jcal_u(t)$ evolving with its age $t$ and fragmenting in a binary way. 

Finally, we define the \emph{multitype growth-fragmentation process}
\[
\Xbf(t) := \left\{\left\{ (\Xcal_u(t-b_u),\Jcal_u(t-b_u)), \; u\in\Ub \; \text{and} \; b_u\le t<b_u+\zeta_u \right\}\right\}, \quad t\ge 0,
\]
where the double brackets denote multisets: $\Xbf(t)$ is the collection of all the particles alive in the system at time $t$. We denote by $\mathbf{P}_{x,i}$ the law of $\Xbf$ started from $(x,i)$ and $(\Fcal_t,t\ge 0)$ the natural filtration of $\Xbf$. 

\begin{Rk}
We emphasise that only the \emph{negative} jumps of $X$ give birth to new cells. One could also be willing to create new particles at the positive jump times, corresponding to cells with negative mass, so that the conservation rule still holds at splittings, similarly as in \cite{DS}. This will simply result in doubling the number of types of the chain $J$, by considering the sign itself as a type. Hence we can restrict without loss of generality to considering only positive cells, \textit{i.e.} negative jumps.
\end{Rk}

We stress that, similarly to \cite[Section 3.4]{Ber-GF}, multitype growth-fragmentation processes are closely related to \emph{multitype branching random walks}, as the process $(-\log\Xcal_u(0), \Jcal_u(0))_{u\in\Ub}$ forms a multitype branching random walk. It is clear from the definition of growth-fragmentation processes that the cell system enjoys a genealogical branching structure. Under mild assumptions, this extends to a temporal branching property. Construct
\[ 
\overline{\Xbf} (t) :=  \left\{\left\{ (\Xcal_u(t-b_u), \Jcal_u(t-b_u),|u|), \; u\in\Ub \; \text{and} \; b_u\le t<b_u+\zeta_u \right\}\right\}, \quad t\ge 0,
\]
by adjunction of the generations to the growth-fragmentation process; and consider its associated filtration $(\overline{\Fcal}_t, t\ge 0)$. A measurable function $f:\R_+ \rightarrow [0,+\infty)$ is called \emph{excessive} for $\Xbf$ if  $\underset{x>a}{\inf} f(x) > 0$ for all $a>0$, and for all $(x,i)\in \R_+\times\Ical$ and all $t\ge 0$,
\begin{equation}\label{excesive}
\mathbf{E}_{x,i}\Bigg[ \sum_{(z,j)\in\Xbf(t)} f(z)\Bigg] \le f(x).
\end{equation}
If such an excessive function exists, then one can rank the elements $(X_1(t),J_1(t)), (X_2(t),J_2(t)), \ldots$ of $\Xbf(t)$ by descending order of their size for any fixed $t$.
\begin{Prop} \label{prop: branching temporal}
Assume that $\Xbf$ has an excessive function. Then for any $t\ge 0$, conditionally on $\overline{\Xbf}(t)=\{\{(x_i, j_i, n_i)\}\}$, the process $(\overline{\Xbf}(t+s), s\ge 0)$ is independent of $\overline{\Fcal}_t$ and distributed as 
\[
\bigsqcup_{i\ge 1} \overline{\Xbf}_i(s) \circ \theta_{n_i},
\]
where the $\overline{\Xbf}_i, i\ge 1,$ are independent processes distributed as $\overline{\Xbf}$ under $\mathbf{P}_{x_i, j_i}$, $\theta_n$ is the shift operator, i.e. $\{\{(z_i, y_i, k_i), i\ge 1\}\}\circ \theta_n := \{\{(z_i, y_i, k_i+n), i\ge 1\}\}$, and $\sqcup$ denotes union of multisets.
\end{Prop}
\begin{proof}
We refer to \cite[Proposition 2]{Ber-GF} for a proof of the statement in the classical framework, which is then easily extended to the multitype case.
\end{proof}

\subsection{Martingales in multitype growth-fragmentation processes} \label{sec: martingales}
We continue the study of martingales for multitype growth-fragmentation processes initiated in \cite{DS} in the signed case. The fact that $(-\log\Xcal_u(0), \Jcal_u(0))_{u\in\Ub}$ forms a multitype branching random walk provides several tools, including genealogical martingales for the growth-fragmentation cell system. The key feature is the following matrix $m(q)$ indexed by the type set $\Ical$, with entries
\begin{equation} \label{eq: matrix m(q)}
m_{i,j}(q) := \Eb_i \left[\sum_{0<s<\zeta} |\Delta X(s)|^q \mathds{1}_{\{J_{\Delta}(s)=j\}} \right], \quad q\in\R.
\end{equation}
This matrix has only nonnegative entries. We make the following two assumptions throughout the paper.

\medskip
   \textbf{Assumption A} : \textit{For $q\in\R$ such that $m(q)$ has finite entries, the matrix $m(q)$ is irreducible.}
\medskip

\noindent In other words, Assumption A means that all the types communicate in the growth-fragmentation cell system (this is not too restrictive, since we could restrict to communication classes otherwise). Since $\Ical$ is finite, this enables us to consider the Perron-Frobenius eigenvalue $\mathrm{e}^{\lambda(q)}$ and an associated positive eigenvector.  

\medskip
   \textbf{Assumption B} : \textit{There exists $\omega\in \R$ such that $\lambda(\omega)=0$.}
\medskip

\noindent Note that, if $\mu$ denotes the measure
\begin{equation} \label{eq: log Laplace convexity}
\mu(\mathrm{d}z) := \Eb_i \left[\sum_{0<s<\zeta} \mathds{1}_{\{|\Delta X(s)|\in\mathrm{d}z\}} \mathds{1}_{\{J_{\Delta}(s)=j\}} \right],
\end{equation}
then for all $q\in\R$, $m_{i,j}(q)$ is the $\log$-Laplace transform of $\mu$ with exponent $q$, and hence is convex. It is a consequence of Kingman's theorem \cite{Kin} that if $m$ is finite and invertible in some neighbourhood of $\omega$, then $\lambda$ is convex on this domain. This implies that $\lambda$ can only have at most two roots.

We shall give a criterion for Assumption B later on in section \ref{sec: cumulants}. If $(v_i)_{i\in \Ical}$ has positive entries and $\omega\ge 0$, we say that $((v_i)_{i\in \Ical}, \omega)$ is \emph{admissible} for $X$ if $\omega\in D$, $\lambda(\omega)=0$ and $(v_i)_{i\in \Ical}$ is an associated eigenvector of $m(\omega)$. In other words, $((v_i)_{i\in \Ical}, \omega)$ is admissible for $X$ if and only if
\[
\forall i\in\Ical, \quad \Eb_i\Bigg[\sum_{0<s<\zeta} v_{J_{\Delta}(s)} |\Delta X(s)|^{\omega} \Bigg] = v_i.
\]
Indeed, by Perron-Frobenius theory, only the leading eigenvalue can be associated to positive eigenvectors. This invariance property extends to a genealogical martingale as follows. Define 
\[
\Gscr_n := \sigma((\Xcal_u,\Jcal_u); |u|\le n),
\]
noting that by definition, if $u\in\Ub$ is such that $|u|=n+1$, then $\Xcal_u(0)$ is $\Gscr_n$--measurable. 
\begin{Prop} \label{prop: genealogical martingale}
For all $(x,i)\in\R^*_+\times\Ical$, the process
\[
\Mcal(n) := \sum_{|u|=n+1} v_{\Jcal_u(0)} \Xcal_u(0)^{\omega}, \quad n\ge 0,
\]
defines a $(\Gscr_n,n\ge 0)$--martingale under $\Pcal_{x,i}$.
\end{Prop}
\begin{proof}
The process $\Mcal$ is obtained as the genealogical martingale of the multitype branching random walk $(-\log\Xcal_u(0), \Jcal_u(0))_{u\in\Ub}$, see \cite[Theorem 3.3]{DS}.
\end{proof}

\medskip
Moreover, the following martingale for $X$ will turn out useful in the next section. In particular, it implies the existence of an excessive function by extending \cite[Theorem 1]{Ber-GF} to the multitype case.

\begin{Prop} \label{prop: martingale M}
For all $(x,i)\in\R^*_+\times\Ical$, under $\Pb_{x,i}$ the process
\[
M(t) := v_{J(t)} |X(t)|^{\omega} + \sum_{0<s\le t\wedge \zeta} v_{J_{\Delta}(s)} |\Delta X(s)|^{\omega}, \quad t\ge 0,
\]
is a uniformly integrable martingale for the filtration $(F_t^X,t\ge 0)$ of $X$, with terminal value 
\[\sum_{0<s<\zeta} v_{J_{\Delta}(s)} |\Delta X(s)|^{\omega}.\]
\end{Prop}
\begin{proof}
We omit the proof as it essentially follows from \cite[Proposition 3.5]{DS}.
\end{proof}

\subsection{Multitype cumulant functions} \label{sec: cumulants}
For any sequence $((v_i)_{i\in\Ical}, \omega)$, define
\[
M(t):= v_{J(t)} |X(t)|^{\omega} + \sum_{0<s\le t\wedge \zeta} v_{J_{\Delta}(s)} |\Delta X(s)|^{\omega}, \quad t\ge 0,
\]
where we omit the dependence on $\omega$ and $(v_i)_{i\in\Ical}$ in the notation of $M$ for simplicity. \cref{prop: martingale M} states that when the underlying sequence is admissible, $M$ is a martingale under $\Pb_i$ for all $i\in\Ical$ (see the signed case \cite{DS}). A converse statement also holds, providing a more tractable characterisation of admissibility.
\begin{Prop} \label{prop: characterisation admissibility}
Let $H$ be the first jump time of $J$. Then $((v_i)_{i\in \Ical}, \omega)$ is admissible for $X$ if and only if, for all $i\in\Ical$,
\[
\Eb_i[M(H)] = v_i.
\]
\end{Prop}
\begin{proof}
The implication $(\Rightarrow)$ follows easily from the optional stopping theorem applied to the martingale $M$ in \cref{prop: martingale M}. Conversely, if we denote $(H_k, k\ge 0)$, the successive jump times of $J$ (with $H_0=0$), then for any $i\in\Ical$, by the Markov property of $(X,J)$ and  self-similarity of $X$,
\begin{align*}
    \Eb_i\Bigg[\sum_{0<s<\zeta} v_{J_{\Delta}(s)} |\Delta X(s)|^{\omega} \Bigg] 
    &= \sum_{k\ge 0} \Eb_i\Bigg[\sum_{H_k< s\le H_{k+1}} v_{J_{\Delta}(s)} |\Delta X(s)|^{\omega} \Bigg] \\
    &= \sum_{k\ge 0} \Eb_i\left[|X(H_k)|^{\omega}\,\Eb_{J(H_k)}\left[\sum_{s\le H} v_{J_{\Delta}(s)} |\Delta X(s)|^{\omega} \right]\right].
\end{align*}
Because we have assumed $\Eb_j[M(H)] = v_j$ for all $j\in \Ical$, this is
\[
\Eb_i\Bigg[\sum_{0<s<\zeta} v_{J_{\Delta}(s)} |\Delta X(s)|^{\omega} \Bigg] 
=
\sum_{k\ge 0} \Eb_i\bigg[|X(H_k)|^{\omega}\Big(v_{J(H_k)}-\Eb_{J(H_k)}\Big[v_{J(H)} |X(H)|^{\omega}\Big] \Big)\bigg]. 
\]
Hence, using again the Markov property and self-similarity of $X$ backwards, we find ourselves with 
\[
\Eb_i\left[\sum_{0<s<\zeta} v_{J_{\Delta}(s)} |\Delta X(s)|^{\omega} \right] 
=
\sum_{k\ge 0} \Eb_i\left[v_{J(H_k)}|X(H_k)|^{\omega}\right] - \sum_{k\ge 0} \Eb_i\left[v_{J(H_{k+1})}|X(H_{k+1})|^{\omega}\right],
\]
which ultimately cancels out, leaving $\Eb_i\left[\sum_{0<s<\zeta} v_{J_{\Delta}(s)} |\Delta X(s)|^{\omega} \right] = v_i$.
\end{proof}
\medskip

Next, we identify multitype cumulant functions $\Kcal_i$, $i\in\Ical$, whose common roots correspond to the admissible exponents $\omega$. To do so, we compute $\Eb_i[M(H)]$ in terms of the underlying MAP characteristics, for any (not necessarily admissible) sequence $((v_i)_{i\in\Ical},\omega)$. The expectation can be written as $\Eb_i[M(H)]= A + B$, where 
\[
A := \Eb_i\left[\sum_{0<s\le H\wedge \zeta} v_{J_{\Delta}(s)} |\Delta X(s)|^{\omega}\right]
\quad \text{and} \quad 
B := \Eb_i\left[ v_{J(H)} |X(H)|^{\omega}\right].
\]
Let us start with the term $A$. For $s>0$, we write as in \eqref{eq: Lamperti MAP}, $X(\varphi^{-1}(s)) = \mathrm{e}^{ \xi(s)}$ and $J(\varphi^{-1}(s)) = \Theta(s)$ under $\Pb_i$, where $\varphi$ is the usual time-change \eqref{eq: Lamperti time change}. From this standpoint, 
\begin{equation}
A
=
\mathtt{E}_i\Bigg[ \sum_{0<s< \rho_i} v_{\iota_{\Delta}(s)} \mathrm{e}^{\omega \xi_i(s^-)} \Big(1- \mathrm{e}^{\Delta \xi_i(s)}\Big)^{\omega} \Bigg] + \mathtt{E}_i\bigg[v_{\iota_{\Delta}(\rho_i)} \mathrm{e}^{\omega \xi_i(\rho_i^-)} \Big(1- \mathrm{e}^{U_{i,\Theta(\rho_i)}}\Big)^{\omega} \bigg], \label{eq: cumulant A}
\end{equation}
where $\iota_{\Delta}(s)$ stands for the type corresponding to the jump of $\xi$ at time $s$. By independence and the compensation formula for $\xi_i$, the first term of \eqref{eq: cumulant A} is 
\[
\begin{split}
\mathtt{E}_i\Bigg[ \sum_{s< \rho_i} v_{\iota_{\Delta}(s)} &\mathrm{e}^{\omega \xi_i(s^-)} \Big(1- \mathrm{e}^{\Delta \xi_i(s)}\Big)^{\omega} \Bigg] \\
&= 
\int_0^{\infty} \mathrm{d}t \, (-q_{i,i}) \mathrm{e}^{q_{i,i}t} \sum_{k\in\Ical} v_k \mathtt{E}_i\left[\int_{0}^{t} \mathrm{d}s \,\mathrm{e}^{\omega \xi_i(s)} \right] \int_{(-\infty,0)} \Lambda_{i}^{(k)}(\mathrm{d}x)(1-\mathrm{e}^{x})^{\omega} \\
&= \sum_{k\in\Ical} v_k \int_{(-\infty,0)}  \Lambda_{i}^{(k)}(\mathrm{d}x)(1-\mathrm{e}^{x})^{\omega} \cdot \frac{1}{\psi_i(\omega)}\int_0^{\infty} \mathrm{d}t \, (-q_{i,i}) \mathrm{e}^{q_{i,i}t}(\mathrm{e}^{\psi_i(\omega)t}-1) \\
&= -\frac{1}{\psi_i(\omega)+q_{i,i}} \cdot \sum_{k\in\Ical} v_k \int_{(-\infty,0)}  \Lambda_{i}^{(k)}(\mathrm{d}x)(1-\mathrm{e}^{x})^{\omega},
\end{split}
\]
provided 
$\psi_i(\omega)+q_{i,i}<0$ (otherwise the expectation blows up). Now, let $\iota^{\star} = \Theta(\rho_i)$ be the type to which the Markov chain jumps at time $\rho_i$. Then $\iota^{\star}$ is independent of $\rho_i$, and for all $j\in\Ical\setminus\{i\}$, $\iota^{\star} = j$ with probability $-\frac{q_{i,j}}{q_{i,i}}$. By conditioning on $\rho_i$, we obtain that the second term of \eqref{eq: cumulant A} is
\[
\begin{split}
\mathtt{E}_i\bigg[v_{\iota_{\Delta}(\rho_i)} &\mathrm{e}^{\omega \xi_i(\rho_i^-)} \Big(1- \mathrm{e}^{U_{i,\iota^{\star}}}\Big)^{\omega} \bigg] \\
&= 
\int_0^{\infty} \mathrm{d}t \, (-q_{i,i})\mathrm{e}^{q_{i,i}t} \sum_{j\in\Ical\setminus\{i\}} \frac{q_{i,j}}{(-q_{i,i})} \mathtt{E}_i\Big[\mathrm{e}^{\omega \xi_i(t)}\Big] \sum_{k\in\Ical} v_k \int_{(-\infty,0)}  \Lambda^{(k)}_{U_{i,j}}(\mathrm{d}x) (1- \mathrm{e}^{x})^{\omega} \\
&= \int_0^{\infty} \mathrm{d}t \, \mathrm{e}^{(\psi_i(\omega)+q_{i,i})t} \sum_{j\in\Ical\setminus\{i\}} q_{i,j} \sum_{k\in\Ical} v_k \int_{(-\infty,0)}  \Lambda^{(k)}_{U_{i,j}}(\mathrm{d}x) (1- \mathrm{e}^{x})^{\omega} \\
&=
-\frac{1}{\psi_i(\omega)+q_{i,i}} \sum_{k\in\Ical} v_k\sum_{j\in\Ical\setminus\{i\}} q_{i,j}  \int_{(-\infty,0)}  \Lambda^{(k)}_{U_{i,j}}(\mathrm{d}x) (1- \mathrm{e}^{x})^{\omega}, 
\end{split}
\]
provided again that  
$\psi_i(\omega)+q_{i,i}<0$. Therefore, we end up with
\[
A
=
-\frac{1}{\psi_i(\omega)+q_{i,i}} \sum_{k\in\Ical} v_k \int_{(-\infty,0)}  \Pi_{i,k}(\mathrm{d}x) (1- \mathrm{e}^{x})^{\omega},
\]
with $\Pi_{i,k}(\mathrm{d}x) := \Lambda^{(k)}_i(\mathrm{d}x) + \sum_{j\in\Ical\setminus\{i\}} q_{i,j} \Lambda^{(k)}_{U_{i,j}}(\mathrm{d}x)$.

We now compute 
\[
B 
=
\mathtt{E}_i\left[ v_{\iota^{\star}} \mathrm{e}^{\omega (\xi_i(\rho_i)+U_{i,\iota^{\star}})}\right].
\] 
As before, we condition on $\rho_i$ and decompose over the possible values $j\in\Ical\setminus\{i\}$ for $\iota^{\star}$:
\begin{align*}
B
&= \sum_{j\in\Ical\setminus\{i\}} \int_0^{\infty}\mathrm{d}s\, (-q_{i,i})\mathrm{e}^{q_{i,i}s} \frac{q_{i,j}}{(-q_{i,i})} v_j \,\mathtt{E}_i\left[ \mathrm{e}^{\omega (\xi_i(s)+U_{i,j})}\right] \\
&= \sum_{j\in\Ical\setminus\{i\}} q_{i,j}v_j\int_0^{\infty}\mathrm{d}s \, \mathrm{e}^{q_{i,i}s} \mathrm{e}^{\psi_i(\omega)s}G_{i,j}(\omega) \\
&= -\frac{1}{\psi_i(\omega)+q_{i,i}} \cdot \sum_{j\in\Ical\setminus\{i\}} q_{i,j}v_j G_{i,j}(\omega),
\end{align*}
as long as $\psi_i(\omega)+q_{i,i}<0$. We come to the conclusion that
\[
\Eb_i[M(H)] 
=
-\frac{1}{\psi_i(\omega)+q_{i,i}} \cdot \left(\sum_{k\in\Ical} v_k \int_{(-\infty,0)}  \Pi_{i,k}(\mathrm{d}x)(1-\mathrm{e}^{x})^{\omega}
+ \sum_{j\in\Ical\setminus\{i\}} q_{i,j}v_j G_{i,j}(\omega) \right).
\]
This is equal to $v_i$ if and only if, 
\[
\Kcal_i(\omega):=(\psi_i(\omega)+q_{i,i}) + \sum_{k\in\Ical} \frac{v_k}{v_i} \int_{(-\infty,0)}  \Pi_{i,k}(\mathrm{d}x)(1-\mathrm{e}^{x})^{\omega}
+ \sum_{j\in\Ical\setminus\{i\}} \frac{v_j}{v_i} q_{i,j} G_{i,j}(\omega)
=
0,
\]
and, thanks to \cref{prop: characterisation admissibility}, Assumption A in \cref{sec: martingales} boils down to the existence of $\omega\in\R$ and a sequence $(v_i)_{i\in\Ical}$ of positive numbers such that, for all $i\in\Ical$, $\Kcal_i(\omega) = 0$. We will call the family $(\Kcal_i, i\in \Ical)$ the \emph{multitype cumulant functions}. We also write 
\begin{equation} \label{eq: kappa}
\kappa_i(q) := (\psi_i(q)+q_{i,i}) + \int_{(-\infty,0)} \Pi_{i,i}(\mathrm{d}x)(1-\mathrm{e}^{x})^{q}, \quad q\ge 0,
\end{equation} 
for the \emph{cumulant} function corresponding to type $i$, so that for $q\ge 0$,
\begin{equation} \label{eq: multitype cumulants}
\Kcal_i(q):=\kappa_i(q) + \sum_{j\in\Ical\setminus\{i\}} \frac{v_j}{v_i} \left(\int_{(-\infty,0)}  \Pi_{i,j}(\mathrm{d}x)(1-\mathrm{e}^{x})^{q}
+ q_{i,j} G_{i,j}(q)\right).
\end{equation}

%
%
\section{The spine decomposition of multitype growth-fragmentation processes}
\label{sec: spine}

\subsection{Description of the spine under the change of measure} \label{sec:change measure}

\textbf{A change of measure.} The martingale $(\Mcal(n), n\ge 0)$ in \cref{prop: genealogical martingale} enables us to introduce a new probability measure $\Phat_{x,i}$ for $x>0, i\in\Ical$. Under this change of measure, the cell system has a spine decomposition that we aim to describe (see \cite[Section 4.1]{BBCK}). The measure $\Phat_{x,i}$ singles out a particular \emph{leaf} $\Lcal \in \partial \Ub = \N^{\N}$. On $\mathscr{G}_n$, for $n\ge 0$, it has Radon-Nikodym derivative $\Mcal(n)$ with respect to $\Pcal_{x,i}$, up to normalisation, \emph{viz.} for all $G_n \in \mathscr{G}_n$,
\[
\Phat_{x,i}(G_n) := \frac{1}{v_{i}x^{\omega}} \Ecal_{x,i}\left[\Mcal(n) \mathds{1}_{G_n} \right].
\]
Moreover, conditionally on $\Gscr_n$, the parent of the particular leaf $\Lcal$ at generation $n+1$ is chosen under $\Phat_{x,i}$ proportionally to its weight in the martingale $\Mcal(n)$. More precisely, let $\ell(n)$ denote the ancestor of a leaf $\ell\in\partial \Ub$ at generation $n$. Then for all $n\ge 0$ and all $u\in \Ub$ such that $|u|=n+1$,
\begin{equation} \label{eq: generation leaf}
\Phat_{x,i} \left( \Lcal(n+1)=u \,\big|\, \mathscr{G}_n\right) := \frac{v_{\Jcal_u(0)}\Xcal_u(0)^{\omega}}{\Mcal(n)}.
\end{equation}
The consistency of formula \eqref{eq: generation leaf} stems from the martingale property of $(\Mcal(n),n\ge 0)$ and the branching structure of the system, thus defining a unique probability measure by an application of the Kolmogorov extension theorem.

One key player is provided by the so called \emph{tagged cell} or \emph{spine}, which consists in following the evolution of the cell associated with the leaf $\Lcal$. The tagged cell will have the role of a backbone in the \emph{spine decomposition} of the cell system under $\Phat_{x,i}$. Let $b_{\ell} = \lim\uparrow b_{\ell(n)}$ for any leaf $\ell\in\partial \Ub$. Then, we define $\Xhat$ by $(\Xhat(t),\Jhat(t)):=\partial$ if $t\ge b_{\Lcal}$ and 
\[
\Xhat(t):= \Xcal_{\Lcal(n_t)}(t-b_{\Lcal(n_t)}) \quad \text{ and } \quad \Jhat(t):= \Jcal_{\Lcal(n_t)}(t-b_{\Lcal(n_t)}) \qquad \textrm{if} \qquad t<b_{\Lcal},
\]
where $n_t$ is the unique integer $n$ such that $b_{\Lcal(n)}\le t < b_{\Lcal(n+1)}$. From the very definition of $\Phat_{x,i}$, for all nonnegative measurable function $f$ and all $\Gscr_n$--measurable nonnegative random variable $B_n$, 
\[
v_{i}x^{\omega}\Ehat_{x,i}\Big[f(\Xcal_{\Lcal(n+1)}(0),\Jcal_{\Lcal(n+1)}(0))B_n\Big]
=
\Ecal_{x,i}\left[ \sum_{|u|=n+1} v_{\Jcal_u(0)}\Xcal_u(0)^{\omega}f(\Xcal_u(0), \Jcal_u(0))B_n\right].
\]
This extends to a temporal identity in the following way. Recall that  $\Xbf(t)=\left\{\left\{(X_k(t),J_k(t)), k\ge 1\right\}\right\},$ for  $t\ge 0$, have been enumerated by descending order of the $X_k(t)$.
\begin{Prop} \label{prop:spine temporal multitype}
For every $t\ge 0$, every nonnegative measurable function $f$ vanishing at $\partial$, and every $\overline{\Fcal}_t$--measurable nonnegative random variable $B_t$, we have
\[
v_i x^{\omega}\Ehat_{x,i}\Big[f(\Xhat(t),\Jhat(t))B_t\Big]
=
\Ecal_{x,i}\left[ \sum_{k\ge 1} v_{J_k(t)} X_k(t)^{\omega}f(X_k(t),J_k(t))B_t\right].
\]
\end{Prop}
\begin{proof}
The proof essentially follows from the arguments presented in the proof of \cite[Proposition 4.1]{BBCK}. We provide its proof for the sake of completeness.

Let $t\ge 0$. Consider the case when $B_t$ is $\overline{\Fcal}_t \cap \Gscr_k$--measurable for some $k\in \N$ (the result would then be readily extended by a monotone class argument). Since $f(\partial)=0$, almost surely,
\[
f(\Xhat(t),\Jhat(t)) B_t\mathds{1}_{\{b_{\Lcal(n+1)}>t\}} \underset{n\rightarrow \infty}{\longrightarrow} f(\Xhat(t),\Jhat(t)) B_t. 
\]
Therefore, by monotone convergence, 
\[
\Ehat_{x,i}\Big[f(\Xhat(t),\Jhat(t)) B_t\Big] = \underset{n\rightarrow \infty}{\lim} \Ehat_{x,i} \Big[f(\Xhat(t),\Jhat(t)) B_t\mathds{1}_{\{b_{\Lcal(n+1)}>t\}}\Big]. 
\]
Now, we want to condition on $\Gscr_n$ and decompose $\Lcal(n+1)$ over the cells at generation $n+1$, provided $n>k$ so that $B_t$ is $\Gscr_n$--measurable. For $u$ such that $b_u>t$, write $u(t)$ for the most recent ancestor of $u$ at time $t$. Then
\begin{multline*}
\Ehat_{x,i} \left[f(\Xhat(t),\Jhat(t)) B_t\mathds{1}_{\{b_{\Lcal(n+1)}>t\}}\right] \\
=
\frac{1}{v_{i}(\omega)x^{\omega}} \Ecal_{x,i} \left[\sum_{|u|=n+1} v_{\Jcal_u(0)} \Xcal_u(0)^{\omega} \mathds{1}_{\{b_u>t\}} f(\Xcal_{u(t)}(t-b_{u(t)}),\Jcal_{u(t)}(t-b_{u(t)}))B_t\right].
\end{multline*}
Splitting over the value of $u(t)$ yields 
\begin{multline} \label{eq:decompose u(t)}
\Ecal_{x,i} \left[\sum_{|u|=n+1} v_{\Jcal_u(0)} \Xcal_u(0)^{\omega} \mathds{1}_{\{b_u>t\}} f(\Xcal_{u(t)}(t-b_{u(t)}),\Jcal_{u(t)}(t-b_{u(t)}))B_t\right] \\
=
\Ecal_{x,i} \left[\sum_{|u'|\le n} \sum_{|u|=n+1} v_{\Jcal_u(0)} \Xcal_u(0)^{\omega} \mathds{1}_{\{b_u>t\}} f(\Xcal_{u'}(t-b_{u'}),\Jcal_{u'}(t-b_{u'}))B_t \mathds{1}_{\{u(t)=u'\}}\right]    
\end{multline}
and by conditioning on $\overline{\Fcal}_t$ and applying the temporal branching property stated in \cref{prop: branching temporal}, 
\begin{align*}
&\Ecal_{x,i} \left[\sum_{|u|=n+1} v_{\Jcal_u(0)} \Xcal_u(0)^{\omega} \mathds{1}_{\{b_u>t\}} f(\Xcal_{u(t)}(t-b_{u(t)}),\Jcal_{u(t)}(t-b_{u(t)}))B_t\right] \\
&= \Ecal_{x,i} \left[\sum_{|u'|\le n} f(\Xcal_{u'}(t-b_{u'}), \Jcal_{u'}(t-b_{u'}))B_t \,\right.\\
&\hspace{4.5cm}\left.\times \Ecal_{\Xcal_{u'}(t-b_{u'}),\Jcal_{u'}(t-b_{u'})}\left[\sum_{|u|=n+1-|u'|} v_{\Jcal_{u'u}(0)} \Xcal_{u'u}(0)^{\omega} \right]\mathds{1}_{\{b_{u'}\le t<b_{u'}+\zeta_{u'}\}}\right]\\
&= \Ecal_{x,i} \left[\sum_{|u'|\le n} f(\Xcal_{u'}(t-b_{u'}),  \Jcal_{u'}(t-b_{u'}))B_t \, \mathds{1}_{\{b_{u'}\le t<b_{u'}+\zeta_{u'}\}} v_{\Jcal_{u'}(t-b_{u'})} \Xcal_{u'}(t-b_{u'})^{\omega}\right].
\end{align*}
Finally, taking $n\rightarrow\infty$ and using monotone convergence, we obtain the desired result.
\end{proof}
\begin{Rk}\label{rk: temporal mart}
\cref{prop:spine temporal multitype} applied with $f:=\mathds{1}_{\{x\neq \partial\}}$ yields that the temporal analogue of $\Mcal(n)$ in \cref{prop: genealogical martingale},
\[
\Mcal_t := \sum_{i=1}^{\infty} v_{J_i(t)}X_i(t)^{\omega}, \quad t\ge 0,
\]
is a supermartingale with respect to $(\Fcal_t)_{t\ge 0}$.
\end{Rk}

\noindent \textbf{The law of the growth-fragmentation under $\Phat_{x,i}$.} We now describe the law of $\Xbfhat$ under $\Phat_{x,i}$. Loosely speaking, the tagged cell will serve as a backbone evolving as some explicit self-similar multitype Markov process, to which we attach independent copies of the original growth-fragmentation process. We must first reconstruct the whole cell system from the spine by recording the negative jumps of $\Xhat$, as detailed in \cite[Section 4.1]{BBCK}. We will label these by couples $(n,j)$, where $n\ge 0$ is the generation of the tagged cell immediately before the jump, and $j\ge 1$ is the rank (for the usual ranking) of the jump among those of the tagged cell at generation $n$ (including the final jump when the generation changes from $n$ to $n+1$). To each such $(n,j)$ corresponds a growth-fragmentation $\Xbfhat_{n,j}$ stemming from the corresponding jump: if the generation does not change during the $(n,j)$--jump, then we set
\[
\Xbfhat_{n,j}(t) := \left\{\left\{ (\Xcal_{uw}(t-b_{uw}+b_u), \Jcal_{uw}(t-b_{uw}+b_u)), \; w\in\Ub \; \text{and} \; b_{uw}\le t+b_u<b_{uw}+\zeta_{uw} \right\}\right\},
\]
where $u$ is the label of the cell born at the $(n,j)$--jump. Otherwise, the $(n,j)$--jump corresponds to a jump for the generation of the tagged cell and the tagged cell jumps from label $u$ to label $uj$ say, in which case 

\begin{multline*} 
\Xbfhat_{n,j}(t) :=   \left\{\left\{ (\Xcal_{u}(t-b_{u}+b_{uj}), \Jcal_{u}(t-b_{u}+b_{uj})),  \;  b_{u}\le t+b_{uj}<b_{u}+\zeta_{u} \right\}\right\} \\ \cup \left\{\left\{ (\Xcal_{uw}(t-b_{uw}+b_{uj}), \Jcal_{uw}(t-b_{uw}+b_{uj})), \; w\notin \mathbb{T}_{uj} \; \text{and} \;  b_{uj}\le b_{uw}\le t+b_{uj}<b_{uw}+\zeta_{uw} \right\}\right\},
\end{multline*}
where for $v\in\Ub$, $\mathbb{T}_v := \{vw, \, w\in\Ub\}$.
We agree that $\Xbfhat_{n,j} := \partial$ when the $(n,j)$--jump does not exist, and this completely defines $\Xbfhat_{n,j}$ for all $n\ge 0$ and all $j\ge 1$.

Let $\Fhat(q) := (\Fhat_{i,j}(q))_{i,j\in\Ical}$ be the matrix with entries
\begin{equation} \label{eq: Fhat spine}
\Fhat_{i,j}(q)
=
\begin{cases}
\displaystyle\frac{v_{j}}{v_i}\left(\displaystyle \int_{(-\infty,0)} \Pi_{i,j}(\mathrm{d}x) (1-\mathrm{e}^{x})^{q+\omega} +  q_{i,j} G_{i,j}(q+\omega)\right)
& \text{if} \; i\neq j, \\[3mm]
\kappa_i(\omega+q) & \text{if} \;  i=j.
\end{cases}
\end{equation}

\begin{Thm} \label{thm:spine multitype}
Under $\Phat_{x,i}$, $(\Xhat(t), \Jhat(t), 0\le t<b_{\Lcal})$ is a self-similar Markov process with types in $\Ical$, whose underlying Markov additive process has the matrix exponent $\Fhat$ in \eqref{eq: Fhat spine}. Moreover, conditionally on $(\Xhat(t), \Jhat(t))_{0\le t<b_{\Lcal}}$ and  $\Jhat_{n,j}$, the processes $\Xbfhat_{n,j}$, $n\ge 0$, $j\ge 1$, are independent and each $\Xbfhat_{n,j}$ has law $\Pbf_{x(n,j), \Jhat_{n,j}}$ where $-x(n,j)$ is the size of the $(n,j)$--th jump.
\end{Thm}

\begin{Rk}
\begin{enumerate}
    \item The law of the generation $n_t$ of the spine at time $t$ is not so explicit as in \cite{BBCK} or \cite{DS} in the constant sign case. In fact, $b_{\Lcal(1)}$ may not be exponential because of the current type of the spine before it jumps.
    \item The proof of \cref{thm:spine multitype} goes through determining all three components $\psihat_{i}$, $\qhat_{i,j}$, and $\Ghat_{i,j}$ of the MAP in \eqref{eq: F matrix}. This sheds light on the structure of the MAP under \eqref{eq: Fhat spine}.
\end{enumerate}
\end{Rk}

We postpone the proof of \cref{thm:spine multitype} until \cref{sec: proof spine}, and discuss instead some applications, which are also new for the signed case \cite{DS}. 
\subsection{Martingale exponents in multitype growth-fragmentation processes}
First, we prove that admissible characteristics $(\omega, (v_i,i\ge 1))$ are associated to the roots of the leading eigenvalue of the spine matrix exponent. Recall from \cref{sec: martingales} the notation $m(q), q\in\R$, for the matrix with nonnegative entries
\[
m_{i,j}(q) := \Eb_i \left[\sum_{0<s<\zeta} |\Delta X(s)|^q \mathds{1}_{\{J_{\Delta}(s)=j\}} \right], 
\]
and $\mathrm{e}^{\lambda(q)}$ for the Perron-Frobenius eigenvalue of $m(q)$. We will also make use of the spectral properties of Markov additive processes listed in \cref{sec: ssmp}. In particular, \cref{prop: Perron-Frobenius MAP} applies to the matrix exponent $\Fhat$ of some spine in our growth-fragmentation process (see \eqref{eq: Fhat spine}): in this case, we shall denote by $\widehat{\chi}(q)$ the leading eigenvalue.

\begin{Prop}\label{prop: 2 exponents}
Assume that $((v_i)_{i\in\Ical}, \omega)$ is admissible, and let $\Fhat$ be the matrix exponent of the associated spine. Then $((v'_i)_{i\in\Ical}, \omega')$ is admissible if and only if $\big(\frac{v'_i}{v_i}\big)_{i\in\Ical}$ is an eigenvector of $\Fhat(\omega'-\omega)$ associated with the eigenvalue $0$. In particular, the exponents $\omega'$ for which $\lambda(\omega')=0$ are exactly the roots of $\widehat{\chi}(\omega'-\omega)=0$. 
\end{Prop}

Note that \cref{prop: 2 exponents} contains (and reproves) the fact that there can be at most two martingale exponents $\omega$, by convexity of the leading eigenvalue (\cref{prop: convexity leading eigenvalue}).

\begin{proof}
Set $\Delta \omega:=\omega'-\omega$ and $c$ the vector with entries $c_i:=\frac{v_i'}{v_i}$, $i\in\Ical$. Then for all $i\in\Ical$,
\begin{equation} \label{eq: eigenvalue Fhat}
\sum_{j\in\Ical} \frac{c_j}{c_i} \Fhat_{i,j}(\Delta\omega) = \kappa_i(\omega') + \sum_{j\ne i} \frac{v_j'}{v_i'} \left(\int_{(-\infty, 0)} \Pi_{i,j}(\mathrm{d}x)(1-\mathrm{e}^{x})^{\omega'} +  q_{i,j} G_{i,j}(\omega')\right) = \Kcal_i(\omega'),
\end{equation}
where the $\Kcal_i$ are the multitype cumulant functions defined in \eqref{eq: multitype cumulants}. This formula proves \cref{prop: 2 exponents} in both directions. The fact that the admissible exponents correspond to the roots of $\widehat{\chi}(\cdot - \omega)$ is a consequence of Perron-Frobenius theory: if $0$ is an eigenvalue of $\Fhat(\omega'-\omega)$ associated with a positive eigenvector, then it must be the leading eigenvalue.

\end{proof}

\medskip
If $\omega$ and $\omega'$ are two martingale exponents, one can actually relate the corresponding spines $(\Xhat,\Jhat)$ and $(\Xhat',\Jhat')$. More precisely, one is obtained from the other one upon tilting the measure by the so-called Wald martingale (see \cref{prop: wald MAP}) for the underlying Markov additive process. Write $(\xihat,\Thetahat)$ and $(\xihat',\Thetahat')$ for the Markov additive processes with respective laws $\widehat{\mathtt{P}}$ and $\widehat{\mathtt{P}}'$ appearing in the Lamperti-Kiu representations of $(\Xhat,\Jhat)$ and $(\Xhat',\Jhat')$ respectively, and let $\Fhat$ be the matrix exponent of $(\xihat,\Thetahat)$. Denote by $v$ and $v'$ the eigenvectors of $m$ corresponding to $\omega$ and $\omega'$, and write $c$ for the vector with entries $c_i:=\frac{v_i'}{v_i}$, $i\in\Ical$. The proof of \cref{prop: 2 exponents} gives that the leading eigenvalue of $\Fhat(\omega'-\omega)$ is $\widehat{\chi}(\omega'-\omega)=0$ and is associated to the positive eigenvector $c$. Hence the Wald martingale at $\omega'-\omega$ for $(\xihat,\Thetahat)$ is 
\[
\widehat{\mathcal{W}}(t) = \frac{c_{\Thetahat(t)}}{c_{\Thetahat(0)}}\mathrm{e}^{(\omega'-\omega)\xihat(t)}, \quad t\ge 0.
\]
The law of $(\xihat,\Thetahat)$ under the probability measure biased by $\widehat{\mathcal{W}}$ is also given by \cref{prop: wald MAP}, and one can check that it coincides with the law of $(\xihat',\Thetahat')$ which is characterised  by \eqref{eq: Fhat spine} but with $(\omega^\prime,v')$ instead of $(\omega, v)$. In other words, for any nonnegative measurable function $f$ and all $x\in\R, i\in\Ical,t\in\R_+$,
\begin{equation}\label{eq: relation between spines}
\widehat{\mathtt{E}}'_{x,i}\big[ f(\xihat'(s),\Thetahat'(s),s\le t)\big] 
=
\widehat{\mathtt{E}}_{x,i}\bigg[ f(\xihat(s),\Thetahat(s),s\le t)\cdot \frac{c_{\Thetahat(t)}}{c_{i}}\mathrm{e}^{(\omega'-\omega)\xihat(t)}\bigg].  
\end{equation}
Alternatively (and perhaps more tellingly), one can apply the many-to-one lemma (\cref{prop:spine temporal multitype}) to relate the two spines. 
 


\subsection{Two genealogical martingales} \label{sec: genealogical martingales Cramer}
In this subsection, we make the following Cramér-type condition, assuming that there exist two \emph{admissible} exponents $\omega_-$ and $\omega_+$ in the interior of the domain where the matrix $m$ is finite, with $0<\omega_-<\omega_+$. In this case, we will add a $+$ or $-$ superscript (or subscript) to the quantities considered before: for example, we denote by $v^+$ and $v^-$ the Perron-Frobenius eigenvectors associated to $m(\omega_+)$ and $m(\omega_-)$ respectively. We also assume that $\widehat{\chi}_-'(\omega_+-\omega_-)$ is finite, where $\widehat{\chi}_-$ is the leading eigenvalue associated to the spine with respect to $\omega_-$. Note that by convexity of $\widehat{\chi}_-$, $\widehat{\chi}_-(0)<0$ and $\widehat{\chi}_-'(\omega_+-\omega_-)>0$. \cref{prop: wald MAP} also shows that the condition that $\widehat{\chi}_-'(\omega_+-\omega_-)$ is finite is the same as assuming that $\widehat{\chi}_+'(\omega_--\omega_+)$ is finite. Since $\widehat{\chi}_-(0)<0$, it is known that $\xihat^-$ drifts to $-\infty$ (see identity (11.12) in \cite{KP}). It is not too hard to see that this entails 
\[
\qquad \Ecal_{1,\pi}\left[\sum_{k\ge 1} \frac{v^-_{\Jcal_k(0)}}{v^-_i} \Xcal_k(0)^{\omega_-} \log |\Xcal_k(0)| \right] = \Ehat_{1,\pi} \Big[ \log \Xhat^-(b_{\Lcal(1)}) \Big] \in (-\infty,0).
\]
Actually, for some specific point we will need to work under the stronger assumption
\[
\text{(H)} \qquad \Ecal_{1,i}\left[\sum_{k\ge 1} \frac{v^-_{\Jcal_k(0)}}{v^-_i} \Xcal_k(0)^{\omega_-} \log |\Xcal_k(0)| \right] \in (-\infty,0), \quad \text{for all } i\in \Ical.
\]
We will emphasize when we assume $\text{(H)}$. 

We state for future reference the following lemma.
\begin{Lem} \label{lem: leading eigenvalue negative}
 For all $\gamma<\omega_+$, the leading eigenvalue of $\xi$ satisfies $\chi(\gamma)<0$.
\end{Lem}
\begin{proof}
Fix $\gamma<\omega_+$. We first remark that for $i\in\Ical$, the multitype cumulant function defined in \eqref{eq: multitype cumulants}, 
\[
\Kcal_i(q)
=
\psi_i(q) + \sum_{j\in\Ical} \frac{v_j}{v_i} \left(\int_{(-\infty,0)}  \Pi_{i,j}(\mathrm{d}x)(1-\mathrm{e}^{x})^{q}
+ q_{i,j} G_{i,j}(q)\right),
\]
is convex. Since $\omega_+$ is a root of $\Kcal_i$, and $\Kcal_i(0) = \psi_i(0)\le 0$, we infer by convexity that $\Kcal_i$ must be negative on $(0,\omega_+)$. Let $(w_i)_{i\in\Ical}$ be a positive eigenvector of $F$ associated with $\chi(\gamma)$. Then for all $i\in\Ical$, using that $\Kcal_i(\gamma)<0$,
\begin{multline*}
	\chi(\gamma) w_i = w_i (\psi_i(\gamma)+q_{i,i}) + \sum_{j\neq i} w_j q_{i,j} G_{i,j}(\gamma) 
				   < -\sum_{j\neq i} \Big(w_i \frac{v_j}{v_i} - w_j\Big) q_{i,j} G_{i,j}(\gamma) \\
				   =-\sum_{j\neq i} \Big(\frac{w_i}{v_i} - \frac{w_j}{v_j}\Big) v_j q_{i,j} G_{i,j}(\gamma). 
\end{multline*}
We now take $i\in\Ical$ such that $\frac{w_i}{v_i}$ is maximal. This yields $\chi(\gamma)w_i<0$, and since $w_i>0$, we get $\chi(\gamma)<0$.
\end{proof}

Our goal is to carry out in the multitype setting the analysis conducted in \cite{BBCK} of the two martingales associated to $\omega_-$ and $\omega_+$ (see notably Sections 2.3 and 3.3 there). We start by the study of the two genealogical martingales $(\Mcal^+(n),n\ge 0)$ and $(\Mcal^-(n),n\ge 0)$ with associated exponents $\omega_+$ and $\omega_-$, namely
\[
\Mcal^+(n)=  \sum_{|u|=n+1} v^+_{\Jcal_u(0)} \Xcal_u(0)^{\omega_+}, \quad n\ge 0,
\]
and 
\[
\Mcal^-(n)=  \sum_{|u|=n+1} v^-_{\Jcal_u(0)} \Xcal_u(0)^{\omega_-}, \quad n\ge 0.
\]
Our arguments are inspired from \cite[Section 2.3]{BBCK}, but rely on a multitype version of Biggins' martingale convergence theorem \cite{Big,Lyo}, due to Kyprianou and Sani \cite{KS}. Note that both $\Mcal^+$ and $\Mcal^-$ converge almost surely as nonnegative martingales; we now investigate whether the limit is degenerate or not.

\begin{Prop} \label{prop: convergence M_+}
For all $(x,i)\in\R_+^*\times \Ical$, the martingale $(\Mcal^+(n),n\ge 0)$ converges to $0$ almost surely under $\Pcal_{x,i}$.
\end{Prop}
\begin{proof}
The proof is based on Biggins' martingale convergence theorem for multitype branching random walk (see \cite[Theorem 1]{KS}). Recall that, as noted in \cref{sec: construction GF multitype}, the process $(-\log\Xcal_u(0), \Jcal_u(0))_{u\in\Ub}$ forms a multitype branching random walk. Since $\lambda(\omega_+)=0$ and $\omega_+>0$, the result will follow provided that $\lambda'(\omega_+)> 0$. But since $\omega_+>\omega_-$, by convexity of $\lambda$ (see the discussion after identity \eqref{eq: log Laplace convexity}), $\lambda'(\omega_+)>0$, which concludes the proof.
\end{proof}

\begin{Rk}
We observe that we can combine the ideas of \cite[Section 4.2]{BBCK} with Theorem \cref{entrancelaw} to construct the growth-fragmentation under a version of $\Phat^+_{x,i}$ starting from $0$. Indeed, the condition that $\widehat{\chi}_+'(0)>0$ guarantees that $\xihat^+$ drifts to $+\infty$, see for instance the strong law of large numbers for MAPs that appears in  identity (11.12) in \cite{KP} (see also Theorem 34 in \cite{DeDoKy}). This defines the law of the spine started from $0$, and one can then rebuild the tree using the spine as a backbone as in \cite{BBCK}.
\end{Rk}

\medskip
\noindent In contrast, the martingale $(\Mcal^-(n),n\ge 0)$ satisfies the following properties. We assume that the chain $\Jhat=(\Jcal_{\Lcal(k)}(0), k\ge 0)$ describing the type of the spine at generation $k$ is irreducible and aperiodic, and henceforth has  invariant probability measure $(\mu_i)_{i\in \mathcal{I}}$ where $\mu_i=\pi_i q_{i,i}$.
\begin{Prop} \label{prop: martingale M^-}
The following results regarding the martingale $(\Mcal^-(n),n\ge 0)$ hold:
\begin{itemize}
\item[i)] $\Mcal^-(1)\in L^{\omega_+/\omega_-}$ under any initial distribution;
\item[ii)] $(\Mcal^-(n),n\ge 0)$ is a uniformly integrable martingale. In particular, its almost sure limit $\Mcal^-(\infty)$ is also an $L^1$ limit, and hence non degenerate.
\item[iii)] Assume the ordinate $\xi$ is non-lattice, \textit{i.e.} does not lie in some $a\Z$, $a>0$, \emph{and assume that $\text{(H)}$ holds}. Then for all $(x,i)\in\R_+\times\Ical$,
\begin{equation}
\Pcal_{x,i}(\Mcal^-(\infty)>t) \underset{t\to\infty}{\sim} C t^{-\omega_+/\omega_-},
\end{equation}
for some constant $C>0$.
\end{itemize}
\end{Prop}

\begin{proof}
The proof is a generalisation of that of \cite[Lemma 2.3]{BBCK} to the multitype setting, which in particular requires the extension of \cite{JO-C} included as an appendix in \cref{sec: appendix}.

We first prove part (i) which uses the finiteness criterion for the moments of the exponential functional of Markov additive processes stated as \cref{prop: moments exponential functionals}. First, we claim that for all $q\ge 0$, for which the matrix $m(q)$ is finite, the process
\begin{equation} \label{eq: S_t}
S_t := \sum_{0<s\le t\wedge \zeta} |\Delta X(s)|^{q}, \quad t\ge 0,
\end{equation}
has predictable compensator
\begin{equation} \label{eq: compensator S_t}
S^{(p)}_t := \int_0^{t\wedge\zeta} \mathrm{d}s  X(s)^{q-\alpha} \int_{(-\infty,0)} (1-\mathrm{e}^x)^{q} \widetilde{\Lambda}_{J(s)}(\mathrm{d}x), \quad t\ge 0,
\end{equation}
where for all $i\in\Ical$, $\widetilde{\Lambda}_i := \Lambda_i + \sum_{k\in\Ical} q_{i,k} \Lambda_{U_{i,k}}$. This follows from the same kind of computations as in the determination of the multitype cumulants, see \cref{sec: cumulants}. Note that by the Lamperti-Kiu representation,
\[
\int_0^{\zeta}  X(s)^{q-\alpha} \mathrm{d}s = \int_0^{\infty} \mathrm{e}^{q \xi(s)}\mathrm{d}s,
\]
is the exponential functional of $\xi$. It therefore follows from \cref{lem: leading eigenvalue negative} and \cref{prop:  moments exponential functionals} (together with a crude bound) that $S^{(p)}$ is bounded in $L^{\omega_+/q}$. Now we take $q=\omega_-$. It remains to prove that $S-S^{(p)}$ is bounded in $L^{\omega_+/\omega_-}$: this would entail that $S$ is bounded in $L^{\omega_+/\omega_-}$, and hence $\Mcal^-(1)\in L^{\omega_+/\omega_-}$ by another crude bound. Now, the process
\[
S_t - S^{(p)}_t = \sum_{0<s\le t\wedge \zeta} |\Delta X(s)|^{\omega_-} - \int_0^{t\wedge\zeta} \mathrm{d}s  X(s)^{\omega_- - \alpha} \int_{(-\infty,0)} (1-\mathrm{e}^x)^{\omega_-} \widetilde{\Lambda}_{J(s)}(\mathrm{d}x), \quad t\ge 0,
\]
is a purely discontinuous martingale, and so has quadratic variation 
\[
V_t := \sum_{0<s\le t\wedge\zeta}  |\Delta X(s)|^{2\omega_-}, \quad t\ge 0.
\]
By the Burkholder-Davis-Gundy inequality, it is enough to prove that, for all $(x,i)\in\R_+\times\Ical$, $\Eb_{x,i}[V_{\zeta}^{\omega_+/(2\omega_-)}]<\infty$. Assume for a moment that $\omega_+/\omega_-<2$. Then for $(x,i)\in\R_+\times\Ical$,
\[
\Eb_{x,i}[V_{\zeta}^{\omega_+/(2\omega_-)}] \le \Eb_{x,i}\bigg[ \sum_{0<s\le \zeta}  |\Delta X(s)|^{\omega_+}\bigg]. 
\]
Setting $C = \min_{i\in\Ical} v_i^+$, and using the definition of $v^+$ and $\omega_+$, we end up with the desired inequality:
\[
\Eb_{x,i}[V_{\zeta}^{\omega_+/(2\omega_-)}] \le C^{-1} \Eb_{x,i}\bigg[ \sum_{0<s\le \zeta} v^+_{J_{\Delta}(s)}  |\Delta X(s)|^{\omega_+}\bigg] = C^{-1}.
\]
Next, suppose $2\le \omega_+/\omega_-<4$. Then replacing $q$ by $2\omega_-$ in \eqref{eq: S_t} and \eqref{eq: compensator S_t}, the previous arguments show that $S^{(p)}_t$ is bounded in $L^{\omega_+/(2\omega_-)}$, and that $S_t-S_t^{(p)}$ is a purely discontinuous martingale with quadratic variation
\[
\sum_{0<s\le t\wedge \zeta} |\Delta X(s)|^{4\omega_-}, \quad t\ge 0.
\]
The same inequality as before shows that the latter process is bounded  in $L^{\omega_+/(4\omega_-)}$, and we conclude again by the Burkholder-Davis-Gundy inequality that $S_t = \sum_{0<s\le t\wedge \zeta} |\Delta X(s)|^{2\omega_-}$ is bounded in $L^{\omega_+/(2\omega_-)}$. The same arguments hold in full generality when $2^k\le \omega_+/\omega_-<2^{k+1}$, $k\ge 0$, by a simple recursion.

Part (ii) is a consequence of Biggins' martingale convergence theorem for multitype branching random walks \cite[Theorem 1]{KS}. We use the same arguments as in the proof of \cref{prop: convergence M_+}: since $\omega_-<\omega_+$ are two roots of $\lambda$, and $\lambda$ is convex (see the discussion after \eqref{eq: log Laplace convexity}), we infer that $\lambda'(\omega_-)<0$. The $L\log L$ integrability condition is straightforward since $m(\omega_+)$ is finite.

 Finally, part (iii) comes from a multitype version of the tail estimates in multiplicative cascades \cite{JO-C}. Setting $\widetilde{\Mcal}^{(i)} := \Mcal^-(\infty)/v^-_i$, the branching property entails that $\widetilde{\Mcal}^{(i)}$ satisfies the identity in distribution:
\[
\widetilde{\Mcal}^{(i)} \overset{\Lcal}{=} \sum_{k\ge 1} \frac{v^-_{\Jcal_k}}{v^-_i}\Delta_k^{\omega_-} \widetilde{\Mcal}_k, 
\]
where the $\Delta_k$ are the absolute values of the jumps of $X$, the $\Jcal_k$ are the corresponding types, and the $\widetilde{\Mcal}_k$ form a sequence of independent copies of $\widetilde{\Mcal}^{(j)}$, with initial type $j=\Jcal_k$, which are further independent of the jumps $\Delta_k$. We want to apply \cref{thm: tail behaviour smoothing}: this would imply the desired tail asymptotics for $\widetilde{\Mcal}^{(i)}$, and hence for $\Mcal^-(\infty)$. In order to apply \cref{thm: tail behaviour smoothing}, the only non-trivial assumption to show is that for $\beta<\omega_+/\omega_-$, $\Mcal^{-}(\infty)\in L^\beta$. We include this fact as \cref{prop: Mcal L^beta} in Appendix \ref{sec: appendix} to keep the paper at its present pace.

\end{proof}

The next result produces temporal martingales related to the genealogical martingales defined in \cref{sec: martingales} (as explained in \cref{rk: temporal mart}, these processes are always supermartingales). The positive case appears as \cite[Corollary 3.5]{BBCK} (beware that our definition of the self-similarity index $\alpha$ is the opposite of that of \cite{BBCK}). 
\begin{Prop}
\noindent For $\alpha\ge0$, the process
\[
\Mcal^+_t:= \sum_{i=1}^{\infty} v^+_{J_i(t)} X_i(t)^{\omega_+}, \quad t\ge 0,
\]
is a $\Pbf_{x,i}$--martingale, whereas for $\alpha\le 0$,  
\[
\Mcal^-_t:= \sum_{i=1}^{\infty} v^-_{J_i(t)} X_i(t)^{\omega_-}, \quad t\ge 0,
\]
is a $\Pbf_{x,i}$--martingale.
\end{Prop}
\begin{proof}
The proof is close to \cite[Corollary 3.5]{BBCK} once we translate everything into the Markov additive process framework. We restrict to proving the result for $\alpha\ge 0$ (similar arguments work in the case when $\alpha\le 0$). By the branching property, the result follows if we prove that $\Ebf_{x,i} [\Mcal_t^+]= v^+_i x^{\omega_+}$.  From the many-to-one lemma (\cref{prop:spine temporal multitype}), we see that
\[
\Ebf_{x,i}[\Mcal_t^+] = v^+_i x^{\omega_+}\Phat_{x,i}[\Xhat^+(t)\ne \partial],
\]
where $\Xhat^+$ denotes the spine corresponding to the exponent $\omega_+$. Hence $\Mcal^+$ is a martingale if and only if $\Xhat^+$ has infinite lifetime. By convexity of $\widehat{\chi}_+$, we know that $\widehat{\chi}_+'(0)>0$. This implies that for the underlying Markov additive process $(\xihat^+,\Thetahat^+)$ of $\Xhat^+$, $\xihat^+$ drifts to $+\infty$ independently of the initial state (see \cite[Corollary XI.2.7]{Asm} and \cite[Lemma 2.14]{Iva}). Therefore, the exponential functional 
\[
I(\alpha \xihat^+) := \int_0^{\infty} \exp(\alpha \xihat^+(u)) \mathrm{d}u,
\]
is infinite since $\alpha\ge 0$, and we conclude by the Lamperti-Kiu construction that $\Xhat^+$ has infinite lifetime.
\end{proof}
Similarly to the positive  case (see \cite[Corollary 3.5]{BBCK}), we deduce the long-term behaviour of  $\Ebf_{x,i}[\Mcal^-_t]$ as $t$ increases. Such a result will be very useful in the next subsection.

\begin{Prop}\label{prop:behamart} Assume that the Markov additive processes $(\xihat^-, \Theta^-)$ and $(\xihat^+, \Theta^+)$ are not concentrated on a lattice. For $\alpha<0$, there exists some constant $c^+_i, C^+_i\in [0,\infty)$ such that for $t$ large enough,
\begin{equation}\label{eq:behamart1}
v^+_i x^{\omega_-} c^+_i t^{\frac{\omega_+-\omega_-}{\alpha}} \le \Ebf_{x,i}[\Mcal^+_t]\le v^+_i x^{\omega_-} C^+_i t^{\frac{\omega_+-\omega_-}{\alpha}},
\end{equation}
and for $\alpha>0$, there exists some constant $c^-_i, C^-_i\in [0,\infty)$ such that for $t$ large enough,
\begin{equation}\label{eq:behamart2}
v^-_i x^{\omega_+} c^-_i t^{-\frac{\omega_+-\omega_-}{\alpha}} \le \Ebf_{x,i}[\Mcal^-_t] \le v^-_i x^{\omega_+} C^-_i t^{-\frac{\omega_+-\omega_-}{\alpha}}.
\end{equation}
Moreover, for $1<p<\omega_+/\omega_-$, we have
\begin{equation}\label{eq:behamart3}
\Ebf_{x,i}\left[\sum_{i=1}^{\infty} \Big(v^-_{J_i(t)}X_i(t)^{\omega_-}\Big)^p\right]\to 0, \qquad \textrm{as}\quad t\to \infty.
\end{equation}
\end{Prop}
\begin{proof} We only  prove \eqref{eq:behamart2} since the behaviour in \eqref{eq:behamart1} follows from similar arguments. Similarly to the proof of the previous proposition, from the many-to-one lemma (\cref{prop:spine temporal multitype}), we see that
\[
\Ebf_{x,i}[\Mcal_t^-] = v_i^- x^{\omega_-}\Phat_{x,i}(\Xhat^-(t)\ne \partial),
\]
where $\Xhat^-$ denotes the spine corresponding to the exponent $\omega_-$. From the Lamperti-Kiu representation \eqref{eq: Lamperti MAP}, we observe
\[
\Ebf_{x,i}[\Mcal_t^-] = v^-_i x^{\omega_-}\mathtt{P}_{0,i}\left(I(\alpha \xihat^-) >t x^{-\alpha}\right),
\]
where
\[
I(\alpha \xihat^-):= \int_0^\infty e^{\alpha\xihat^-(s)}\mathrm{d}s,
\]
and  $(\xihat^-, \widehat{\Theta}^-)$  denotes the underlying Markov additive process with Matrix exponent $\Fhat_-$ and leading eigenvalue $\widehat{\chi}_-$. Since $\alpha>0$ and $\widehat{\chi}_-^\prime(0)<0$, the  exponential functional $I(\alpha \xihat^-)$ is a.s. finite. Moreover, since $(\xihat^-, \widehat{\Theta}^-)$ satisfies  Cram\'er's condition with $\omega_+-\omega_-$, we deduce   from Proposition \ref{prop: tails exponential functionals} that the asymptotic behaviour in \eqref{eq:behamart2} is satisfied.

In order to obtain  \eqref{eq:behamart3}, we first observe again from the many-to-one lemma (\cref{prop:spine temporal multitype}) and the Lamperti-Kiu representation \eqref{eq: Lamperti MAP} that
\[
\begin{split}
\Ebf_{x,i}\left[\sum_{i=1}^{\infty} \Big(v^-_{J_i(t)}X_i(t)^{\omega_-}\Big)^p\right]&= v^-_i x^{\omega_-}\Ehat_{x,i}\left[ (v^-_{\Jhat(t)})^{(p-1)}\Xhat(t)^{\omega_- (p-1)}\mathds{1}_{\{t< \zeta\}}\right]\\
&= v^-_i x^{p\omega_-}\mathtt{E}_{0,i}\left[ (v^-_{\widehat{\Theta}( \varphi(tx^{-\alpha}))})^{(p-1)}e^{\omega_- (p-1)\xihat^-(\varphi(tx^{-\alpha}))}\mathds{1}_{\{t< x^\alpha I(\alpha \xihat^-)\}}\right]\\
&\le C_{x,v^-, p}\Bigg(\mathtt{E}_{0,i}\left[ e^{\omega_- (p-1)\xihat^-(\varphi(tx^{-\alpha}))}\mathds{1}_{\{ I(\alpha \xihat^-)=\infty\}}\right]\\
&\hspace{4cm}+ \mathtt{P}_{0,i}\left(  t<x^\alpha I(\alpha \xihat^-)<\infty\right)\Bigg),
\end{split}
\]
where $C_{x,v^-, p}=v^-_i x^{p\omega_-}\max_{j\in \mathcal{I}}(v^-_{j})^{(p-1)}$. On the one hand, the second term in the right-hand side clearly goes to 0 as $t$ increases. On the other hand, from the change of measure induced by the Wald martingale  (see Proposition \ref{prop: wald MAP}) and the fact that $\varphi$ is a stopping time, it is clear that
\[
\mathtt{E}_{0,i}\left[ e^{\omega_- (p-1)\xihat^-(\varphi(tx^{-\alpha}))}\mathds{1}_{\{ I(\alpha \xihat^-)=\infty\}}\right]\le C^\prime\mathtt{E}^{((p-1)\omega_-)}_{0,i}\left[ e^{\varphi(tx^{-\alpha})\widehat{\chi}_-(\omega_- (p-1))}\mathds{1}_{\{ I(\alpha \xihat^-)=\infty\}}\right],
\]
where $C^\prime=\max_{j\in \mathcal{I}} \frac{w_i(\omega_- (p-1))}{w_j(\omega_- (p-1))}$.
The right hand side of the above identity clearly goes to 0 since $\widehat{\chi}_-(\omega_- (p-1))<0$, as $t$ increases. In particular, this implies that \eqref{eq:behamart3}  holds.  \end{proof}

We conclude this section by the following result on the finiteness of moments of the temporal martingale, which will come in handy in the next section.
\begin{Prop} \label{prop: L^p bound temporal martingale}
Let $\alpha\le 0$. For all $0<p<\omega_+/\omega_-$, the martingale $(\Mcal_t^-, t\ge 0)$ is bounded in $L^p$.
\end{Prop}
\begin{proof}
The proof follows closely that of \cite[Theorem 3.7]{BBCK}. Recall from \cref{sec: construction GF multitype} the notation $(\overline{\Fcal}_t, t\ge 0)$ for the filtration of the growth-fragmentation process enriched with the generations. Let $x>0$ and $i\in\Ical$. By the branching property in \cref{prop: branching temporal}, for $t\ge 0$,
\[
\Ecal_{x,i}\Big[\Mcal^-(n) \, \Big| \, \overline{\Fcal}_t\Big]
\ge
 \sum_{|u|\le n} \mathds{1}_{\{b_u\le t\}} v^-_{\Jcal_{u}(t-b_u)} \Xcal_u(t-b_u)^{\omega_-}. 
\]
Take $n\to \infty$. The left-hand side converges to $\Ecal_{x,i}[\Mcal^-(\infty) \, | \, \overline{\Fcal}_t]$, whereas the right-hand side converges to $\Mcal_t^-$. Therefore 
\[
\Ecal_{x,i}\Big[\Mcal^-(\infty) \, \Big| \, \overline{\Fcal}_t\Big]
\ge
\Mcal_t^-. 
\]
Then, by Jensen's inequality, we get
\[
\Ebf_{x,i}[|\Mcal_t^-|^p]
\le
\Ecal_{x,i}[\Mcal^-(\infty)^p],
\]
for all $p\in (1,\omega_+/\omega_-)$. This concludes the proof.
\end{proof}



\subsection{Convergence of the empirical measure}
Here, we are interested in the convergence of the empirical measure $\rho^{(\alpha)}$ given by 
\[
\langle \rho^{(\alpha)}_t, f\rangle:=\sum_{i=1}^\infty v^-_{J_i(t)}X_i(t)^{\omega_-}f(t^{-1/\alpha} X_i(t), J_i(t)),
\]
for $f:(0,\infty)\times \mathcal{I}\to \mathbb{R}$ bounded function and such that $x\mapsto f(x,i)$ is continuous, for all $i\in \mathcal{I}$, and $\alpha<0$. Throughout this section, we write $\xihat$ instead of $\xihat^-$, and $(\Xhat, \Jhat)$ instead of $(\Xhat^-,\Jhat^-)$. We shall also suppose that the process $\xihat $ associated with the tagged cell $\Xhat$  is not concentrated on a lattice. In order to state our result, we introduce the probability measure $\rho$ on $(0,\infty)\times \mathcal{I}$ by 
\[
\int_{(0,\infty)\times\mathcal{I} }f(y, i) \rho(\mathrm{d}y,\mathrm{d}i)=-\frac{1}{\alpha{\Big|\widehat{\mathtt{E}}^\natural_{0,\pi}[\xihat_1]\Big|}}\sum_{i\in \mathcal{I}}\pi_i\widehat{\mathtt{E}}^\natural_{0,i}\left[\frac{1}{I\big(\alpha\xihat\,\,\big)} f\left(\frac{1}{I\big(\alpha\xihat\,\,\big)^{1/\alpha}}, i\right)\right],
\]
where we recall that $\widehat{\mathtt{P}}^\natural$ denotes the law of the dual of $\xihat$ and 
\[
I\big(\alpha\xihat\,\,\big):=\int_0^\infty e^{\alpha\xihat(s)}{\rm d} s.
\]
We also observe that $\Big|\widehat{\mathtt{E}}^\natural_{0,\pi}[\xihat_1]\Big|=-\widehat{\mathtt{E}}_{0,\pi}[\xihat_1] =-\widehat{\chi}_-'(0)$. We have the following convergence result for the empirical measure $\rho^{(\alpha)}_t$ (see \cite{Dadoun} for the classical growth-fragmentation framework).
\begin{Thm}\label{empmeas}  For every $1<p<\omega_+/\omega_{-}$ and for every $f:(0,\infty)\times \mathcal{I}\to \mathbb{R}$ bounded function and such that $x\mapsto f(x,i)$ is continuous, for all $i\in \mathcal{I}$,
\[
\lim_{t\to\infty}\langle \rho^{(\alpha)}_t, f\rangle=\Mcal^-(\infty) \int_{(0,\infty)\times \mathcal{I}} f(y, i) \rho(\mathrm{d} y, \mathrm{d} i), \qquad \textrm{in}\quad L^p.
\]
In particular, the random measure $\rho^{(\alpha)}_t$ converges in probability towards $\Mcal^-(\infty)\rho $, as $t\to\infty$, in the space of finite measures on $(0,\infty)\times\mathcal{I} $ endowed with the topology of weak convergence.
\end{Thm} 
In order to deduce the above result, the following lemma is required.
\begin{Lem}\label{lemconv} Under $\Phat_{1,i}$, the pair $(t^{-1/\alpha} \Xhat(t), \Jhat(t))$ converges in distribution to $\rho$. Moreover 
\[
\int_{(0,\infty)\times \mathcal{I}} y^{-\alpha q} \rho(\mathrm{d} y, \mathrm{d} i)<\infty,
\]
for $0<q<1-(\omega_+-\omega_-)/\alpha$.
\end{Lem}
\begin{proof} We observe that $(1/\Xhat, \Jhat)$ is a self-similar Markov process with types and  index $-\alpha$ which is associated with the Markov additive process $(-\xihat, \Thetahat)$. Hence, from Theorem \ref{entrancelaw},  all we need to verify is that $-\xihat$ has finite and strictly positive mean which, in particular, implies that the associated ascending ladder MAP $\hat{H}^+$ has  also finite mean, see for instance Theorem 35 in \cite{DeDoKy}. The finiteness of the mean of $-\xihat$ follows from the identity
\[
-\widehat{\mathtt{E}}_{0,\pi}[\widehat{\xi}_1]=-\widehat{\chi}_-'(0)>0.
\]
In order to deduce the existence of moments, we first observe that since $\omega_+-\omega_-$ is a Cramér number for $\Fhat(z)$, $-(\omega_+-\omega_-)/\alpha$ is a Cramér number for $\Fhat(-\alpha z)$. By identity \eqref{eq: F^natural expression} and the fact that all positive eigenvectors of $\Fhat^\natural(\alpha z)$ are associated to the same eigenvalue, we infer that $-(\omega_+-\omega_-)/\alpha$ is also a Cramér number for $\Fhat^\natural(\alpha z)$. Thus a straightforward application of  Theorem \ref{entrancelaw} and Proposition \ref{prop: moments exponential functionals} will imply the second part of this Lemma.
\end{proof}
\begin{proof}[Proof of Theorem \ref{empmeas}]
From the branching property at time $t$ and self-similarity of $\Xbf$, we have under the event $\Xbf(t)=\{\{(x_1, j_1), (x_2, j_2), \cdots\}\}$ that the following identity holds
\[
\begin{split}
\langle \rho^{(\alpha)}_{t+t^2}, f\rangle&=\sum_{i=1}^\infty  v^-_{J_i(t+t^2)}X_i(t+t^2)^{\omega_-}f\Big((t+t^2)^{-1/\alpha} X_i(t+t^2), J_i(t+t^2)\Big)\\
&=\sum_{i=1}^\infty   v^-_{j_{i}}x_i^{\omega_-}\sum_{k=1}^\infty \frac{v^-_{J_{i,k}(x_i^{-\alpha} t^2)}}{v^-_{j_i}}X_{i,k}(x_i^{-\alpha}t^2)^{\omega_-}f\Big((t+t^2)^{-1/\alpha} x_iX_{i,k}(x_i^{-\alpha} t^2), J_{i,k}(x_i^{-\alpha}t^2)\Big).
\end{split}
\]
Observe that for each $i\ge 1$, the families $\{(X_{i,k}, J_{i,k})$,  $k\ge 1$\} are independent and have the same law as $\mathbf{X}$ under  $\mathbf{P}_{1, j_i}$. For simplicity of exposition, we introduce the random variables
\[
\mathcal{N}_{i}(t):=\sum_{k\ge 1} \frac{v^-_{J_{i,k}(x_i^{-\alpha} t^2)}}{v^-_{j_i}}X_{i,k}(x_i^{-\alpha}t^2)^{\omega_-}f\Big((t+t^2)^{-1/\alpha} x_iX_{i,k}(x_i^{-\alpha} t^2), J_{i,k}(x_i^{-\alpha}t^2)\Big),
\]
and observe that $(\mathcal{N}_{i}(t))_{i\ge 1}$ are independent conditionally on $\{\{(x_1, j_1), (x_2, j_2), \cdots\}\}$. In other words,  we may rewrite
\[
\langle \rho^{(\alpha)}_{t+t^2}, f\rangle=\sum_{i=1}^\infty   \lambda_i(t)\mathcal{N}_{i}(t),
\]
where $\lambda_i(t)=v^-_{J_i(t)}X_i(t)^{\omega_-} =v^-_{j_{i}}x_i^{\omega_-}$. Now, let us  introduce 
\[
\overline{\mathcal{N}}_{i}:=C_1||f||_\infty\sup_{t\ge 0}\sum_{k\ge 1} v^-_{J_{i,k}( t)}X_{i,k}(t)^{\omega_-}, \qquad\textrm{with} \qquad C_1:=\max_{\ell\in \mathcal{I}}\frac{1}{v^-_{\ell}},
\]
and observe $|\mathcal{N}_{i}(t)|\le \overline{\mathcal{N}}_{i}$,  for all $t\ge 0$.  Moreover, conditionally on $\{\{(x_1, j_1), (x_2, j_2), \cdots\}\}$,  by setting $\mathfrak{s}^{(\ell)}:=\{k: j_k=\ell \}$ for $\ell\in \mathcal{I}$,  
we deduce that  the random variables  $(\overline{\mathcal{N}}_{i})_{i\in \mathfrak{s}^{(\ell)}}$ are i.i.d with common distribution 
 \[
 \mathcal{N}:=C_1||f||_\infty\sup_{t\ge 0}\sum_{k\ge 1}^\infty v^-_{J_{k}(t)}X_{k}(t)^{\omega_-}, \qquad\textrm{under}\quad \mathbf{P}_{1,\ell}.
 \] 
Additionally, from \cref{prop: L^p bound temporal martingale} (recall that $1<p<\omega_+/\omega_-$),  it is clear that $(\overline{\mathcal{N}}_{i})_{i\in \mathfrak{s}^{(\ell)}} \in L^p(\mathbf{P}_{1,\ell})$.  Similarly, we obtain
\[
\sup_{t\ge 0}\mathbf{E}_{1,\ell}\left[\left(\sum_{i=1}^\infty v^{-}_{J_i(t)}X_i(t)^{\omega_-}\right)^p\right]<\infty,
\]
and using Proposition \ref{prop:behamart}, we deduce
\begin{equation}\label{tomate}
\lim_{t\to\infty}\mathbf{E}_{1,\ell}\left[\sum_{i=1}^\infty \left(v^{-}_{J_i(t)}X_i(t)^{\omega_-}\right)^p\right]=0.
\end{equation}
With these properties at hand, we  will deduce the following variation of the law of large numbers for any ${\bf k}\in \mathcal{I}$,
\begin{equation}\label{varlln}
\lim_{t\to \infty}\sum_{i=1}^\infty  \lambda_i(t) \Big(\mathcal{N}_{i}(t)-\mathbf{E}_{1,{\bf k}}\Big[\mathcal{N}_{i}(t)| \Xbf(t)\Big]\Big)=0, \qquad \textrm{in} \quad L^p(\mathbf{P}_{1, {\bf k}}).
\end{equation}
In order to do so, we follow a similar strategy as in Lemma 1.5 in Bertoin \cite{Ber06}. 
We take $a>0$, an arbitrarily large real number and introduce the truncated random variables 
\[
\mathcal{N}_{i}(t,a):=\mathds{1}_{\{|\mathcal{N}_{i}(t)|<a\}}\mathcal{N}_{i}(t)\qquad \textrm{for}\quad i\ge 1.
\]
Thus, we deduce the following upper-bound
\begin{equation}\label{equpb}
\begin{split}
\left|\sum_{i=1}^\infty  \lambda_i(t) \Big(\mathcal{N}_{i}(t)-\mathbf{E}_{1,{\bf k}}\Big[\mathcal{N}_{i}(t)| \Xbf(t)\Big]\Big)\right|&\le \left|\sum_{i=1}^\infty  \lambda_i(t)\Big(\mathcal{N}_{i}(t)-\mathcal{N}_{i}(t, a)\Big)\right|\\
&+\left|\sum_{i=1}^\infty   \lambda_i(t) \Big(\mathcal{N}_{i}(t, a)-\mathbf{E}_{1,{\bf k}}\Big[\mathcal{N}_{i}(t, a)\Big| \Xbf(t)\Big]\Big)\right|\\
&+\left|\sum_{i=1}^\infty   \lambda_i(t) \mathbf{E}_{1,{\bf k}}\Big[\Big(\mathcal{N}_{i}(t, a)-\mathcal{N}_{i}(t)\Big)\Big| \Xbf(t)\Big]\right|.
\end{split}
\end{equation}
We consider the first series in the right-hand side of the above inequality and since 
\[
\Big|\mathcal{N}_{i}(t)-\mathcal{N}_{i}(t, a)\Big|\le\mathds{1}_{\big\{|\overline{\mathcal{N}}_{i}|> a\big\}}\overline{\mathcal{N}}_{i},
\]
we get from Jensen's inequality and the Markov property, the following inequality
\[
\begin{split}
\mathbf{E}_{1,{\bf k}}&\left[\left|\sum_{i=1}^\infty  \lambda_i(t)\Big(\mathcal{N}_{i}(t)-\mathcal{N}_{i}(t, a)\Big)\right|^p\right]\\
&\hspace{1cm}\le \mathbf{E}_{1,{\bf k}}\left[\left(\sum_{i=1}^\infty  \lambda_i(t)\right)^p\sum_{i=1}^\infty\frac{\lambda_i(t)}{\sum_{k=1}^\infty  \lambda_k(t)}\mathds{1}_{\big\{|\overline{\mathcal{N}}_{i}|> a\big\}}\overline{\mathcal{N}}_{i}^{\, p}\right]\\
&\hspace{1cm}=\mathbf{E}_{1,{\bf k}}\left[\left(\sum_{i=1}^\infty   \lambda_i(t)\right)^p\sum_{i=1}^\infty\frac{ \lambda_i(t)}{\sum_{k=1}^\infty   \lambda_k(t)}\mathbf{E}_{1,J_i(t)}\left[\mathds{1}_{\big\{|\overline{\mathcal{N}}|> a\big\}}\overline{\mathcal{N}}^p\right]\right]\\
&\hspace{1cm}\leq\mathbf{E}_{1,{\bf k}}\left[\left(\sum_{i=1}^\infty  v^-_{J_{i}(t)}X_i(t)^{\omega_-}\right)^p\right]\max_{j\in \mathcal{I}}\mathbf{E}_{1,j}\left[\mathds{1}_{\big\{|\overline{\mathcal{N}}|> a\big\}}\overline{\mathcal{N}}^p\right],
\end{split}
\]
and the latter quantity converges to $0$ as $a\to \infty$, uniformly for $t\ge 0$. The same argument also shows that 
\[
\lim_{a\to \infty}\sup_{t\ge 0}\mathbf{E}_{1,{\bf k}}\left[\left|\sum_{i=1}^\infty   \lambda_i(t) \mathbf{E}_{1,{\bf k}}\Big[\Big(\mathcal{N}_{i}(t, a)-\mathcal{N}_{i}(t)\Big)\Big| \Xbf(t)\Big]\right|^p\right]=0.
\]
Next, for the second term in the right hand side in \eqref{equpb}, we observe that conditionally on $\Xbf(t)=\{\{(x_1, j_1), (x_2, j_2), \cdots\}\}$, the random variables
\[
\mathcal{N}_{i}(t, a)-\mathbf{E}_{1,{\bf k}}\Big[\mathcal{N}_{i}(t, a)\Big| \Xbf(t)\Big],
\]
are centered, independent and bounded in absolute value by $a$. Hence, conditionally on $\Xbf(t)$,
\[
\sum_{i=1}^n\lambda_i(t)\left(\mathcal{N}_{i}(t, a)-\mathbf{E}_{1,{\bf k}}\Big[\mathcal{N}_{i}(t, a)\Big| \Xbf(t)\Big]\right), \qquad n\ge 1, 
\]
is a martingale bounded in $L^p(\mathbf{P}_{1, {\bf k}})$ and there exists a universal constant $C_p$ such that 
\[
\mathbf{E}_{1,{\bf k}}\left[\left|\sum_{i=1}^\infty\lambda_i(t)\left(\mathcal{N}_{i}(t, a)-\mathbf{E}_{1,{\bf k}}\Big[\mathcal{N}_{i}(t, a)\Big| \Xbf(t)\Big]\right)\right|^p\bigg|\Xbf(t)\right]\le C_p a^p\sum_{i=1}^\infty \lambda_i(t)^p.
\]
Thus from \eqref{tomate}, the latter quantity converges to $0$ as $t\to \infty$ in $L^p(\mathbf{P}_{1, {\bf k}})$. Putting all pieces together allows us to deduce \eqref{varlln}.

In other words, the proof will be completed if we show 
\begin{equation}\label{convermesemp}
\lim_{t\to\infty} \sum_{i=1}^\infty \lambda_i(t)\mathbf{E}_{1,{\bf k}}\Big[\mathcal{N}_{i}(t)\Big| \Xbf(t)\Big]=\Mcal^-(\infty) \int_{(0,\infty)\times\mathcal{I}} f(y, i) \rho(\mathrm{d}y,\mathrm{d}i), \qquad \textrm{in}\quad L^p.
\end{equation}
 From Proposition \ref{prop:spine temporal multitype}, we get
 \[
 \mathbf{E}_{1,{\bf k}}\Big[\mathcal{N}_{i}(t)\Big| \Xbf(t)\Big]=\Ehat_{1,j_i}\Big[f((1+t^{-1})^{-1/\alpha}t^{-2/\alpha}x_i\Xhat(x_i^{-\alpha}t^2),\Jhat(x_i^{-\alpha}t^2))\Big]. 
 \]
 From Lemma \ref{lemconv},  the pair $(s^{-1/\alpha} \Xhat(s), \Jhat(s))$ converges in distribution to $\rho$, under $\Phat_{1,j_i}$.  On the one hand, it follows that 
 \[
 \Ehat_{1,j_i}\Big[f((1+t^{-1})^{-1/\alpha}t^{-2/\alpha}x_i\Xhat(x_i^{-\alpha}t^2),\Jhat(x_i^{-\alpha}t^2))\Big]\xrightarrow[t\to\infty]{}\int_{(0,\infty)\times\mathcal{I} }f(y, i) \rho(\mathrm{d}y,\mathrm{d}i),
 \]
 uniformly in $i$ such that $x_i^{-\alpha} t^{2}>\sqrt{t}$, i.e. $x_i>t^{3/(2\alpha)}$. On the other hand, applying Proposition \ref{prop:spine temporal multitype}, we see that
 \begin{equation}\label{cebolla1}
 \sum_{i=1}^\infty  v^-_{J_{i}(t)}X_i(t)^{\omega_-}\mathds{1}_{\big\{X_i(t)\le t^{3/(2\alpha)}\big\}}, 
 \end{equation}
 has, under  $\mathbf{P}_{1, {\bf k}}$, mean equals
 \[
  v_{\bf k}\Phat_{1,{\bf k}}\Big(t^{-1/\alpha}\Xhat(t)<t^{1/(2\alpha)}\Big)
 \]
 which tends to $0$ as $t\to \infty$. Since \eqref{cebolla1} is bounded in $L^q(\mathbf{P}_{1, {\bf k}})$ for every $p<q<\omega_+/\omega_-$ (\cref{prop: L^p bound temporal martingale}), by H\"older's inequality it also converges to $0$. Putting all pieces together, we deduce that \eqref{convermesemp} holds and thus the first part of the statement.
 
The second part of the statement is derived using the same arguments as in Dadoun \cite{Dadoun}, that is we use a diagonal extraction procedure since the space $\mathcal{C}_c((0,\infty)\times\mathcal{I} )$ of continuous function on $(0,\infty)\times\mathcal{I}$ with compact support   is separable. In other words  there exists an extraction  $\sigma: \mathbb{N}\to \mathbb{N}$, such that, almost surely for all $f\in \mathcal{C}_c((0,\infty)\times\mathcal{I} )$
 \[
 \langle \rho^{(\alpha)}_{t_{\sigma(n)}}, f\rangle\xrightarrow[n\to\infty]{}\Mcal^-(\infty) \int_{(0,\infty)\times\mathcal{I}} f(y, i) \rho(\mathrm{d}y,\mathrm{d}i), 
 \]
 in other words $\rho^{(\alpha)}_{t_{\sigma(n)}}$ converges vaguely to $\Mcal^-(\infty)\rho$, a.s. Since the total mass is conserved, i.e.
\[
 \langle \rho^{(\alpha)}_{t_{\sigma(n)}}, 1\rangle=\sum_{i=1}^\infty v^-_{J_i(t)}X_i(t)^{\omega_-}\xrightarrow[n\to\infty]{}\Mcal^-(\infty)=\langle \Mcal^-(\infty)\rho, 1\rangle, \qquad \textrm{a.s.,}
 \]
 thus the convergence of $\rho^{(\alpha)}_{t_{\sigma(n)}}$ towards $\Mcal^-(\infty)\rho$  is weak which allows us to conclude.
\end{proof} 
\subsection{Proof of \cref{thm:spine multitype}} \label{sec: proof spine}
It is plain that the spine $(\Xhat, \Jhat)$ inherits the Markov property and self-similarity of $(X,J)$, and therefore it can be described in terms of a MAP \emph{via} the representation \eqref{eq: Lamperti MAP}. Without loss of generality, we may restrict to the homogeneous case, i.e. when $\alpha=0$. The result is then easily extended thanks to  the Lamperti time-change. We aim at finding the characteristics $(\psihat_i,\qhat_{i,j},\Ghat_{i,j})$ of the matrix exponent of this MAP. 

\medskip
\noindent \textbf{Description of the spine.}
Let $\Hhat$ be the first time when the type of the spine changes, and $\Jhat(\Hhat)$ denote the corresponding type. Fix $q\ge 0$, and $i,j\in\Ical$.
\begin{itemize}[leftmargin=*]
\item[$\rhd$] \textsc{Determining the Laplace exponent $\psihat_{i}$.}
This part is similar to the proof of \cite[Theorem 4.2]{BBCK}. We denote by $\xihat$ the first component of the MAP corresponding to $\Xhat$, that is
\[
\xihat(s)=\log \Xhat(s), \qquad s\ge 0, 
\]
and $\xihat_k$ the underlying Lévy processes depending on type $k\in\Ical$. We want to show that the Lévy process $\xihat_i$ has Laplace exponent $\psihat_i(q) = \kappa_i(q+\omega)-\kappa_i(\omega)$. Notice that a process $\eta$ with Laplace exponent $\psihat_i$ can be written as the superposition $\eta = \eta^{(1)} + \eta^{(2)}$ of independent L\'evy processes $\eta^{(1)}$ and $\eta^{(2)}$, with respective Laplace exponents $\psi^{(1)}(q):= \psi_i(q+\omega)-\psi_i(\omega)$ and 
\[
\psi^{(2)}(q):= \int_{-\infty}^0 \left( (1-\mathrm{e}^x)^{q+\omega}-(1-\mathrm{e}^x)^{\omega}\right)\Pi_{i,i}(\mathrm{d}x).
\] In particular, $\eta^{(2)}$ is a compound Poisson process with Lévy measure $\mathrm{e}^{\omega x}\widetilde{\Pi}_{i,i}(\mathrm{d}x)$, where $\widetilde{\Pi}_{i,i}(\mathrm{d}x)$ is the image measure of $\Pi_{i,i}(\mathrm{d}x)\mathds{1}_{\{x<0\}}$ through $x\mapsto \log(1-\mathrm{e}^x)$. Let $T$ be the first time when $\eta^{(2)}$ has a jump. The branching property of the cell system and the Markov property of $\eta$ ensures that the result will hold if we manage to prove that the distribution of $(\xihat_i(t), t\le b_{\Lcal(1)})$ is the same as that of $(\eta(t),t\le T)$. Let $f,g$ be two nonnegative measurable functions defined respectively on the space of \textit{càdlàg} trajectories and on $(-\infty,0)$. Let
\[
\Delta \xihat(s)=\log \frac{ \Xhat(s)}{\Xhat(s-)}, \qquad s\ge 0,
\]
then 
\begin{align*}
    &\Ehat_i\Big[f(\xihat(s), s<b_{\Lcal(1)})g(\Delta \xihat(b_{\Lcal(1)})) \mathds{1}_{\{b_{\Lcal(1)}<\Hhat\}}\Big] \\
    &=
    \mathtt{E}_i\Bigg[\sum_{t>0} \frac{v_{\iota_{\Delta}(t)}}{v_i} \mathrm{e}^{\omega \xi(t^-)}\Big(1-\mathrm{e}^{\Delta \xi(t)}\Big)^{\omega} \mathds{1}_{\{\iota_{\Delta}(t)=i\}} \mathds{1}_{\{t\le \rho_i\}} f(\xi(s),s<t)g(\log(1-\mathrm{e}^{\Delta\xi(t)})) \Bigg] \\
    &= 
    \mathtt{E}_i\Bigg[\sum_{0<t<\rho_i} \mathrm{e}^{\omega \xi_i(t^-)}\Big(1-\mathrm{e}^{\Delta \xi_i(t)}\Big)^{\omega} \mathds{1}_{\{\iota_{\Delta}(t)=i\}}  f(\xi_i(s),s<t)g(\log(1-\mathrm{e}^{\Delta\xi_i(t)}))\Bigg] \\
    & \quad + 
    \mathtt{E}_i\Big[ \mathrm{e}^{\omega \xi_i(\rho_i^-)}\Big(1-\mathrm{e}^{U_{i,\Theta(\rho_i)}}\Big)^{\omega} \mathds{1}_{\{\iota_{\Delta}(\rho_i)=i\}}  f(\xi_i(s),s<\rho_i)g(\log(1-\mathrm{e}^{U_{i,\Theta(\rho_i)}}))\Big].
\end{align*}
The compensation formula for $\xi_i$ entails that the first term is 
\begin{multline} \label{eq: psihat first term}
    \mathtt{E}_i\Bigg[\sum_{0<t<\rho_i} \mathrm{e}^{\omega \xi_i(t^-)}\Big(1-\mathrm{e}^{\Delta \xi_i(t)}\Big)^{\omega} \mathds{1}_{\{\iota_{\Delta}(t)=i\}}  f(\xi_i(s),s<t)g(\log(1-\mathrm{e}^{\Delta\xi_i(t)}))\Bigg] \\
    =
    \int_0^{\infty} \mathrm{d}t\, \mathrm{e}^{q_{i,i}t} \mathtt{E}_i\Big[f(\xi_i(s),s<t)\mathrm{e}^{\omega\xi_i(t)}\Big] \int_{-\infty}^0 g(\log(1-\mathrm{e}^x)) (1-\mathrm{e}^x)^{\omega} \Lambda_{i}^{(i)}(\mathrm{d}x).
\end{multline}
The second term can be computed as follows:
\begin{align}
    &\mathtt{E}_i\Big[ \mathrm{e}^{\omega \xi_i(\rho_i^-)}\Big(1-\mathrm{e}^{U_{i,\Theta(\rho_i)}}\Big)^{\omega} \mathds{1}_{\{\iota_{\Delta}(\rho_i)=i\}}  f(\xi_i(s),s<\rho_i)g(\log(1-\mathrm{e}^{U_{i,\Theta(\rho_i)}}))\Big] \notag \\
    &= \int_0^{\infty} \mathrm{d}t\, (-q_{i,i}) \mathrm{e}^{q_{i,i}t} \mathtt{E}_i\Big[f(\xi_i(s),s<t)\mathrm{e}^{\omega \xi_i(t)}\Big] \sum_{k\in\Ical\setminus\{i\}} \frac{q_{i,k}}{(-q_{i,i})} \int_{-\infty}^0 g(\log(1-\mathrm{e}^x)) (1-\mathrm{e}^{x})^{\omega} \Lambda_{U_{i,k}}^{(i)}(\mathrm{d}x) \notag \\
    &= \int_0^{\infty} \mathrm{d}t\, \mathrm{e}^{q_{i,i}t} \mathtt{E}_i\Big[f(\xi_i(s),s<t)\mathrm{e}^{\omega \xi_i(t)}\Big] \sum_{k\in\Ical\setminus\{i\}} q_{i,k} \int_{-\infty}^0 g(\log(1-\mathrm{e}^x)) (1-\mathrm{e}^{x})^{\omega} \Lambda_{U_{i,k}}^{(i)}(\mathrm{d}x). \label{eq: psihat second term}
\end{align}
Combining \eqref{eq: psihat first term} and \eqref{eq: psihat second term}, we finally obtain
\begin{multline*}
\Ehat_i\Big[f(\xihat(s), s<b_{\Lcal(1)})g(\Delta \xihat(b_{\Lcal(1)})) \mathds{1}_{\{b_{\Lcal(1)}<\Hhat\}}\Big] \\
=
\int_0^{\infty} \mathrm{d}t \,\mathrm{e}^{q_{i,i}t} \mathtt{E}_i\Big[f(\xi_i(s),s<t)\mathrm{e}^{\omega\xi_i(t)}\Big] \int_{-\infty}^0 g(y) \mathrm{e}^{\omega y} \widetilde{\Pi}_{i,i}(\mathrm{d}y).
\end{multline*}
This proves that $(\xihat_i(s),s<b_{\Lcal(1)})$ and $\Delta \xihat(b_{\Lcal(1)})$ are independent. The latter is distributed as $-\mathds{1}_{\{y<0\}}\frac{1}{q_{i,i}+\psi_i(\omega)}\mathrm{e}^{\omega y} \widetilde{\Pi}_{i,i}(\mathrm{d}y)$, which is the law of $\Delta \eta^{(2)}(T)$. The former is $\xi_i$ biased by the exponential martingale, and killed at an independent exponential time with parameter $-q_{i,i}-\psi_i(\omega)$, hence has Laplace exponent $\psi_i(q+\omega)+q_{i,i}$. We retrieve the Laplace exponent of $\eta^{(1)}$ killed at $T$, and conclude that $(\xihat_i(t),t\le b_{\Lcal(1)})$ evolves as $(\eta(t), t\le T)$.
\medskip
\item[$\rhd$] \textsc{Determining the Laplace transform $\Ghat_{i,j}$ of the special jumps.}
We want to compute
\[
\Ehat_i\left[ \left|\frac{\Xhat(\Hhat)}{\Xhat(\Hhat^-)}\right|^q \mathds{1}_{\{\Jhat(\Hhat)=j\}}\right].
\]
We first split over the possible current generations for this special jump to occur:
\begin{equation} \label{eq:geometric}
\Ehat_i\left[ \left|\frac{\Xhat(\Hhat)}{\Xhat(\Hhat^-)}\right|^q \mathds{1}_{\{\Jhat(\Hhat)=j\}}\right]
=
\sum_{k=0}^{\infty} \underbrace{\Ehat_i\left[ \left|\frac{\Xhat(\Hhat)}{\Xhat(\Hhat^-)}\right|^q \mathds{1}_{\{\Jhat(\Hhat)=j\}} \cdot  \mathds{1}_{\{b_{\Lcal(k)}< \Hhat \le b_{\Lcal(k+1)}\}}\right]}_{:=a_k}.
\end{equation}
For $k\ge 1$, applying successively the Markov property at time $b_{\Lcal(1)}$, the self-similarity of $\Xhat$ and the definition of $\Phat_i$, we get
\begin{align*}
a_k
&= 
\Ehat_i\left[ \mathds{1}_{\Hhat>b_{\Lcal(1)}} \cdot \Ehat_{\Xhat(b_{\Lcal(1)}),i}\left[\left|\frac{\Xhat(\Hhat)}{\Xhat(\Hhat^-)}\right|^q \mathds{1}_{\{\Jhat(\Hhat)=j\}} \cdot  \mathds{1}_{\{b_{\Lcal(k-1)}< \Hhat \le b_{\Lcal(k)}\}}\right]\right] \\
&=  \Phat_i(\Hhat>b_{\Lcal(1)}) \cdot  \Ehat_{i}\left[\left|\frac{\Xhat(\Hhat)}{\Xhat(\Hhat^-)}\right|^q \mathds{1}_{\{\Jhat(\Hhat)=j\}} \cdot  \mathds{1}_{\{b_{\Lcal(k-1)}< \Hhat \le b_{\Lcal(k)}\}}\right]  \\
&=
\Eb_i\left[\sum_{0<s<\zeta} \frac{v_{J_{\Delta}(s)}}{v_i} \, |\Delta X(s)|^{\omega} \mathds{1}_{\{H\ge s\}}\mathds{1}_{\{J_{\Delta}(s)=i\}}\right] \cdot a_{k-1}, 
\end{align*}
where recall that $H$ denotes the first jump time of $J$. Therefore 
$a_k = \mu_{i,i}(\omega) \cdot a_{k-1} = \mu_{i,i}(\omega)^k \cdot a_{0}$,
with \[
\mu_{i,i}(\omega) := \Eb_i\left[\sum_{0<s<\zeta} \, |\Delta X(s)|^{\omega} \mathds{1}_{\{H\ge s\}} \mathds{1}_{\{J_{\Delta}(s)=i\}}\right].
\] Then, provided $\mu_{i,i}(\omega)<1$, identity \eqref{eq:geometric} triggers
\begin{equation} \label{eq:geo sum}
\Ehat_i\left[ \left|\frac{\Xhat(\Hhat)}{\Xhat(\Hhat^-)}\right|^q \mathds{1}_{\{\Jhat(\Hhat)=j\}}\right]
=
\frac{a_0}{1-\mu_{i,i}(\omega)}.
\end{equation}
It remains to compute $a_0$ and $\mu_{i,i}(\omega)$. We begin with the latter:
\begin{align*}
    \mu_{i,i}(\omega) 
    &= \mathtt{E}_i\left[ \sum_{0<s\le \rho_i} \mathrm{e}^{\omega \xi(s^-)} \Big(1-\mathrm{e}^{\Delta \xi(s)}\Big)^{\omega} \mathds{1}_{\{\iota_{\Delta}(s)=i\}} \right] \\
    &= \mathtt{E}_i\left[ \sum_{0<s<\rho_i} \mathrm{e}^{\omega \xi_i(s^-)} \Big(1-\mathrm{e}^{\Delta \xi_{i}(s)}\Big)^{\omega}\mathds{1}_{\{\iota_{\Delta}(s)=i\}}\right] +  \mathtt{E}_i\bigg[\mathrm{e}^{\omega \xi_i(\rho_i^-)} \Big(1-\mathrm{e}^{U_{i,\Theta(\rho_i)}}\Big)^{\omega} \mathds{1}_{\{\iota_{\Delta}(\rho_i) = i\}} \bigg].
\end{align*}
A computation similar to \eqref{eq: cumulant A} gives
\[
\mu_{i,i}(\omega) 
=
-\frac{1}{q_{i,i}+\psi_i(\omega)} \int_{-\infty}^0 \Pi_{i,i}(\mathrm{d}x) (1-\mathrm{e}^x)^{\omega},
\]
provided $\psi_i(\omega)+q_{i,i}<0$. On the other hand, in $a_0$ the type of the spine can either change because $J$ jumps to $j$ (\textit{i.e.} $\Hhat<b_{\Lcal(1)}$), or because one picks a jump of type $j$ at time $b_{\Lcal(1)}$ (\textit{i.e.} $\Hhat=b_{\Lcal(1)}$). This writes
\[
a_{0} = A + B,
\]
with 
\[
A = \Ehat_i\Bigg[ \left|\frac{\Xhat(\Hhat)}{\Xhat(\Hhat^-)}\right|^q \mathds{1}_{\{\Jhat(\Hhat)=j\}} \cdot  \mathds{1}_{\{\Hhat < b_{\Lcal(1)}\}}\Bigg]
\quad
\text{and}
\quad
B =  \Ehat_i\Bigg[ \left|\frac{\Xhat(\Hhat)}{\Xhat(\Hhat^-)}\right|^q \mathds{1}_{\{\Jhat(\Hhat)=j\}} \cdot  \mathds{1}_{\{\Hhat = b_{\Lcal(1)}\}}\Bigg].
\]
Performing the change of measure, we first get
\[
A = \Eb_i\Bigg[\sum_{H<t<\zeta} \frac{v_{J_{\Delta}(t)}}{v_i} |\Delta X(t)|^{\omega} \left|\frac{X(H)}{X(H^-)}\right|^q \mathds{1}_{\{J(H)=j\}}\Bigg],
\]
with $J(H)$ being the type to which $J$ first jumps. We now apply the Markov property at time $H$ and self-similarity of $X$:
\[
A = \Eb_i\left[ \left|\frac{X(H)}{X(H^-)}\right|^q \mathds{1}_{\{J(H)=j\}} |X(H)|^{\omega} \Eb_{j}\left[\sum_{0<t<\zeta} \frac{v_{J_{\Delta}(t)}}{v_i} |\Delta X(t)|^{\omega} \right]\right].
\]
By admissibility of $((v_i)_{i\in \Ical}, \omega)$, $\Eb_{j}\left[\sum_{0<t<\zeta} v_{J_{\Delta}(t)} |\Delta X(t)|^{\omega} \right]=v_j$. Hence, 
\begin{align*}
A 
&= \frac{v_j}{v_i} \Eb_i\left[ \left|\frac{X(H)}{X(H^-)}\right|^{q+\omega} |X(H^-)|^{\omega}  \mathds{1}_{\{J(H)=j\}}  \right] \\
&= \frac{v_j}{v_i} \mathtt{E}_i\left[ \mathrm{e}^{(q+\omega)U_{i,j}} \mathrm{e}^{\omega \xi_i(\rho_i^-)} \mathds{1}_{\{\Theta(\rho_i)=j\}}\right],
\end{align*}
and by independence we obtain
\begin{align*}
A 
&= \frac{v_j}{v_i} \frac{q_{i,j}}{(-q_{i,i})} G_{i,j}(q+\omega) \int_0^{\infty} \mathrm{d}s\, (-q_{i,i}) \mathrm{e}^{q_{i,i}s}\, \mathtt{E}_i[\mathrm{e}^{\omega\xi_i(s)}] \\
&= -\frac{v_j}{v_i} q_{i,j} \frac{G_{i,j}(q+\omega)}{\psi_i(\omega)+q_{i,i}}.
\end{align*}
Besides, 
\begin{align*}
B 
&=
\Eb_i\Bigg[ \sum_{0<t<\zeta} \mathds{1}_{\{t\le H\}} \mathds{1}_{\{J_{\Delta}(t)=j\}} \frac{v_{J_{\Delta}(t)}}{v_i}|\Delta X(t)|^{\omega}\left|\frac{\Delta X(t)}{X(t^-)}\right|^q  \Bigg] \\
&=
\frac{v_{j}}{v_i}\mathtt{E}_i\Bigg[ \sum_{0<t<\rho_i} \mathrm{e}^{\omega \xi_i(t^-)} \Big(1-\mathrm{e}^{\Delta \xi_i(t)}\Big)^{q+\omega} \mathds{1}_{\{\iota_{\Delta}(t)=j\}} \Bigg] + \frac{v_j}{v_i} \mathtt{E}_i\Big[ \mathrm{e}^{\omega \xi_i(\rho_i^-)} \Big(1-\mathrm{e}^{U_{i,\Theta(\rho_i)}}\Big)^{\omega}\mathds{1}_{\{\iota_{\Delta}(\rho_i)=j\}} \Big].
\end{align*}
By the compensation formula as in \eqref{eq: cumulant A}, we finally get
\[
B 
=
-\frac{v_{j}}{v_i} \frac{1}{q_{i,i}+ \psi_i(\omega)} \int_{-\infty}^0 \Pi_{i,j}(\mathrm{d}x) (1-\mathrm{e}^{x})^{q+\omega}.
\]
We can now come back to \eqref{eq:geo sum} and deduce that 
\[
\Ehat_i\left[ \left|\frac{\Xhat(\Hhat)}{\Xhat(\Hhat^-)}\right|^q \mathds{1}_{\{\Jhat(\Hhat)=j\}}\right]
=
-\frac{\displaystyle\frac{v_{j}}{v_i}\left(\displaystyle\int_{-\infty}^0 \Pi_{i,j}(\mathrm{d}x) (1-\mathrm{e}^{x})^{q+\omega} +  q_{i,j} G_{i,j}(q+\omega)\right)}{(\psi_i(\omega)+q_{i,i})+ \displaystyle\int_{-\infty}^0 \Pi_{i,i}(\mathrm{d}x) (1-\mathrm{e}^x)^{\omega}}.
\]
Recalling \eqref{eq: kappa}, we are left with
\[
\Ehat_i\left[ \left|\frac{\Xhat(\Hhat)}{\Xhat(\Hhat^-)}\right|^q \mathds{1}_{\{\Jhat(\Hhat)=j\}}\right]
=
-\frac{\displaystyle\frac{v_{j}}{v_i}\left(\displaystyle\int_{-\infty}^0 \Pi_{i,j}(\mathrm{d}x) (1-\mathrm{e}^{x})^{q+\omega} +  q_{i,j} G_{i,j}(q+\omega)\right)}{\kappa_i(\omega)}.
\]
Note that, because $\Kcal_i(\omega)=0$,  we get
\[
\kappa_i(\omega) = -\sum_{j\in\Ical\setminus\{i\}}\frac{v_{j}}{v_i}\left(\int_{-\infty}^0 \Pi_{i,j}(\mathrm{d}x) |\mathrm{e}^{x}-1|^{\omega} +  q_{i,j} G_{i,j}(\omega)\right),
\]
which upon taking $q=0$ already gives $\qhat_{i,j}$ up to a multiplicative constant.
\medskip
\item[$\rhd$] \textsc{Determining the exponential parameters $\qhat_{i,j}$.}
Recall that we have assumed homogeneity. Thus, for $q\ge 0$, we wish to compute, 
\[
\Ehat_i\left[\mathrm{e}^{q\Hhat}\mathds{1}_{\{\Jhat(\Hhat)=j\}}\right]
=
\sum_{k=0}^{\infty} \underbrace{\Ehat_i\left[ \mathrm{e}^{q\Hhat}\mathds{1}_{\{\Jhat(\Hhat)=j\}} \mathds{1}_{\{b_{\Lcal(k)}< \Hhat \le b_{\Lcal(k+1)}\}}\right]}_{:=a_k'}.
\]
Again, using the definition of $\Phat_i$ and the Markov property just as we did with $a_k$, we end up with
\[
a_k' = r a_{k-1}', \quad k\ge 1,
\]
where 
\begin{align}
r 
&= \mathtt{E}_i\Bigg[\sum_{0<s\le\rho_i} \mathrm{e}^{qs} \mathrm{e}^{\omega \xi(s^-)} \Big(1-\mathrm{e}^{\Delta \xi(s)}\Big)^{\omega} \mathds{1}_{\{\iota_{\Delta}(s)=i\}} \Bigg] \notag \\
&= \mathtt{E}_i\Bigg[\sum_{0<s<\rho_i} \mathrm{e}^{qs} \mathrm{e}^{\omega \xi_i(s^-)} \Big(1-\mathrm{e}^{\Delta \xi_i(s)}\Big)^{\omega} \mathds{1}_{\{\iota_{\Delta}(s)=i\}}\Bigg] + \mathtt{E}_i\left[ \mathrm{e}^{q\rho_i} \mathrm{e}^{\omega \xi_i(\rho_i^-)} \Big(1-\mathrm{e}^{U_{i,\Theta(\rho_i)}}\Big)^{\omega} \mathds{1}_{\{\iota_{\Delta}(\rho_i)=i\}}\right]. \label{eq: spine proof r}
\end{align}
Then, we use the compensation formula and we obtain that the first term is
\begin{align}
\mathtt{E}_i\Bigg[\sum_{0<s<\rho_i} \mathrm{e}^{qs} \mathrm{e}^{\omega \xi_i(s^-)} \Big(1-\mathrm{e}^{\Delta \xi_i(s)}\Big)^{\omega} \mathds{1}_{\{\iota_{\Delta}(s)=i\}} \Bigg] &= \int_0^{\infty} \mathrm{d}s\, \mathrm{e}^{(q+q_{i,i})s}\mathrm{e}^{\psi_i(\omega)s} \int_{-\infty}^0 \Lambda_{i}^{(i)}(\mathrm{d}x)(1-\mathrm{e}^x)^{\omega} \notag \\
&= - \frac{1}{q+q_{i,i}+\psi_i(\omega)} \int_{-\infty}^0 \Lambda_{i}^{(i)}(\mathrm{d}x)(1-\mathrm{e}^x)^{\omega}. \label{eq: spine proof first term}
\end{align}
By independence, the second term of \eqref{eq: spine proof r} is 
\begin{align}
\mathtt{E}_i\Bigg[ \mathrm{e}^{q\rho_i} \mathrm{e}^{\omega \xi_i(\rho_i^-)} &\Big(1-\mathrm{e}^{U_{i,\Theta(\rho_i)}}\Big)^{\omega} \mathds{1}_{\{\iota_{\Delta}(\zeta)=i\}} \Bigg] \notag \\
&=
\sum_{k\in\Ical\setminus\{i\}} \frac{q_{i,k}}{(-q_{i,i})}\int_0^{\infty} \mathrm{d}s\, (-q_{i,i})\mathrm{e}^{(q+q_{i,i})s}\mathrm{e}^{\psi_i(\omega)s} \int_{-\infty}^0 \Lambda^{(i)}_{U_{i,k}}(\mathrm{d}x)(1-\mathrm{e}^x)^{\omega} \notag \\
&= 
- \frac{1}{q+q_{i,i}+\psi_i(\omega)} \sum_{k\in\Ical\setminus\{i\}} q_{i,k}\int_{-\infty}^0 \Lambda_{U_{i,k}}^{(i)}(\mathrm{d}x)(1-\mathrm{e}^x)^{\omega}. \label{eq: spine proof second term}
\end{align}
Thanks to \eqref{eq: spine proof first term} and \eqref{eq: spine proof second term}, equation \eqref{eq: spine proof r} boils down to
\[
r = - \frac{1}{q+q_{i,i}+\psi_i(\omega)} \int_{-\infty}^0 \Pi_{i, i}(\mathrm{d}x)(1-\mathrm{e}^x)^{\omega}.
\]
On the other hand,
\[
a_0' = \Ehat_i\left[ \mathrm{e}^{q\Hhat} \mathds{1}_{\{\Jhat(\Hhat)=j\}} \mathds{1}_{\{\Hhat \le b_{\Lcal(1)}\}}\right],
\]
and we may split the indicator over $\{\Hhat<b_{\Lcal(1)}\}$ and $\{\Hhat=b_{\Lcal(1)}\}$. We therefore get $a_0'=A'+B'$, where
\[
A'
=
\Eb_i \Bigg[ \sum_{0<t<\zeta} \mathrm{e}^{qH} \mathds{1}_{\{J(H)=j\}} \mathds{1}_{\{H<t\}} \frac{v_{J_{\Delta}(t)}}{v_i}\left|\Delta X(t)\right|^{\omega}\Bigg], 
\]
and
\[
B'=\mathtt{E}_i \Bigg[ \sum_{0<t\le \rho_i} \mathrm{e}^{qt} \frac{v_{j}}{v_i}\mathrm{e}^{\omega\xi(t^-)}\Big(1-\mathrm{e}^{\Delta \xi(t)}\Big)^{\omega} \mathds{1}_{\{\iota_{\Delta}(t)=j\}}\Bigg].    
\]
First of all, $B'$ can be rewritten as follows
\[
B'
=
\frac{v_{j}}{v_i} \mathtt{E}_i \Bigg[ \sum_{0<t<\rho_i} \mathrm{e}^{qt} \mathrm{e}^{\omega\xi(t^-)}\Big(1-\mathrm{e}^{\Delta \xi(t)}\Big)^{\omega} \mathds{1}_{\{\iota_{\Delta}(t)=j\}}\Bigg] + \frac{v_{j}}{v_i}\mathtt{E}_i \Big[\mathrm{e}^{q\rho_i} \mathrm{e}^{\omega\xi_i(\rho_i^-)}\Big(1-\mathrm{e}^{U_{i,\Theta(\rho_i)}}\Big)^{\omega} \mathds{1}_{\{\iota_{\Delta}(\rho_i)=j\}}\Big].
\]
Continuing along the lines of \eqref{eq: spine proof first term}, \eqref{eq: spine proof second term}, we eventually get to
\[
B'
=
- \frac{v_j}{v_i} \frac{1}{q+q_{i,i}+\psi_i(\omega)} \int_{-\infty}^0 \Pi_{i,j}(\mathrm{d}x) (1-\mathrm{e}^x)^{\omega}.
\]
Moreover, by using the Markov property at time $H$, self-similarity of $X$, and by admissibility of $((v_i)_{i\in\Ical},\omega)$, we have
\begin{align*}
A'
&=
\Eb_i \Bigg[ \mathrm{e}^{qH} \mathds{1}_{\{J(H)=j\}} \Eb_{X(H), j} \Bigg[ \sum_{0<t<\zeta} \frac{v_{J_{\Delta}(t)}}{v_i}\left|\Delta X(t)\right|^{\omega}\Bigg]\Bigg] \\
&=
\Eb_i \Bigg[ \mathrm{e}^{qH} \mathds{1}_{\{J(H)=j\}} |X(H)|^{\omega}\Eb_{j} \Bigg[ \sum_{0<t<\zeta} \frac{v_{J_{\Delta}(t)}}{v_i}\left|\Delta X(t)\right|^{\omega}\Bigg]\Bigg] \\
&=
\frac{v_j}{v_i}\Eb_i \left[ \mathrm{e}^{qH} \mathds{1}_{\{J(H)=j\}} |X(H)|^{\omega}\right]. 
\end{align*}
Now, on the event that $J(H)=j$, we have $X(H)=\mathrm{e}^{\xi_i(\rho_i^-)+U_{i,j}}$ under $\mathbb{P}_i$. This entails
\begin{align*}
A'
&=
-\frac{v_j}{v_i} \frac{q_{i,j}}{q_{i,i}}\mathtt{E}_i[\mathrm{e}^{(q+\psi_i(\omega))\rho_i}] G_{i,j}(\omega) \\
&=
-\frac{v_j}{v_i}  \frac{q_{i,j} G_{i,j}(\omega)}{q+q_{i,i}+\psi_i(\omega)}.
\end{align*}
Therefore,
\[
a_0
=
- \frac{\displaystyle\frac{v_j}{v_i} q_{i,j} G_{i,j}(\omega) + \displaystyle\frac{v_j}{v_i} \displaystyle\int_{-\infty}^0\Pi_{i,j}(\mathrm{d}x) (1-\mathrm{e}^x)^{\omega}}{q+q_{i,i}+\psi_i(\omega)}.
\]
We finally conclude that
\[
\Ehat_i\left[\mathrm{e}^{q\Hhat}\mathds{1}_{\{\Jhat(\Hhat)=j\}}\right]
=
\frac{a_0}{1-r}
=
- \frac{\displaystyle\frac{v_j}{v_i} q_{i,j} G_{i,j}(\omega) + \displaystyle\frac{v_j}{v_i} \displaystyle\int_{-\infty}^0\Pi_{i,j}(\mathrm{d}x) (1-\mathrm{e}^x)^{\omega}}{q+q_{i,i}+\psi_i(\omega)+\displaystyle\int_{-\infty}^0 \Pi_{i,i}(\mathrm{d}x)(1-\mathrm{e}^x)^{\omega}}.
\]
This shows that, for all $i,j\in\Ical$, with $j\ne i$, the jump time of the chain $\Jhat$ from state $i$ to state $j$ is an exponential random variable, with parameter
\[
\qhat_{i,j} = \frac{v_j}{v_i} \left(q_{i,j} G_{i,j}(\omega) +  \int_{-\infty}^0\Pi_{i,j}(\mathrm{d}x) |1-\mathrm{e}^x|^{\omega}\right).
\]
\item[$\rhd$] \textsc{The matrix exponent.} The previous calculations determine  $\Fhat(q) = (\Fhat_{i,j}(q))_{i,j\in\Ical}$, the matrix exponent of the spine as the matrix with entries:
\[ 
\forall i\in \Ical, \quad \Fhat_{i,i}(q)
=
\kappa_i(\omega+q)
\]
and 
\[
\forall i,j\in \Ical, i\neq j, \quad \Fhat_{i,j}(q)
=
\frac{v_{j}}{v_i}\left(\int_{-\infty}^0 \Pi_{i,j}(\mathrm{d}x) (1-\mathrm{e}^{x})^{q+\omega} +  q_{i,j} G_{i,j}(q+\omega)\right).
\]
\end{itemize}

\noindent \textbf{Proof of the second assertion.} We finally prove the second assertion of Theorem \ref{thm:spine multitype} directly in the general setting, by mimicking \cite{DS}. We shall only prove the statement for the first generation (this is then easily extended using the branching property). Let $f_1, f_2$ be nonnegative measurable functionals respectively on the space of \textit{càdlàg} trajectories and sequences of types in $\Ical$, and $g_k$, $k\ge 1$, be nonnegative measurable functionals on the space of multiset--valued paths. For $t>0$, denote by $(\mathfrak{X}_k(t), j\ge 1)$ the sequence consisting of the value of $\Xcal_{\varnothing}(t)$, and all those jumps of $\Xcal_{\varnothing}$ that happened strictly before time $t$, ranked in descending order of their absolute value, and write $(\mathfrak{j}_k(t), k\ge 1)$ for the corresponding types. Our goal is to show that
\begin{multline*}
\Ehat_{x,i} \bigg[f_1(\Xcal_{\varnothing}(s), 0\le s\le b_{\Lcal(1)}) f_2(\mathfrak{j}_k(b_{\Lcal(1)}), k\ge 1) \prod_{k\ge 1} g_k(\Xbfhat_{0,k})\bigg] \\
=
\Ehat_{x,i} \bigg[f_1(\Xcal_{\varnothing}(s), 0\le s\le b_{\Lcal(1)}) f_2(\mathfrak{j}_k(b_{\Lcal(1)}), k\ge 1)  \prod_{k\ge 1} \Ebf_{\mathfrak{X}_k(b_{\Lcal(1)}),\mathfrak{j}_k(b_{\Lcal(1)})} \left[g_k(\Xbf)\right]\bigg].
\end{multline*}
But,
\begin{multline*}
\Ehat_{x,i} \bigg[f_1(\Xcal_{\varnothing}(s), 0\le s\le b_{\Lcal(1)}) f_2(\mathfrak{j}_k(b_{\Lcal(1)}), k\ge 1) \prod_{k\ge 1} g_k(\Xbfhat_{0,k})\bigg]  \\
=
\Ecal_{x,i} \bigg[\sum_{t>0} \frac{v_{J_{\Delta}(t)}}{v_{i}} |\Delta \Xcal_{\varnothing}(t)|^{\omega} f_1(\Xcal_{\varnothing}(s), 0\le s\le t) f_2(\mathfrak{j}_k(t), k\ge 1) \prod_{k\ge 1} g_k(\Xbfhat_{0,k})\bigg],
\end{multline*}
and the definition of the $\Xbfhat_{0,k}$ together with the branching property under $\Pcal_{x,i}$ give
\begin{multline*}
\Ehat_{x,i} \bigg[f_1(\Xcal_{\varnothing}(s), 0\le s\le b_{\Lcal(1)}) f_2(\mathfrak{j}_k(b_{\Lcal(1)}), k\ge 1) \prod_{k\ge 1} g_k(\Xbfhat_{0,k})\bigg] \\
=
\Ecal_{x,i} \left[\sum_{t>0} \frac{v_{J_{\Delta}(t)}}{v_{i}} |\Delta \Xcal_{\varnothing}(t)|^{\omega} f_1(\Xcal_{\varnothing}(s), 0\le s\le t) f_2(\mathfrak{j}_k(t), k\ge 1) \prod_{k\ge 1} \Ebf_{\mathfrak{X}_k(t), \mathfrak{j}_k(t)} \left[g_k(\Xbf)\right]\right].
\end{multline*}
Applying the change of measure backwards, we get the desired identity. Therefore Theorem \ref{thm:spine multitype} is proved.

\section{Appendix: implicit renewal theory on multitype branching trees} \label{sec: appendix}
In the case when there is no type, the study of some variants of $\Mcal^-(\infty)$ has arised in a variety of contexts and dates back to Mandelbrot \cite{Man} and then Kahane and Peyrière \cite{Kah-Pey}. In \cite{Liu}, Liu presents a unified account on these \emph{multiplicative cascades}, through a distributional equation. 

This appendix is concerned with a multitype version of Goldie's implicit renewal theorem \cite{Gol}, with applications to two random equations. The first set of equations that we consider in \cref{sec: tail behaviour affine} is:
\begin{equation}\label{eq: random affine equation}
R \overset{\Lcal}{=} \frac{v_J}{v_i} AR^{(J)}+B,
\end{equation}
under $\Pb_i$ for all $i\in\Ical$, where $(v_j,j\in\Ical)$ is a (deterministic) positive vector, $A$ and $B$ are positive random variables, $J$ is a random variable in $\Ical$, $(R^{(j)}, j\in\Ical)$ denotes a random vector, independent of $(A,B,J)$, whose entries are distributed according to $R$ under $\Pb_j, j\in \Ical$. The second equation of interest (\cref{sec: tail behaviour smoothing}) is the \emph{multitype} linear homogeneous recursion or smoothing transform
\begin{equation} \label{eq: smoothing transform}
R \overset{\Lcal}{=} \sum_{k=1}^{\infty} \frac{v_{J_k}}{v_i} C_k R_k^{(J_k)},
\end{equation}
under $\Pb_i$ for all $i\in\Ical$, where $(v_j,j\in\Ical)$ is a (deterministic) positive vector, $(C_k, k\ge 1)$ is a nonnegative random vector, $(J_k,k\ge 1)$ is a sequence of types in $\Ical$,  $(R_k^{(j)}, k\ge 1)$ are nonnegative i.i.d. variables with the same law as $R$, under $\Pb_j$, for each $j\in\Ical$, independent of $(J_k, C_k, k\ge 1)$. We denote by  $Q_i$ the law of  $R$ under $\Pb_i$. Our goal is to generalise to some extent the results of \cite{JO-C}. We follow closely their approach.
Although the statements we give are often tailored for our purposes, we believe that this approach could be extended in many different ways: random number of offspring, removing the moment condition in \cref{thm: tail behaviour smoothing} in the spirit of \cite[Theorem 4.2]{JO-C}, treating the $\alpha\le 1$ case in \cref{thm: tail behaviour smoothing}, etc.

We construct the following multitype branching tree indexed by $\Ub$, similarly to the growth-fragmentation tree. Denote by $P_i$, $i\in\Ical$, the law of a vector $(J_{k}, C_k, k \ge 1)$ which will stand for the types and displacements of the offspring in the tree conditionally on the parent type being $i$. Let $i\in\Ical$, and construct the law $\Pcal_i$ of the multitype tree as follows. The root $\varnothing$ has type $i$ and marks $M_{\varnothing}:=(J_{k}, C_k, k \ge 1)$ distributed according to $P_i$. Independently of each another, each child $k$ of $\varnothing$ is then assigned the type $J_k$ and the displacement $C_k$, and comes itself with marks $M_{k}:=(J_{(k,l)}, C_{(k,l)}, l\ge 1)$ which, conditionally on $M_{\varnothing}$, is a sample of $P_{J_k}$. More generally, any $u\in \Ub$ with $|u|=n+1$ will have a type $J_u$ and a displacement $C_u$ from its parent, and is assigned marks $M_u$ which are distributed as $P_{J_u}$ conditionally on $\Gscr_n := \sigma(M_v, |v|\le n)$. To each $u\in\Ub$, we also associate its \emph{position} $X_u$ defined by recursion by
\begin{equation} \label{eq: def Xu Ru tree}
X_{\varnothing}=1, \quad \text{and} \quad X_{(u,k)} = C_{(u,k)} \cdot X_{u}, 
\end{equation}
for all $u\in \Ub, k\in \N$. We will also need independent copies $R^{(J_u)}_u$ with law $Q_{J_u}$ (we will sometimes omit the superscript $J_u$, writing $R_u = R^{(J_u)}_u$ for ease of notation). This completes the construction of the multitype branching tree $\Tcal$. We see equation \eqref{eq: smoothing transform} as a fixed point equation on the genealogies of this tree. 

\subsection{The implicit renewal theorem on $\Tcal$}
We describe and prove a version of Goldie's implicit renewal theorem \cite{Gol} on multitype branching trees. Our statement is largely inspired by \cite[Theorem 3.1]{JO-C}, from which our proof is a simple adaptation. A key role will be played by the analogue of the matrix exponent in \eqref{eq: matrix m(q)}, defined here as the matrix with entries
\begin{equation} \label{eq: matrix (C,J)}
m_{i,j}(q) = E_i\left[\sum_{k\ge 1}  C_k^q \, \mathds{1}_{\{J_k=j\}}\right].
\end{equation}
Again, by Perron-Frobenius theory, whenever $m(q)$ is finite and irreducible (note that irreducibility does not depend on $q$), we can define its leading eigenvalue to be $\mathrm{e}^{\lambda(q)}$. We shall then write $(v_i(q), i\in\Ical)$ for an associated positive eigenvector (we will often omit the dependence on $q$). 

\begin{Thm} \label{thm: renewal multitype tree}
Under $\Pb_i$, for any $i\in\Ical$, let $(J_k, C_k, k\ge 1)$ be distributed as $P_i$, $(R_k^{(j)}, k\ge 1)$ be an independent sequence of nonnegative i.i.d. variables with law $Q_j$ for each $j\in\Ical$, and $R$ be a random variable with law $Q_i$. Assume that $P_i(C_k>0)>0$ and that the measure $P_i(\log C_k \in \mathrm{d}y, C_k>0)$ is non-lattice for some $k\ge 1$. We assume that there exist $0\le\gamma<\alpha$ such that 
\begin{itemize}
\item[(i)] $m(q)$ is finite and irreducible for $q$ in a domain containing $[\gamma, \alpha]$,
\item[(ii)] $\lambda(\alpha) =0$ and if $(v_j, j\in\Ical)$ denotes a positive eigenvector associated to $\lambda(\alpha)=0$, \[0<E_i\left[\sum_{k\ge 1}  \frac{v_{J_k}}{v_i}C_k^\alpha \log C_k  \right]<\infty,\]
\item[(iii)] $\Eb_j[R^{\beta}]<\infty,$ for all $0<\beta<\alpha$ and all $j\in\Ical$.
\end{itemize}
If, for every $i\in\Ical$,
\begin{equation} \label{eq: R integrability}
\int_0^{\infty} \bigg|\Pb_i(R>t) - \Eb_i\bigg[\sum_{k\ge 1} \mathds{1}_{\{\frac{v_{J_k}}{v_i} C_k R_k^{(J_k)} >t\}} \bigg] \bigg| t^{\alpha-1} \mathrm{d}t < \infty,
\end{equation}
then for all $i\in\Ical$, there exists a constant $a_i\ge 0$ such that 
\[
\Pb_i(R>t) \underset{t\rightarrow \infty}{\sim} a_i \cdot t^{-\alpha}.
\]
\end{Thm}

\noindent The proof of \cref{thm: renewal multitype tree} relies on the following two lemmas. For a measure $\mu$ on $\Ical \times \R$ and a function $f:\Ical\times\R \rightarrow \R$, we define (whenever we can) the convolution $f\ast\mu$ as 
\[
\forall x\in\R, \quad f\ast \mu(x) := \int_{\Ical\times\R} f(j, x-y) \mu(\mathrm{d}j,\mathrm{d}y).
\]
Similarly, for two measure-valued mappings $\mu:i\in\Ical\mapsto \mu_i(\mathrm{d}j,\mathrm{d}y)$ and $\nu: i\in\Ical\mapsto \nu_i(\mathrm{d}j,\mathrm{d}y)$ on $\Ical\times\R$, let $\mu\ast\nu_i$, $i\in\Ical$, be the measure on $\Ical\times\R$:
\begin{equation} \label{eq: convolution measures}
(\mu\ast\nu)_i(\mathrm{d}j,\mathrm{d}x) := \int_{\Ical\times\R} \mu_i(\mathrm{d}J,\mathrm{dY}) \nu_J(\mathrm{d}j,\mathrm{d}x-Y).
\end{equation}
We will mostly restrict to the case when $\nu=\mu$, and write $(\mu_i)^{\ast n}$ for the measure $(\mu\ast\cdots\ast\mu)_i$ ($n$ times). This corresponds simply to the distribution of a multitype random walk $S_n := \sum_{i=1}^n Y_i$, where the distribution of $Y_i$ at each step depends on the type in $\Ical$.

\begin{Lem} \label{lem: renewal measure}
Recall the notation regarding the tree $\Tcal$, and write $Y_u = \log X_u$. Let $\alpha>0$ and for $m\in\N$ and $i\in\Ical$, define the measure on $\Ical\times\R$
\begin{equation} \label{eq: mu_m}
\mu^{(i)}_m(\mathrm{d}j, \mathrm{d}y) := \Eb_i\left[\sum_{|u|=m} \frac{v_j}{v_i}\mathrm{e}^{\alpha y} \mathds{1}_{\{J_u=j, Y_u \in \mathrm{d}y\}} \right],
\end{equation}
and $\eta^{(i)} = \mu^{(i)}_1$. Assume that $P_i(C_k>0)>0$ and that the measure $P_i(\log C_k \in \mathrm{d}y, C_k>0)$ is non-lattice for some $k\ge 1$. Moreover, suppose that assumption $(ii)$ of \cref{thm: renewal multitype tree} holds. Then, for all $i\in\Ical$, $\eta^{(i)}$ is a non-lattice probability measure with positive mean, and for all $m\in\N$, $\mu^{(i)}_m = \big(\eta^{(i)}\big)^{\ast m}$ is the $m$ times convolution of $\eta^{(i)}$ in the sense of \eqref{eq: convolution measures}.
\end{Lem}
\begin{proof} 
The fact that $\eta^{(i)}$ is a non-lattice probability measure is clear from our assumptions. Furthermore, by assumption $(ii)$, $\eta^{(i)}$ has positive mean
\[
\Eb_i\left[\sum_{k\ge 1}  \frac{v_{J_k}}{v_i} C_k^{\alpha}\log C_k\right]>0.
\]
Finally, by conditioning on $\Gscr_m$,
\begin{align*}
\mu^{(i)}_{m+1}(\mathrm{d}j, \mathrm{d}y) 
&= \Eb_i\left[\sum_{|u|=m} \frac{v_{J_u}}{v_i}\mathrm{e}^{\alpha Y_u} \sum_{k\ge 1}\frac{v_j}{v_{J_u}}\mathrm{e}^{\alpha (y-Y_u)} \mathds{1}_{\{J_{uk}=j, \log C_{k} \in \mathrm{d}y - Y_u\}} \right] \\
&=  \Eb_i\left[\sum_{|u|=m} \frac{v_{J_u}}{v_i}\mathrm{e}^{\alpha Y_u} \eta^{(J_u)}(\mathrm{d}j, \mathrm{d}y-Y_u)\right] \\
&= \mu_{m}^{(i)} \ast \eta^{(i)}(\mathrm{d}j,\mathrm{d}y).
\end{align*}
The last statement of \cref{lem: renewal measure} follows.
\end{proof}

Recall the notation $X_u$ and $R_u$ from \eqref{eq: def Xu Ru tree}.
\begin{Lem} \label{lem: delta n term}
Assume that assumptions $(i)$, $(ii)$ and $(iii)$ from \cref{thm: renewal multitype tree} hold, and denote by $(v_j, j\in\Ical)$ a positive eigenvector associated to $\lambda(\alpha)=0$. For all $i\in\Ical$, $n\in\N$ and $t\in\R$, set
\[
\delta^{(i)}_n(t) = \mathrm{e}^{\alpha t}\Eb_i\left[\sum_{|u|=n}  \mathds{1}_{\{\frac{v_{J_u}}{v_i}X_u R_u > \mathrm{e}^t\}}\right], 
\]
and let
\[
\tilde{\delta}^{(i)}_n(t) := \int_{-\infty}^t \mathrm{e}^{-(t-s)} \delta_n^{(i)}(s) \mathrm{d}s, \quad t\ge \R,
\]
be its Laplace transform. Then for all $i\in\Ical$ and $t\in\R$, $\tilde{\delta}^{(i)}_n(t) \to 0$ as $n\to\infty$.
\end{Lem}
\begin{proof}
We notice the following facts from assumptions $(i)$ and $(ii)$: for all $i\in\Ical$, $\Eb_i\Big[ \sum_{k\ge 1} \frac{v_k}{v_i} C_k^{\alpha}\Big] = 1$, $\Eb_i\Big[ \sum_{k\ge 1} \frac{v_k}{v_i} C_k^{\gamma}\Big] <\infty$. Moreover, as a byproduct of \cref{lem: renewal measure}, the measure $\nu^{(i)}$ has positive mean
\[
\Eb_i\left[\sum_{k\ge 1}\frac{v_{J_k}}{v_i} C_k^{\alpha} \log C_k\right] >0.
\]
By convexity of $\beta \mapsto \Eb_i\Big[ \sum_{k\ge 1} \frac{v_k}{v_i} C_k^{\beta}\Big]$, these three facts entail that there exists $\beta\in(0,\alpha)$ such that, for all $i\in\Ical$, $\Eb_i\Big[ \sum_{k\ge 1} \frac{v_k}{v_i} C_k^{\beta}\Big]<1$. Call $K:=\max_{i\in\Ical}\Eb_i\Big[ \sum_{k\ge 1} \frac{v_k}{v_i} C_k^{\beta}\Big]<1$. Now observe the following: for $n\in\N$ and $t\in\R$,
\begin{align*}
\tilde{\delta}^{(i)}_n(t)
&=
\int_{-\infty}^t \Eb_i\left[\sum_{|u|=n}  \mathds{1}_{\{\frac{v_{J_u}}{v_i}X_u R_u > \mathrm{e}^s\}}\right] \mathrm{e}^{-(t-s)}\mathrm{e}^{\alpha s}\mathrm{d}s \\
&\le \int_{-\infty}^t \Eb_i\left[\sum_{|u|=n}  \mathds{1}_{\{\frac{v_{J_u}}{v_i}X_u R_u > \mathrm{e}^s\}}\right] \mathrm{e}^{\alpha s}\mathrm{d}s \\
&\le \mathrm{e}^{(\alpha-\beta)t} \Eb_i\left[\sum_{|u|=n}  \int_{-\infty}^t \mathds{1}_{\{\frac{v_{J_u}}{v_i}X_u R_u > \mathrm{e}^s\}}\mathrm{e}^{\beta s}\mathrm{d}s\right] \\
& \le \frac{\mathrm{e}^{(\alpha-\beta)t}}{\beta} \Eb_i\left[\sum_{|u|=n} \left(\frac{v_{J_u}}{v_i} X_u R_u\right)^{\beta}\right],
\end{align*}
by direct integration. Set $c_1 := \max_{i,j\in\Ical} \left(\frac{v_j}{v_i}\right)^{\beta-1}>0$. By assumption $(iii)$, $c_2:= \max_{i\in\Ical} \Eb_i[R^{\beta}] <\infty$, so by conditioning on the type $J_u$ of $u$ and using independence, we end up with
\begin{equation} \label{eq: proof delta n}
\tilde{\delta}^{(i)}_n(t)
\le
\frac{c_1 c_2 \mathrm{e}^{(\alpha-\beta)t}}{\beta} \cdot \Eb_i\left[\sum_{|u|=n} \frac{v_{J_u}}{v_i} X_u^{\beta}\right].
\end{equation}
The expectation in \eqref{eq: proof delta n} can now be split using the branching property:
\[
\Eb_i\left[\sum_{|u|=n} \frac{v_{J_u}}{v_i} X_u^{\beta}\right]
=
\Eb_i\left[\sum_{|u|=n-1} \frac{v_{J_u}}{v_i} X_u^{\beta}\cdot\Eb_{J_u}\left[\sum_{k\ge 1}\frac{v_k}{v_{J_u}} C_k^{\beta}  \right]\right]
\le
K \cdot \Eb_i\left[\sum_{|u|=n-1} \frac{v_{J_u}}{v_i} X_u^{\beta}\right],
\]
by definition of $K$. Iterating this argument yields that $\tilde{\delta}^{(i)}_n(t) \to 0$ (actually faster than $K^n$) as $n\to\infty$.
\end{proof}

\begin{proof}[Proof of \cref{thm: renewal multitype tree}]
We write $X_u = \mathrm{e}^{Y_u}$, $u\in\Ub$. Let $i\in\Ical$. In view of applying renewal theory, considering telescoping sums over the tree $\Tcal$ provides, for all $t\in\R$ and $n\in\N$,
\begin{align}
&\Pb_i(R>\mathrm{e}^t) \notag \\
&= \sum_{m=0}^{n-1} \Eb_i\left[\sum_{|u|=m}  \mathds{1}_{\{\frac{v_{J_u}}{v_i} X_u R_u > \mathrm{e}^{t}\}} - \sum_{|w|=m+1}  \mathds{1}_{\{\frac{v_{J_w}}{v_i} X_w R_w > \mathrm{e}^{t}\}}\right] + \Eb_i\left[\sum_{|u|=n}  \mathds{1}_{\{\frac{v_{J_u}}{v_i} X_u R_u > \mathrm{e}^{t}\}}\right]
 \notag
 \\ 
&=
\sum_{m=0}^{n-1} \Eb_i\left[\sum_{|u|=m}  \Big(\mathds{1}_{\{\frac{v_{J_u}}{v_i} R_u > \mathrm{e}^{t-Y_u}\}} - \sum_{k\ge 1}  \mathds{1}_{\{\frac{v_{J_{(u,k)}}}{v_i}C_{(u,k)} R_{(u,k)} > \mathrm{e}^{t-Y_u}\}}\Big) \right] + \Eb_i\left[\sum_{|u|=n}  \mathds{1}_{\{\frac{v_{J_u}}{v_i} X_u R_u > \mathrm{e}^{t}\}}\right].  \label{eq: renewal telescoping sums}
\end{align}
Conditioning on $\Gscr_m$, $0\le m \le n-1$, each node $u$ in the first term corresponds to $\frac{v_{J_u}}{v_i}\mathrm{e}^{\alpha(Y_u-t)}g(J_u, t-Y_u)$, where
\[
g(j, y) = \frac{v_i}{v_j}\mathrm{e}^{\alpha y} \left( \Pb_j\Big(R > \frac{v_i}{v_j} \mathrm{e}^{y}\Big) - \Eb_j\left[\sum_{k\ge 1}  \mathds{1}_{\{C_{k} R^{(J_k)}_{k} > \frac{v_i}{v_{J_{k}}}\mathrm{e}^{y}\}}\right] \right)
\]
Set also 
\[
\delta^{(i)}_n(t) = \mathrm{e}^{\alpha t}\Eb_i\left[\sum_{|u|=n}  \mathds{1}_{\{\frac{v_{J_u}}{v_i}X_u R_u > \mathrm{e}^t\}}\right].
\]
Hence, multiplying \eqref{eq: renewal telescoping sums} by $\mathrm{e}^{\alpha t}$, we get
\[
\mathrm{e}^{\alpha t} \Pb_i(R>\mathrm{e}^t) 
=
\sum_{m=0}^{n-1} \Eb_i\left[\sum_{|u|=m} \frac{v_{J_u}}{v_i}\mathrm{e}^{\alpha Y_u}g(J_u, t-Y_u) \right]+ \delta^{(i)}_n(t).
\]
Recall \eqref{eq:  mu_m} and set $\nu^{(i)}_n := \sum_{m=0}^n \mu^{(i)}_m$, so that the previous equation becomes
\begin{equation} \label{eq: P(R>t)}
\mathrm{e}^{\alpha t} \Pb_i(R>\mathrm{e}^t) 
=
g\ast \nu^{(i)}_{n-1}(t)+ \delta^{(i)}_n(t).
\end{equation}
Now let $r^{(i)}(t) := \mathrm{e}^{\alpha t} \Pb_i(R>\mathrm{e}^t),$ for  $t\in\R$, and define the operator $f\mapsto \tilde{f}$ by
\[
\tilde{f}(t) := \int_{-\infty}^t \mathrm{e}^{-(t-s)} f(s) \mathrm{d}s, \quad t\ge \R.
\]
Applying this operator to \eqref{eq: P(R>t)}, we obtain for all $t\in\R$,
\begin{equation} \label{eq: rtilde finite n}
\tilde{r}^{(i)}(t) = \tilde{g}\ast\nu^{(i)}_{n-1}(t)+ \tilde{\delta}^{(i)}_n(t).
\end{equation}
We want to take $n\to \infty$ in the previous equality. Thanks to \cref{lem: delta n term}, the second term of \eqref{eq: rtilde finite n} vanishes when $n\to\infty$. We now deal with the second term using \cref{lem: renewal measure}. Indeed, the latter provides that the measure $\eta^{(i)}:=\mu_1^{(i)}$ is a non-lattice probability measure with positive mean, and that $\nu_n^{(i)}=\sum_{m=0}^n \big(\eta^{(i)}\big)^{\ast m}$. Let 
\[
\nu^{(i)} = \sum_{m=0}^\infty \big(\eta^{(i)}\big)^{\ast m},
\] 
its renewal measure. By Kesten's theorem \cite{Kes} (see also \cite{AMN}), for every function $G:\Ical\times\R\to\R$ which is measurable, bounded, continuous in the second variable, and directly Riemann integrable, the integral $G\ast\nu^{(i)}$ is finite. Now, by our assumption \eqref{eq: R integrability}, $g(j,\cdot)\in L^1$ for all $j\in\Ical$. Hence, by \cite[Lemma 9.2]{Gol}, $|\tilde{g}|$ is directly Riemann integrable, and therefore $|\tilde{g}|\ast\nu^{(i)}$ is finite. By Fubini's theorem, we get that $\tilde{g}\ast\nu^{(i)}_{n-1}(t) \to \tilde{g}\ast\nu^{(i)}(t)$ for all $t\in\R$. 

We come to the conclusion that as $n\to\infty$, \eqref{eq: rtilde finite n} boils down to $\tilde{r}^{(i)} = \tilde{g}\ast\nu^{(i)}$. Now take $t\to\infty$. By Kesten's renewal theorem \cite{Kes}, 
\[
\tilde{r}^{(i)} (t) = \mathrm{e}^{-t}\int_0^{\mathrm{e}^t} s^{\alpha} \Pb_i(R>s) \mathrm{d}s \underset{t\to\infty}{\longrightarrow} a_i,
\]
where $a_i$ is an (explicit) nonnegative constant. It remains to apply \cite[Lemma 9.3]{Gol} to obtain
\[
\Pb_i(R>t) \underset{t\rightarrow \infty}{\sim} a_i \cdot t^{-\alpha},
\]
which is \cref{thm: renewal multitype tree}.
\end{proof}


\subsection{Tail behaviour of the random affine equation} \label{sec: tail behaviour affine}
We present a first application of \cref{thm: renewal multitype tree}, where we study the solutions of \eqref{eq: random affine equation}. 
\begin{Thm} \label{thm: tail behaviour affine}
Under $\Pb_i$, for any $i\in\Ical$, let $(A,B,J)$ be distributed as $P_i$, $(R^{(j)}, j\in\Ical)$ be a sequence of nonnegative variables with law $Q_j$ for each $j\in\Ical$, independent of $(A,B,J)$, and $R$ be a random variable with law $Q_i$. Assume that $P_i(A>0)>0$ and that the measure $P_i(\log A \in \mathrm{d}y, A>0)$ is non-lattice. Similarly to \eqref{eq: matrix (C,J)}, for $q\ge 0$, let $m(q)$ be the matrix
\begin{equation} \label{eq: matrix m(q) affine}
m_{i,j}(q) = E_i\big[  A^q \cdot  \mathds{1}_{\{J=j\}}\big], \quad i,j\in\Ical,
\end{equation}
and $\lambda(q)$ its leading eigenvalue. We assume that there exist $\alpha>0$ and $0\le\gamma<\alpha$ such that 
\begin{itemize}
\item[(i)] The matrix $m(q)$ defined by \eqref{eq: matrix m(q) affine} is finite and irreducible for $q$ in a domain containing $[\gamma, \alpha]$,
\item[(ii)] $\lambda(\alpha) =0$ and if $(v_j, j\in\Ical)$ denotes a positive eigenvector associated to $\lambda(\alpha)=0$, \[0<E_i\left[  \frac{v_{J}}{v_i}A^\alpha \log A  \right]<\infty,\]
\item[(iii)] $\Eb_j[R^{\beta}]<\infty$ for all $0<\beta<\alpha$ and all $j\in\Ical$,
\item[(iv)] $E_i[A^{\alpha}]<\infty$ and $E_i[B^{\alpha}]<\infty$ for all $i\in\Ical$.
\end{itemize}
Assume that \eqref{eq: random affine equation} holds with the same notation.
Then for all $i\in\Ical$, there exists a constant $b_i\ge 0$ such that 
\[
\Pb_i(R>t) \underset{t\rightarrow \infty}{\sim} b_i\cdot t^{-\alpha}.
\]
\end{Thm}
\begin{proof}
We need to show that, under the assumptions in \cref{thm: tail behaviour affine}, \eqref{eq: R integrability} holds. Then \cref{thm: renewal multitype tree} would imply the desired assymptotics. In our setting, \eqref{eq: R integrability} rephrases as 
\[
\Eb_i\left[ R^\alpha - \left(\frac{v_{J}}{v_i} A R^{(J)}\right)^{\alpha}\right] <\infty.
\]
As $E_i[B^{\alpha}]<\infty$ by assumption, this is equivalent to
\[
\Eb_i\left[ R^\alpha - \left(\frac{v_{J}}{v_i} A R^{(J)}\right)^{\alpha} - B^\alpha\right] <\infty.
\]
Since $R$ solves \eqref{eq: random affine equation}, we can take $R = (v_{J}/v_i )A R^{(J)} + B$ under $\Pb_i$. Write $\alpha = p \gamma$, with $p:= \lceil \alpha \rceil$, and $\gamma\in[0,1]$. Then expanding the first sum, and using the inequality $(x+y)^{\gamma} \le x^\gamma + y^\gamma$ for $x,y\ge 0$, we get
\[
\Eb_i\left[ R^\alpha - \left(\frac{v_{J}}{v_i} A R^{(J)}\right)^{\alpha} - B^\alpha\right]
\le
\Eb_i\left[ \left(\sum_{k=1}^{p-1} \dbinom{p}{k} \left(\frac{v_{J}}{v_i} A R^{(J)}\right)^{k} B^{p-k} \right)^{\gamma}\right].
\]
We now condition on $(A, B, J)$. By Jensen's inequality ($\gamma\le 1$), and the independence of $(R_j, j\in\Ical)$ and $(A, B, J)$,
\[
\Eb_i\left[ R^\alpha - \left(\frac{v_{J}}{v_i} A R^{(J)}\right)^{\alpha} - B^\alpha\right]
\le
\Eb_i\left[ \left(\sum_{k=1}^{p-1} \dbinom{p}{k} \left(\frac{v_{J}}{v_i} A \right)^{k} \Eb_J[R^{k}] B^{p-k} \right)^{\gamma}\right].
\]
Using the monotonicity of $k\mapsto \Eb_j[R^k]^{1/k}$, we obtain for all $1\le k\le p-1$, 
\[ \Eb_J[R^{k}] \le  \Eb_J[R^{p-1}]^{k/(p-1)} \le M^{p} \qquad \textrm{a.s.,}
\]
 with $M = \max_{j\in\Ical} (1 \vee \Eb_j[R^{p-1}]^{1/(p-1)})$ (note that $M<\infty$ by assumption $(iii)$). Therefore 
\[
\Eb_i\left[ R^\alpha - \left(\frac{v_{J}}{v_i} A R^{(J)}\right)^{\alpha} - B^\alpha\right]
\le
M^{\alpha} \Eb_i\left[ \left(\sum_{k=1}^{p-1} \dbinom{p}{k} \left(\frac{v_{J}}{v_i} A \right)^{k} B^{p-k} \right)^{\gamma}\right].
\]
Finally, we can factorise the sum, and we get 
\begin{multline*}
\Eb_i\left[ R^\alpha - \left(\frac{v_{J}}{v_i} A R^{(J)}\right)^{\alpha} - B^\alpha\right]
\le
M^{\alpha} \Eb_i\left[ \left(\left(\frac{v_{J}}{v_i} A + B\right)^p - \left(\frac{v_{J}}{v_i} A \right)^p - B^p \right)^{\gamma}\right] \\
\le 
M^{\alpha} \Eb_i\bigg[ \left(\frac{v_{J}}{v_i} A + B\right)^{\alpha}\bigg].
\end{multline*}
Thus we conclude using $(iv)$ that 
\[
\Eb_i\left[ R^\alpha - \left(\frac{v_{J}}{v_i} A R^{(J)}\right)^{\alpha} - B^\alpha\right]<\infty.
\]
This proves \eqref{eq: R integrability}, and the renewal theorem (\cref{thm: renewal multitype tree}) provides a constant $b_i\ge 0$ such that
\[
\Pb_i(R>t) \underset{t\rightarrow \infty}{\sim} b_i \cdot t^{-\alpha}.
\]
\end{proof}


\subsection{Tail behaviour of the multitype smoothing transform} \label{sec: tail behaviour smoothing}

We now turn to describing the tail asymptotics of the solutions of \eqref{eq: smoothing transform}. The next result is tailored for our purposes, but we believe that one could replace condition $(iii)$ by a condition on $(J_k, C_k, k\ge 1)$, in the spirit of \cite[Theorem 4.1 or 4.2]{JO-C}, and that their proofs should extend to our setup. We here decided to give a minimal statement for our purposes. Again, we follow \cite{JO-C} closely.

\begin{Thm}
\label{thm: tail behaviour smoothing}
Under $\Pb_i$, for any $i\in\Ical$, let $(J_k, C_k, k\ge 1)$ be distributed as $P_i$, $(R_k^{(j)}, k\ge 1)$ be an independent sequence of nonnegative i.i.d. variables with law $Q_j$ for each $j\in\Ical$, and $R$ be a random variable with law $Q_i$. Assume that $P_i(C_k>0)>0$ and that the measure $P_i(\log C_k \in \mathrm{d}y, C_k>0)$ is non-lattice for some $k\ge 1$. We assume that there exist $\alpha>1$ and $0\le\gamma<\alpha$ such that 
\begin{itemize}
\item[(i)] The matrix $m(q)$ defined by \eqref{eq: matrix (C,J)} is finite and irreducible for $q$ in a domain containing $[\gamma, \alpha]$,
\item[(ii)] $\lambda(\alpha) =0$ and if $(v_j, j\in\Ical)$ denotes a positive eigenvector associated to $\lambda(\alpha)=0$, \[0<E_i\left[\sum_{k\ge 1}  \frac{v_{J_k}}{v_i}C_k^\alpha \log C_k  \right]<\infty,\]
\item[(iii)] $\Eb_j[R^{\beta}]<\infty$ for all $0<\beta<\alpha$ and all $j\in\Ical$,
\item[(iv)] $\Eb_j\Big[\Big(\sum_{k\ge 1} C_k\Big)^{\alpha}\Big]<\infty$ for all $j\in\Ical$.
\end{itemize}
Assume that \eqref{eq: smoothing transform} holds with the same notation.
Then for all $i\in\Ical$, there exists a constant $a_i\ge 0$ such that 
\[
\Pb_i(R>t) \underset{t\rightarrow \infty}{\sim} a_i \cdot t^{-\alpha}.
\]
\end{Thm}
\noindent Our proof relies on the following two technical lemmas.

\begin{Lem} \label{lem: tail technical 1}
Under $\Pb_i$, for any $i\in\Ical$, let $(J_k, C_k, k\ge 1)$ be distributed as $P_i$, $(R_k^{(j)}, k\ge 1)$ be an independent sequence of i.i.d. variables with law $Q_j$ for each $j\in\Ical$, and $R$ be a random variable with law $Q_i$. For $\alpha>0$, assume that $\sum_{k\ge 1} \Big(C_kR_k^{(J_k)}\Big) ^{\alpha} < \infty$ a.s., and that $\Eb_i[R^{\beta}]<\infty$ for all $i\in\Ical$ and $0<\beta<\alpha$. Moreover, assume that for some $\varepsilon\in(0,1)$, $\Eb_i\big[\big(\sum_{k\ge1} C_k^{\alpha/(1+\varepsilon)}\big)^{1+\varepsilon}\big]<\infty$ for all $i\in\Ical$. Then for all $i\in\Ical$,
\begin{multline*}
\int_0^{\infty} \left(\Eb_i\left[\sum_{k\ge 1} \mathds{1}_{\{C_k R_k^{(J_k)} >t\}} \right] - \Pb_i\left(\sup_{k\ge 1} \left( C_k R_k^{(J_k)}\right)>t\right) \right) t^{\alpha-1} \mathrm{d}t \\
=
\Eb_i\left[\sum_{k\ge 1} \left(C_k R_k^{(J_k)}\right)^{\alpha} - \bigg(\sup_{k\ge 1}  C_k R_k^{(J_k)}\bigg)^{\alpha}\right]
<
\infty.
\end{multline*}
\end{Lem}
\begin{proof}
The integrand is of course positive, so that the integral makes sense. Moreover, the equality is well-known, and stems from a simple application of Fubini's theorem. The point is to show that this integral is finite. We condition the integrand on $(J_k, C_k, k\ge 1)$:
\begin{multline} \label{eq: lem tail conditioning}
\Eb_i\left[\sum_{k\ge 1} \mathds{1}_{\{C_k R_k^{(J_k)} >t\}}  - \mathds{1}_{\{\sup_{k\ge 1} \left( C_k R_k^{(J_k)}\right)>t\}}\right] 
=  \\
\Eb_i\left[\Eb_i\left[\sum_{k\ge 1} \mathds{1}_{\{C_k R_k^{(J_k)} >t\}}\, \bigg| \, (J_k, C_k, k\ge 1)\right] -1+ \Eb_i\left[\mathds{1}_{\{\sup_{k\ge 1} \left( C_k R_k^{(J_k)}\right)\le t\}}\, \bigg| \, (J_k, C_k, k\ge 1)\right] \right].
\end{multline}
By independence, the last term in the expectation of \eqref{eq: lem tail conditioning} is 
\[
\Eb_i\left[\mathds{1}_{\{\sup_{k\ge 1} \left( C_k R_k^{(J_k)}\right)\le t\}}\, \bigg| \, (J_k, C_k, k\ge 1)\right]
= 
\prod_{k\ge 1} Q_{J_k}\left(R\le \frac{t}{C_k}\,\Big| \,(C_k, k\ge 1)\right)
=
\prod_{k\ge 1} (1-F_{J_k}(t/C_k)),
\]
where $F_j(r):=Q_j(R>r)$ for $j\in\Ical$ and $r\in\R$. The simple inequality $1-x\le \mathrm{e}^{-x}$, $x\ge 0$, yields
\[
\Eb_i\left[\mathds{1}_{\{\sup_{k\ge 1} \left( C_k R_k^{(J_k)}\right)\le t\}}\, \Big| \, (J_k, C_k, k\ge 1)\right]
\le
\mathrm{e}^{-\sum_{k\ge 1} F_{J_k}(t/C_k)}.
\]
Hence the right-hand side of \eqref{eq: lem tail conditioning} is less than $\Eb_i\big[g\big(\sum_{k\ge 1} F_{J_k}(t/C_k)\big)\big]$, where $g:x\mapsto x-1+\mathrm{e}^{-x}$. We now use that $g$ is non-decreasing on $\R_+$. For future purposes, let $\delta=\alpha\varepsilon/(1+\varepsilon)$, and $\beta=\alpha-\delta$. Then by Markov's inequality,
\[
\sum_{k\ge 1} F_{J_k}(t/C_k)
\le
\sum_{k\ge 1} t^{-\beta} C_k^{\beta} \Eb_{J_k}[R^\beta]
\le
M t^{-\beta} \sum_{k\ge 1} C_k^{\beta},
\]
where $M:= \max_{j\in\Ical} \Eb_{j}[R^\beta] < \infty$ by assumption. Define $S_{\beta} := \sum_{k\ge 1} C_k^{\beta}$. By monotonicity of $g$, 
\[
\Eb_i\left[g\bigg(\sum_{k\ge 1} F_{J_k}(t/C_k)\bigg)\right]
\le 
\Eb_i\left[g\big(Mt^{-\beta}S_{\beta}\big)\right].
\]
In total, integrating \eqref{eq: lem tail conditioning} gives
\[
\int_0^\infty \Eb_i\left[\sum_{k\ge 1} \mathds{1}_{\{C_k R_k^{(J_k)} >t\}}  - \mathds{1}_{\{\sup_{k\ge 1} \left( C_k R_k^{(J_k)}\right)>t\}}\right] t^{\alpha-1} \mathrm{d}t
\le 
\int_0^{\infty}\Eb_i\left[g\big(Mt^{-\beta}S_{\beta}\big)\right]  t^{\alpha-1} \mathrm{d}t,
\]
and by the change of variables $s:=Mt^{-\beta}S_{\beta}$, we get
\[
\int_0^\infty \Eb_i\left[\sum_{k\ge 1} \mathds{1}_{\{C_k R_k^{(J_k)} >t\}}  - \mathds{1}_{\{\sup_{k\ge 1} \left( C_k R_k^{(J_k)}\right)>t\}}\right] t^{\alpha-1} \mathrm{d}t
\le 
\frac{1}{\beta} M^{\alpha/\beta}\Eb_i\left[S_{\beta}^{\alpha/\beta}\right]\int_0^{\infty} g(s)  s^{-\alpha/\beta-1} \mathrm{d}s.
\]
Plainly, since $\alpha/\beta=1+\varepsilon\in(1,2)$, the above integral is finite. The expectation is also finite by assumption, and the result follows.
\end{proof}

\begin{Lem} \label{lem: tail technical 2}
Under $\Pb_i$, for any $i\in\Ical$, let $(J_k, C_k, k\ge 1)$ be distributed as $P_i$, and $(R_k^{(j)}, k\ge 1)$ be an independent sequence of i.i.d. variables with law $Q_j$ for each $j\in\Ical$. For $\alpha>1$, assume that $\sum_{k\ge 1} C_kR_k^{(J_k)} < \infty$ a.s., that $Q_i[R^{\beta}]<\infty$, and that $\Eb_i\big[\big(\sum_{k\ge1} C_k\big)^{\alpha}\big]<\infty$ for all $i\in\Ical$ and $0<\beta<\alpha$. Then for all $i\in\Ical$,
\[
\Eb_i\left[\left(\sum_{k\ge 1} C_k R_k^{(J_k)}\right)^{\alpha} - \sum_{k\ge 1} \Big(C_k R_k^{(J_k)}\Big)^{\alpha}\right]
<
\infty.
\]
\end{Lem}

\begin{proof}
Define $p:=\lceil \alpha \rceil$, and $\gamma:=\alpha/p$. We prove the following inequality:
\begin{equation}\label{eq: tail technical 2 intermediate}
\Eb_i\left[\left(\sum_{k\ge 1} C_k R_k^{(J_k)}\right)^{\alpha} - \sum_{k\ge 1} \Big(C_k R_k^{(J_k)}\Big)^{\alpha}\right]
\le
M^{\alpha} \Eb_i \left[\left(\sum_{k\ge 1} C_k\right)^{\alpha}\right],
\end{equation}
where $M:= \max_{j\in\Ical}  \Eb_{j}\big[R^{p-1}\big]^{1/(p-1)}$. The claim is straigthforward with our assumptions once \eqref{eq: tail technical 2 intermediate} is established. For $k\in\N$, write 
\begin{equation}\label{eq: list expansion}
A_p := \Big\{ \qbf=(q_k)_{k\in\N}, \; 0\le q_k\le p-1 \; \text{and} \; \sum_{k\in\N} q_k=p\Big\}.
\end{equation}
A sequence $\qbf$ in $A_p$ has only finitely many non-zero entries $r_1,\ldots,r_k$, and for ease of notation we define $\binom{p}{\qbf}$ as the corresponding multinomial coefficient $\binom{p}{r_1,\ldots,r_k}$, and the product $\prod_{k\ge 1} a_k^{q_k}$ of the terms of a nonnegative sequence $(a_k)_{k\ge 1}$ in a similar way. Then, expanding the first sum, and using that $\big(\sum_{k\ge 1} x_k\big)^{\gamma} \le \sum_{k\ge 1} x_k^{\gamma}$, for $0< \gamma\le 1$ and $x_k\ge 0$,
\[
\Eb_i\left[\left(\sum_{k\ge 1} C_k R_k^{(J_k)}\right)^{\alpha} - \sum_{k\ge 1} \Big(C_k R_k^{(J_k)}\Big)^{\alpha}\right]
\le
\Eb_i\left[ \left( \sum_{\qbf\in A_p} \binom{p}{\qbf} \prod_{k\ge 1}(C_{k} R^{(J_k)}_{k})^{q_k}  \right)^{\gamma} \right].
\]
We now condition on $(J_k, C_k, k\ge 1)$ and use Jensen's inequality with $\gamma\le 1$ to obtain
\begin{equation} \label{eq: tail technical 2 conditoning}
\Eb_i\left[\left(\sum_{k\ge 1} C_k R_k^{(J_k)}\right)^{\alpha} - \sum_{k\ge 1} \Big(C_k R_k^{(J_k)}\Big)^{\alpha}\right]
\le
\Eb_i\left[ \left( \sum_{\qbf\in A_p} \binom{p}{\qbf}  \prod_{k\ge 1} C_k^{q_k} \Eb_{J_k}\big[R^{q_k}\big] \right)^{\gamma} \right].
\end{equation}
Fix $\qbf\in A_p$. By monotonicity of $q \mapsto \Eb_i[X^q]^{1/q}$, we can upper-bound the product of expectations:
\[
\prod_{k\ge 1}\ \Eb_{J_k}\big[R^{q_k}\big] \le  \prod_{k\ge 1} \Eb_{J_k}\big[R^{p-1}\big]^{q_k/(p-1)} \le M^{p},
\]
where $M:= \max_{j\in\Ical} \Eb_{j}\big[R^{p-1}\big]^{1/(p-1)}$. Plugging this inequality into \eqref{eq: tail technical 2 conditoning} provides
\[
\Eb_i\left[\left(\sum_{k\ge 1} C_k R_k^{(J_k)}\right)^{\alpha} - \sum_{k\ge 1} \Big(C_k R_k^{(J_k)}\Big)^{\alpha}\right]
\le
M^{\alpha} \Eb_i\left[\left( \sum_{\qbf\in A_p} \binom{p}{\qbf}  \prod_{k\ge 1} C_k^{q_k}  \right)^{\gamma} \right].
\]
Finally, we can factorise everything back to obtain
\begin{align*}
\Eb_i\left[\left(\sum_{k\ge 1} C_k R_k^{(J_k)}\right)^{\alpha} - \sum_{k\ge 1} \Big(C_k R_k^{(J_k)}\Big)^{\alpha}\right]
&\le
M^{\alpha} \Eb_i\left[\left(\left(\sum_{k\ge 1} C_k \right)^{p} - \sum_{k\ge 1} C_k^{p}\right)^\gamma\right] \\
&\le
M^{\alpha} \Eb_i \left[\left(\sum_{k\ge 1} C_k\right)^{\alpha}\right],
\end{align*}
which is \eqref{eq: tail technical 2 intermediate}.
\end{proof}

\begin{proof}[Proof of \cref{thm: tail behaviour smoothing}]
The theorem is an immediate consequence of \cref{thm: renewal multitype tree} once we prove \eqref{eq: R integrability}. To do so, first write, for $i\in\Ical$ and $t\ge 0$,
\begin{align}
\left|\Pb_i(R>t) - \Eb_i\left[\sum_{k\ge 1} \mathds{1}_{\{\frac{v_{J_k}}{v_i} C_k R_k^{(J_k)} >t\}} \right] \right| 
\le
&\left|\Pb_i(R>t) - \Pb_i\left(\sup_{k\ge 1} \left(\frac{v_{J_k}}{v_i} C_k R_k^{(J_k)}\right)>t\right)\right| \label{eq: integrability term 1}
\\
&\hspace{-2cm} +\left| \Eb_i\left[\sum_{k\ge 1} \mathds{1}_{\{\frac{v_{J_k}}{v_i} C_k R_k^{(J_k)} >t\}} \right] - \Pb_i\left(\sup_{k\ge 1} \left(\frac{v_{J_k}}{v_i} C_k R_k^{(J_k)}\right)>t\right) \right|. \label{eq: integrability term 2}
\end{align}
Both differences in \eqref{eq: integrability term 1} and \eqref{eq: integrability term 2} are nonnegative, so that we can remove the absolute values.

We first deal with the second term using \cref{lem: tail technical 1}. Since $\alpha>1$,  we can choose $\varepsilon>0$ small enough so that $\alpha/(1+\varepsilon)>1$. Now  for all $i\in\Ical$,
\[
\Eb_i\left[ \left(\sum_{k\ge 1} \Big(C_k R_k^{(J_k)}\Big)^{\alpha/(1+\varepsilon)}\right)^{1+\varepsilon}\right]
\le 
\Eb_i\left[ \left(\sum_{k\ge 1} C_k R_k^{(J_k)}\right)^\alpha\right]
<
\infty,
\]
and we conclude by an application of \cref{lem: tail technical 1} that
\begin{equation} \label{eq: first term finite}
\int_0^{\infty} \left(\Eb_i\left[\sum_{k\ge 1} \mathds{1}_{\{\frac{v_{J_k}}{v_i} C_k R_k^{(J_k)} >t\}} \right] - \Pb_i\left(\sup_{k\ge 1} \left(\frac{v_{J_k}}{v_i} C_k R_k^{(J_k)}\right)>t\right) \right) t^{\alpha-1} \mathrm{d}t
<
\infty.
\end{equation}

As for the first term \eqref{eq: integrability term 1}, Fubini's theorem implies
\[
\int_0^{\infty}\left(\Pb_i(R>t) - \Pb_i\left(\sup_{k\ge 1} \left(\frac{v_{J_k}}{v_i} C_k R_k^{(J_k)}\right)>t\right)\right) t^{\alpha-1} \mathrm{d}t 
=
\Eb_i\left[R^{\alpha} - \sup_{k\ge 1} \left(\frac{v_{J_k}}{v_i} C_k R_k^{(J_k)}\right)^{\alpha}\right].
\]
We then again split 
\begin{multline*}
\Eb_i\left[R^{\alpha} - \sup_{k\ge 1} \left(\frac{v_{J_k}}{v_i} C_k R_k^{(J_k)}\right)^{\alpha}\right]
=
\Eb_i\left[R^{\alpha} - \sum_{k\ge 1} \Big(\frac{v_{J_k}}{v_i} C_k R_k^{(J_k)}\Big)^{\alpha}\right] \\
+ \Eb_i\left[\sum_{k\ge 1} \Big(\frac{v_{J_k}}{v_i} C_k R_k^{(J_k)}\Big)^{\alpha} - \sup_{k\ge 1} \left(\frac{v_{J_k}}{v_i} C_k R_k^{(J_k)}\right)^{\alpha}\right].
\end{multline*}
Notice that both terms are positive, so that the splitting makes sense. The last term is actually \eqref{eq: first term finite}, and so is finite. For the first one, we use \cref{lem: tail technical 2}. This concludes the proof of \eqref{eq: R integrability}, and hence \cref{thm: tail behaviour smoothing} is proved.
\end{proof}

\subsection{Moments of $\Mcal^-(\infty)$}
In this section, we prove that the limit $\Mcal^-(\infty)$ of the martingale $(\Mcal^-(n), n\ge 0)$ is in $L^{\beta}$ for all $0<\beta<\omega_+/\omega_-$. This allows to leverage \cref{thm: tail behaviour smoothing} to prove the last item of \cref{prop: martingale M^-} (the other assumptions being either clearly fulfilled or already proved). We mostly rely on \cite{Liu}.

\begin{Prop} \label{prop: Mcal L^beta}
The variable $\Mcal^-(\infty)$ is in $L^{\beta}$ for all $0<\beta<\omega_+/\omega_-$.
\end{Prop}
\noindent We divide the proof of \cref{prop: Mcal L^beta} into several steps, following \cite{Liu}. First, we call attention to a multitype version of a random linear recursion, for which we provide conditions for the solutions to be in $L^{\beta}$. Then, we prove that (a variant of) $\Mcal^-(\infty)$ satisfies such a recursion and that all the conditions are fulfilled, thus concluding the proof of \cref{prop: Mcal L^beta}.

\medskip
\noindent \textbf{The multitype linear recursion.} We start by setting the framework for the multitype linear recursion. Consider, for each $i\in\Ical$, two positive random variables $A^{(i)}$ and $B^{(i)}$ whose distributions depend on the type $i$. Let $J^{(i)}, i\in\Ical,$ be random variables taking values in $\Ical$. We are interested in the solutions $X$ to the following random equation in law:
\begin{equation} \label{eq: multitype recursion}
\forall i\in\Ical, \quad X^{(i)} \overset{\Lcal}{=} A^{(i)} X^{(J^{(i)})} + B^{(i)},
\end{equation}
where the $X^{(i)}, i\in\Ical,$ are independent of the $A^{(i)}, B^{(i)}, J^{(i)}, i\in \Ical$. It will be convenient to consider, for $i\in\Ical$, i.i.d. sequences $(A_k^{(i)}, k\ge 1)$, $(B_k^{(i)}, k\ge 1)$, $(X_k^{(i)}, k\ge 1)$ and $(J_k^{(i)}, k\ge 1)$ with respective laws $A^{(i)}$, $B^{(i)}$, $X^{(i)}$ and $J^{(i)}$ under measure $P_i$. Moreover, with a slight abuse of notation we set, under $P_i$, $J_0=i$ and by recursion $J_{k+1} = J_{k+1}^{(J_k)}$. We assume that the chain $J=(J_k, k\ge 0)$ is irreducible and aperiodic, and write $\pi$ for its invariant probability measure. Denote $P_{\pi}$ the measure $\sum_{i\in\Ical} \pi(i) P_i$. The following lemma is a multitype version of \cite[Theorem 1]{Grin}.

\begin{Lem} \label{lem: Grin}
Suppose that for all $i\in\Ical$, $P_i(A^{(i)} =0)=0$ and that for some positive vector $(c_i,i\in\Ical)$,
\[
-\infty< E_{\pi}\left[\log \Big(\frac{c_{J_1}}{c_{J_0}} A^{(J_0)}\Big)\right]<0, \quad \text{and} \quad E_i\left[\log^+ B^{(i)}\right]<\infty.
\]
Then the recursion \eqref{eq: multitype recursion} has a unique solution, which is $P_i$--almost surely given by 
\[
X^{(i)} = \sum_{k\ge 0} A_1^{(i)}A_2^{(J_1)}\cdots A_k^{(J_{k-1})} B_{k+1}^{(J_{k})}.
\]
\end{Lem}

\begin{proof}
It is sufficient to prove that the series $\sum_k A_1^{(i)}A_2^{(J_1)}\cdots A_k^{(J_{k-1})} B_{k+1}^{(J_{k})}$ is convergent (the claim then follows by iterating \eqref{eq: multitype recursion}). We mimic the proof in \cite{Grin}. Write $C_k^{(i)}:= \log^+ B_k^{(i)}$. Since
\[
\sum_{k\ge 1} P_i(C_k^{(i)}\ge A k) \le \frac1A\sum_{k\ge 1} \int_{A(k-1)}^{Ak} P_i(C_1^{(i)}\ge x) \mathrm{d}x = \frac1A \int_0^\infty P_i(C_1^{(i)}\ge x) \mathrm{d}x = E_i[C_1^{(i)}] <\infty,
\]
by Borel-Cantelli lemma, $\limsup_{k\to\infty} C_k^{(i)}/k < A$, $P_i$--a.s. Taking $A\to 0$, we deduce that $C_k^{(i)}/k \to 0$, $P_i$--a.s. This being true for all $i\in\Ical$, we have $C_k^{(J_{k-1})}/k \to 0$, $P_i$--a.s. Now for $k\ge 0$,
\begin{align*}
\left(A_1^{(i)}A_2^{(J_1)}\cdots A_k^{(J_{k-1})} B_{k+1}^{(J_{k})}\right)^{1/k}
&= \left(\frac{c_{i}}{c_{J_{k}}}\right)^{1/k}\left(\frac{c_{J_1}}{c_{i}}A_1^{(i)} \cdot \frac{c_{J_2}}{c_{J_1}}A_2^{(J_1)}\cdots \frac{c_{J_k}}{c_{J_{k-1}}}A_k^{(J_{k-1})} B_{k+1}^{(J_{k})}\right)^{1/k}  \\
&\le C^{1/k} \exp\big(C_{k+1}^{(J_{k})}/k\big) \cdot \exp\Big( \frac1k \sum_{l=1}^k \log \Big(\frac{c_{J_l}}{c_{J_{l-1}}}A_l^{(J_{l-1})}\Big)\Big),
\end{align*}
where $C= \max_{j\in\Ical} \frac{c_i}{c_j}$. By the previous point and the ergodic theorem of \cite{Lal}, $P_i$--almost surely,
\[
\underset{k\to\infty}{\limsup}\left(A_1^{(i)}A_2^{(J_1)}\cdots A_k^{(J_{k-1})} B_{k+1}^{(J_{k})}\right)^{1/k} \le \exp\Big( E_{\pi}\Big[\log \Big(\frac{c_{J_1}}{c_{J_0}} A^{(J_0)}\Big)\Big]\Big) < 1.
\]
We conclude the proof by an application of Cauchy's criterion on series convergence.
\end{proof}

We now use \cref{lem: Grin} to provide conditions for the solutions of \eqref{eq: multitype recursion} to be in some $L^\beta$ (see \cite[Lemma 3.2]{Liu}).

\begin{Prop} \label{prop: moments multitype recursion}
In addition to the assumptions of \cref{lem: Grin}, suppose that for some $\beta>0$ and some positive vector $(c_i, i\in\Ical)$, $E_i[(\frac{c_{J_1}}{c_i}|A^{(i)}|)^\beta]<1$ and $E_i[|B^{(i)}|^\beta]<\infty$ for all $i\in\Ical$. Then the solution of \eqref{eq: multitype recursion} satisfies $E_i[|X^{(i)}|^\beta]<\infty,$ for all $i\in\Ical$.
\end{Prop}
\begin{proof}
Denote by $\widetilde{A}^{(j)}_k := \frac{c_{J_k}}{c_{j}}A^{(j)}_k$.
Using \cref{lem: Grin}, one can write $X^{(i)}$ as
\[
X^{(i)} = \sum_{k\ge 0} \frac{c_i}{c_{J_k}} \widetilde{A}_1^{(i)}\widetilde{A}_2^{(J_1)}\cdots \widetilde{A}_k^{(J_{k-1})} B_{k+1}^{(J_{k})}.
\]
Assume $\beta\ge 1$ for the moment. Then by Minkowski's inequality, 
\[
||X^{(i)}||_{\beta} \le C \sum_{k\ge 0}  \big|\big| \widetilde{A}_1^{(i)}\widetilde{A}_2^{(J_1)}\cdots \widetilde{A}_k^{(J_{k-1})} B_{k+1}^{(J_{k})} \big|\big|_{\beta},
\]
where $||Y||_{\beta} = E_i[|Y|^{\beta}]^{1/\beta}$ and $C= \max_{j\in\Ical} \frac{c_i}{c_j}$. We now condition on the types $J_1, \ldots, J_k$ and use conditional independence to get:
\[
||X^{(i)}||_{\beta} \le C \sum_{k\ge 0} E\left[ E_i\big[|\widetilde{A}^{(i)}_1|^{\beta}\big] E_{J_1}\big[|\widetilde{A}^{(J_1)}_1|^{\beta}\big]\cdots E_{J_{k-1}}\big[|\widetilde{A}^{(J_{k-1})}_1|^{\beta}\big] E_{J_{k}}\big[B_{1} ^{\beta}\big] \right]^{1/\beta}.
\]
Set $a := \max_{j\in\Ical} E_j\big[|\widetilde{A}^{(j)}_1|^{\beta}\big]^{1/\beta}<1$ and $b:= \max_{j\in\Ical} E_j\big[B_1^{\beta}\big]^{1/\beta}$. Then 
\[
||X^{(i)}||_{\beta} 
\le
C b \sum_{k\ge 0} a^k = \frac{C b}{1-a} <\infty.
\]
For $\beta< 1$, one uses the inequality 
\[
||X^{(i)}||^{\beta}_{\beta} \le C \sum_{k\ge 0}  \big|\big| \widetilde{A}_1^{(i)}\widetilde{A}_2^{(J_1)}\cdots \widetilde{A}_k^{(J_{k-1})} B_{k+1}^{(J_{k})} \big|\big|_{\beta}^{\beta},
\]
and concludes in the same way.
\end{proof}

\medskip
\noindent \textbf{Moments of $\Mcal^-(\infty)$ through equation \eqref{eq: multitype recursion}.} Next, we prove that such an equation as \eqref{eq: multitype recursion} arises naturally when considering $\widetilde{\Mcal}^-(\infty):= \Mcal^-(\infty)/v_{\Jhat(0)}$. Recall the notation from the change of measure in \cref{sec:change measure}, and let $\Jhat_k := \Jhat(b_{\Lcal(k)})$ and $\Xhat_k := \Xhat(b_{\Lcal(k)})$. For all $u\in\Ub$, denote by $M_u$ the value of limit $\widetilde{\Mcal}^-(\infty)$ on the tree re-rooted at $u$. 
\begin{Prop}
Set $\mathcal{A}^{(i)}_k := \frac{v_{\Jcal_k(0)}}{v_i} |\Xcal_k(0)|^{\omega_-}, k\ge 1,$ and
\[
A^{(i)} := \mathcal{A}^{(i)}_{\Lcal(1)}, \quad \text{and} \quad B^{(i)} := \sum_{k\ne \Lcal(1)} \mathcal{A}^{(i)}_k M_k.
\] 
Then $\widetilde{\Mcal}^-(\infty)$ solves the family of equations, for all $i\in\Ical$ under $\Phat^-_{1,i}$, 
\begin{equation} \label{eq: multitype recursion spine}
\widetilde{\Mcal}^-(\infty) \overset{\Lcal}{=} A^{(i)} M_{\Lcal(1)} + B^{(i)}.
\end{equation}
Moreover, the joint law is described by the following formula: for all nonnegative measurable functions $f, g, h$,
\begin{equation} \label{eq: multitype recursion joint law}
\Ehat^-_{1,i}\left[ f(A^{(i)}) g(B^{(i)}) h(M_{\Lcal(1)}) \right]
=
\Ecal_{1,i} \left[ \sum_{k\ge 1} \mathcal{A}^{(i)}_k f(\mathcal{A}^{(i)}_k) g\left(\sum_{l\ne k} \mathcal{A}^{(i)}_l M_l \right) \Ecal_{1,\Jcal_k(0)}\left[h(M_k)M_k\right] \right].
\end{equation}
In particular, under $\Phat^-_{1,i}$ the $M_{\Lcal(1)}$ is independent of $(A^{(i)}, B^{(i)})$ and conditionally on $\Jhat_1$, is distributed as $\widetilde{\Mcal}^-(\infty)$ under $\Phat^-_{1,\Jhat_1}$.
\end{Prop}
\begin{proof}
We note that under $\Phat^-_{1,i}$, $\widetilde{\Mcal}^-(\infty)$ satisfies $\widetilde{\Mcal}^-(\infty) = \sum_{k\ge 1} \mathcal{A}^{(i)}_k M_k$ by the branching property. The fact that $\widetilde{\Mcal}^-(\infty)$ solves \eqref{eq: multitype recursion spine} is then obvious by decomposing over the spine and the other cells.  For the second statement, by definition of $\Phat_{1,i}^-$, we have
\[
\Ehat^-_{1,i}\left[ f(A^{(i)}) g(B^{(i)}) h(M_{\Lcal(1)}) \right]
=
\Ecal_{1,i} \left[ \sum_{k\ge 1} \mathcal{A}^{(i)}_k f(\mathcal{A}^{(i)}_k) g\left(\sum_{l\ne k} \mathcal{A}^{(i)}_l M_l \right) h(M_k)M_k \right].
\]
Using the independence of the $M_k, k\ge 1,$ we retrieve \eqref{eq: multitype recursion joint law}.
\end{proof}

Given that \eqref{eq: multitype recursion spine} holds, for all $i\in\Ical$, it will be convenient to denote by $M^{(i)}$ a random variable with the same law as $\widetilde{\Mcal}^-(\infty)$ under $\Phat^-_{1,i}$, defined on the same probability space and independent of all other quantities. We are then interested in the set of equations
\[
M^{(i)} \overset{\Lcal}{=} A^{(i)} M^{(\Jhat_1)} + B^{(i)},
\]
for all $i\in \Ical$. This recasts \eqref{eq: multitype recursion spine} in the framework of \eqref{eq: multitype recursion}. In  order to apply \cref{prop: moments multitype recursion}, we will need the following moment bounds on $B^{(i)}$.
\begin{Lem} \label{lem: moment B spine}
We have the following upper-bounds on $\Ehat_{1,i}[|B^{(i)}|^{p}], p>0$:
\begin{itemize}
\item If $p\le 1$, then $\Ehat_{1,i}[|B^{(i)}|^{p}] \le \Ecal_{1,i} \big[\big( \sum_{k\ge 1} \mathcal{A}^{(i)}_k \big)^{p+1}\big]$;
\item If $p>1$, then $\Ehat_{1,i}[|B^{(i)}|^{p}] \le \Ecal_{1,i}\big[\big(\sum_{k\ge 1} \mathcal{A}^{(i)}_k\big)^{p+1}\big] \cdot \max_{j\in\Ical} \Ecal_{1,j}[|M^{(j)}|^p] $.
\end{itemize}
Furthermore, if $(w_i,i\in\Ical)$ is a positive vector and we set $c_i:=\frac{w_i^{1/p}}{v_i^{1+1/p}}$ for $p>0$, then 
\[\Ehat_{1,i}\left[\left(\frac{c_{\Jhat_1}}{c_i}A^{(i)}\right)^{p}\right] = \Ecal_{1,i}\left[\sum_{k\ge 1} \frac{w_{\Jcal_k(0)}}{w_i} |\Xcal_k(0)|^{(p+1)\omega_-} \right].\]
\end{Lem}
The proof is adapted from \cite[Lemma 4.2]{Liu}.
\begin{proof}
We use \eqref{eq: multitype recursion joint law} with $f=h=1$ and $g:x\mapsto x^p$:
\[
\Ehat^-_{1,i}\left[ |B^{(i)}|^p\right]
=
\Ecal_{1,i} \left[ \sum_{k\ge 1} \mathcal{A}^{(i)}_k  \left(\sum_{l\ne k} \mathcal{A}^{(i)}_l M_l \right)^p \right].
\]
Assume $p\le 1$. By the conditional Jensen inequality for $p\le 1$,
\begin{align*}
\Ehat^-_{1,i}\left[ |B^{(i)}|^p\right]
&\le
\Ecal_{1,i} \left[ \sum_{k\ge 1} \mathcal{A}^{(i)}_k \Ecal_{1,i} \left[\sum_{l\ne k} \mathcal{A}^{(i)}_l M_l \, \bigg| \, \Gscr_1 \right]^p \right] \\
&\le 
\Ecal_{1,i} \left[ \sum_{k\ge 1} \mathcal{A}^{(i)}_k \left(\sum_{l\ne k} \mathcal{A}^{(i)}_l \right)^p \right],
\end{align*}
since $\Ecal_{1,i} [ M_l \, | \, \Gscr_1 ] = \Ecal_{1,i} \Big[ \Ecal_{1,\Jcal_l(0)}\Big[ \widetilde{\Mcal}^-(\infty) \Big] \Big] = 1$ for all $l$ (recall from \cref{prop: martingale M^-}(ii) that $\Mcal(\infty)$ is a limit in $L^1$).
Hence
\[
\Ehat^-_{1,i}\left[ |B^{(i)}|^p\right]
\le
\Ecal_{1,i} \left[\left( \sum_{k\ge 1} \mathcal{A}^{(i)}_k \right)^{p+1}\right].
\]
Now assume $p>1$. We use Jensen's inequality on the sum:
\[
\left(\sum_{k\ge 1} a_k z_k \right)^p \le \sum_{k\ge 1} a_k z_k^p \quad \text{if} \; \sum_{k\ge 1} a_k =1.
\]
This gives
\begin{align*}
\Ehat^-_{1,i}\left[ |B^{(i)}|^p\right]
&\le \Ecal_{1,i} \left[ \sum_{k\ge 1} \mathcal{A}^{(i)}_k  \left(\sum_{l\ne k} \mathcal{A}^{(i)}_l  \right)^{p-1}  \cdot \sum_{l\ne k} \mathcal{A}^{(i)}_l M_l^p \right] \\
&\le \Ecal_{1,i} \left[ \left(\sum_{k\ge 1} \mathcal{A}^{(i)}_k \right)^{p} \cdot \sum_{l\ne k} \mathcal{A}^{(i)}_l \Ecal_{1,\Jcal_l(0)}\left[M_l^p\right] \right] \\
&\le
 \Ecal_{1,i}\left[\left(\sum_{k\ge 1} \mathcal{A}^{(i)}_k\right)^{p+1}\right] \cdot \max_{j\in\Ical} \Ecal_{1,j}\left[|M^{(j)}|^p\right].
\end{align*}
The second claim about $A^{(i)}$ is straightforward.
\end{proof}

\medskip
\noindent \textbf{Proof of \cref{prop: Mcal L^beta}.} First of all, notice that the recursion \eqref{eq: multitype recursion spine} satisfies the assumptions of \cref{lem: Grin}. Indeed, we assumed that for all $i\in\Ical$,
\[
\Ehat_{1,i}\left[\log \Big(\frac{v_i}{v_{\Jhat_1}}A^{(i)}\Big)\right] = \omega_- \Ecal_{1,i}\left[\sum_{k\ge 1} \frac{v_{\Jcal_k(0)}}{v_i} |\Xcal_k(0)|^{\omega_-} \log |\Xcal_k(0)| \right] \in (-\infty,0),
\]
which is a stronger requirement. Then, we use formula \eqref{eq: multitype recursion joint law} to see that 
\begin{equation} \label{eq: transfer moment p -> p+1}
\Ehat_{1,i}\left[|M^{(i)}|^p\right] = \Ehat_{1,i}\left[\Ecal_{1,\Jhat_1}\left[|M^{(i)}|^{p+1}\right]\right].
\end{equation}
Now from \cref{prop: moments multitype recursion} and \cref{lem: moment B spine}, we see that the conditions 
\begin{equation} \label{eq: three assumptions}
\Ecal_{1,i}\left[|M^{(i)}|^p\right]<\infty, \quad 
\Ecal_{1,i} \left[\left( \sum_{k\ge 1} \mathcal{A}^{(i)}_k \right)^{p+1}\right]<\infty \quad \text{and} \quad
\Ecal_{1,i}\left[\sum_{k\ge 1} \frac{w_{\Jcal_k(0)}}{w_i} |\Xcal_k(0)|^{(p+1)\omega_-} \right]<1,
\end{equation}
for all $i\in\Ical$, for some positive vector $w$ and $p>0$, imply that $\Ehat_{1,i}[|M^{(i)}|^p]<\infty$ and hence by \eqref{eq: transfer moment p -> p+1}, $\Ecal_{1,i}[|M^{(i)}|^{p+1}]<\infty$ for all $i\in\Ical$. 

Assume that $\omega_+/\omega_-\in (1,2]$, and consider $p<\omega_+/\omega_- - 1$. We now check the three assumptions in \eqref{eq: three assumptions}. Since $p\le 1$ and $\Ecal_{1,i}[|M^{(i)}|]<\infty$, we have $\Ecal_{1,i}[|M^{(i)}|^p]<\infty$. In addition, since $p+1 < \omega_+/\omega_-$, $\Ecal_{1,i} \Big[\Big( \sum_{k\ge 1} \mathcal{A}^{(i)}_k \Big)^{p+1}\Big]<\infty$ by the first claim of \cref{prop: martingale M^-}. Finally, we know that $\omega_-<(p+1)\omega_-<\omega_+$, and that $\lambda(\omega_-)=\lambda(\omega_+)=0$. By convexity of $\lambda$, $\lambda((p+1)\omega_-)<0$. If we choose $w$ to be a positive eigenvector of $m((p+1)\omega_-)$ associated with the leading eigenvalue, we therefore get 
\[
\Ecal_{1,i}\left[\sum_{k\ge 1} \frac{w_{\Jcal_k(0)}}{w_i} |\Xcal_k(0)|^{(p+1)\omega_-} \right]<1.
\]
This altogether triggers $\Ecal_{1,i}[|M^{(i)}|^{p+1}]<\infty$ for all $i\in\Ical$, and hence \cref{prop: Mcal L^beta} follows.

The general claim follows by recursion. For example, if $\omega_+/\omega_-\in(2,3]$, we first leverage the previous arguments for $p=1$ and obtain that $\Ecal_{1,i}[|M^{(i)}|^{2}]<\infty$. Taking any $\beta<\omega_+/\omega_--1$, the very same arguments provide $\Ecal_{1,i}[|M^{(i)}|^{\beta+1}]<\infty$ (note that the first step for $p=1$ is here used to deduce that $\Ecal_{1,i}[|M^{(i)}|^{\beta}]<\infty$). This concludes the proof of \cref{prop: Mcal L^beta}.

\subsection{Proof of Theorem \ref{entrancelaw}}
Our arguments rely heavily on the seminal work of Bertoin and Yor \cite{BY2002} where the entrance law of positive self-similar Markov processes is determined. Before we prove  Theorem \ref{entrancelaw}, let us first discuss a duality property for self-similar Markov processes with types and establish some nice properties of their associated resolvent operators.  

The dual of  self-similar Markov processes with types $(X,J)$ is defined as follows. First, we set
\[
    \varphi^\natural(t):= \inf\left\{s>0, \; \int_0^s \exp(-\alpha \xi(u)) \mathrm{d}u > t\right\}, \quad t\ge 0,
\]
and then we define for $x>0$  the pair $(X^\natural, J^\natural)$ where
\[
X^\natural(t) := x \exp(-\xi(\varphi^\natural(tx^{-\alpha}))) \quad \textrm{and}\quad J^\natural(t) := \Theta^\natural(\varphi^\natural(tx^{-\alpha})), \quad t\ge 0,
\]
where $\Theta^\natural$ is a Markov chain with intensity matrix $Q^\natural$ defined  in \eqref{dualQ} and with the convention that $(X^\natural(t),J^\natural(t))=\partial$ when $t\ge \zeta^\natural$ where
\[
\zeta^\natural:=x^{\alpha}\int_0^{\infty} \exp(-\alpha \xi(u)) \mathrm{d}u.
\]
For $x>0$ and $j\in \Ical$, we denote by $\mathbbm{P}^\natural_{x, j}$ for  the distribution of the pair $(X^\natural, J^\natural)$.

Recall that for $q\ge 0$ and a measurable $f:\mathbb{R}_+\times \mathcal{I}\to \mathbb{R}_+$,  the resolvent operators are given by
\begin{equation} \label{eq: def of resolvents}
V^{(q)}f(x, j)=\mathbbm{E}_{x, j} \left[ \int_0^\infty e^{-qt} f(X(t), J(t)) \mathrm{d} t \right] \quad \textrm{and}\quad V^{\natural, (q)}f(x, j)=\mathbbm{E}^\natural_{x, j} \left[ \int_0^\zeta e^{-qt} f(X(t), J(t)) \mathrm{d} t \right]. 
\end{equation}
From Lemma \ref{dlemma}, we deduce the following result.
\begin{Prop}
\label{prop: wd} 
For every $q\ge 0$ and every measurable functions $f, g:\mathbb{R}\times \Ical \to \mathbb{R}$, we have
\begin{equation}\label{dualid}
\sum_{j\in \Ical}\int_{0}^{\infty}  f(x, j) V^{(q)}g(x, j) \mu(\mathrm{d} x, j)=\sum_{k\in\Ical}\int_{0}^{\infty}  g(x, k) V^{\natural,(q)}f(x,k) \mu(\mathrm{d} x,k),
\end{equation}
where
\[
\mu(\mathrm{d} x, j)=x^{\alpha-1}\pi_{j}\mathrm{d} x, \qquad \textrm{for}\quad x>0 \textrm{ and } j\in \mathcal{I}.
\]
\end{Prop}
\begin{proof} Using the  Lamperti-type transform of $(X, J)$ in \eqref{eq: Lamperti MAP}, we first observe
\[
V^{(q)}g(x, j) =  \mathtt{E}_{ j} \left[ \int_0^\infty e^{-qx^\alpha I_u(\xi)} g(x e^{\xi(u)}, \Theta(u)) x^\alpha e^{\alpha \xi(u)} \mathrm{d} u \right],
\]
where
  $$I_u(\xi) = \int_0^u \exp( \alpha \xi(s) )\mathrm{d} s.$$
Thus, it is clear
  \[
  \begin{split}
\sum_{j\in \mathcal{I}}&\int_{0}^{\infty}  f(x, j) V^{(q)}g(x, j) \mu(\mathrm{d} x, j)\\
&=\sum_{j\in \mathcal{I}}\pi_j\mathtt{E}_j \left[ \int_0^\infty \int_0^\infty x^{\alpha-1} f(x,j) e^{- qx^\alpha I_u(\xi)} g(xe^{\xi(u)}, \Theta(u)) x^\alpha e^{\alpha \xi(u)} \mathrm{d} u \mathrm{d} x \right] \\
&=\sum_{j\in \mathcal{I}}\pi_j\int_0^\infty y^{\alpha-1}  \int_0^\infty \mathtt{E}_j \left[g(y, \Theta(u))e^{- qy^\alpha e^{-\alpha \xi(u)} I_u(\xi)} f(ye^{-\xi(u)}, j) y^\alpha e^{-\alpha \xi(u)}   \mathrm{d} u \right] \mathrm{d} y.
\end{split}
\]
From Lemma \ref{dlemma}, we deduce
\begin{align*}
\sum_{j\in\Ical}{\pi_j} \mathtt{E}_j& \left[ g(y, \Theta^\prime(0))e^{- qy^\alpha I^\prime_u(\xi)} f(ye^{\xi^\prime(u)}, \Theta^\prime(u)) y^\alpha e^{\alpha \xi^\prime(u)} \right] \\
&\hspace{4cm}=\sum_{k\in\Ical}{\pi_k }g(y, k)\mathtt{E}^\natural_k \left[ e^{- qy^\alpha I_u(\xi)} f(ye^{\xi(u)},  \Theta(u)) y^\alpha e^{\alpha \xi(u)}\right],
\end{align*}
where $\xi^\prime(s)=\xi(u-s)-\xi(u)$ and $\Theta^\prime(s)=\Theta(u-s)$. Therefore
\[\begin{split}
\sum_{j\in \mathcal{I}}&\int_{0}^{\infty}  f(x, j) V^{(q)}g(x, j) \mu(\mathrm{d} x, j)\\
&=\sum_{k\in \mathcal{I}}\pi_k\int_0^\infty y^{\alpha-1}g(y, k) \int_0^\infty  \mathtt{E}^\natural_k \left[ e^{ -q y^\alpha  I_u(\xi)} f(ye^{\xi(u)}, \Theta(u)) y^\alpha e^{\alpha \xi(u)} \right] \mathrm{d} u \mathrm{d} y.
\end{split}
\]
Finally, using the  Lamperti-type transform of $(X, J)$, under $\mathbbm{P}^\natural_{y,k}$, we observe
\[
\begin{split}
\sum_{k\in \mathcal{I}}&\int_{0}^{\infty}  g(y, k) V^{\natural,(q)}f(y, k) \mu(\mathrm{d} y, k)\\
&=\sum_{k\in \mathcal{I}} \pi_k\int_{0}^\infty y^{\alpha-1} g(y,k) \int_0^\infty  \mathtt{E}^\natural_k \left[ e^{ -q y^\alpha  I_u(\xi)} f(ye^{\xi(u)}, \Theta(u)) y^\alpha e^{\alpha \xi(u)}  \right] \mathrm{d} u \mathrm{d} y.
\end{split}
\]
Hence putting all the pieces together, we observe that  (\ref{dualid}) holds.
\end{proof}
The key to the proof of  Theorem \ref{entrancelaw} is a consequence of the Markov additive renewal theorem for the potential measure of the MAP $(\xi, \Theta)$.  There is a relatively wide body of literature concerning Markov additive renewal theory, see for instance \cite{Alsmeyer94, Alsmeyer14, Kesten74, Lalley84}. Although the literature mostly deals with the case of discrete-time, one can nonetheless identify the following renewal-type theorem for the potential measure $U_{ij}({\rm d} x)$, defined as
\[
U_{ij}({\rm d} x)=\mathtt{E}_{i}\left[\int_0^\infty \mathds{1}_{\{\xi(t)\in {\rm d} x, \Theta(t)=j\}}{\rm d} t\right], \qquad x\in \mathbb{R}.
\]
To give a precise statement, let us  recall that a function $g:\mathbb{R}\to \mathbb{R}$ is said to be directly Riemann integrable if 
\[
\lim_{n\to\infty} \underline{g_n}( y)=\lim_{n\to\infty} \overline{g_n}(y)=g(y) \qquad \textrm{is integrable,}
\]
where for every integer $n$, $\underline{g_n}$ (respectively $\overline{g_n}$) denotes the largest (respectively, the smallest) function with $\underline{g_n}\le g $ (respectively $g\le\overline{g_n}$) which is constant on the intervals $[k2^{-n}, (k+1)2^{-n})$, for every $k\in \mathbb{Z}$. We may now state the following
\begin{Thm}\label{mrenw} Assume that $\xi$ is not concentrated on a lattice and has positive mean 
\[
0<m=\mathtt{E}_{0,\pi}[\xi(1)]<\infty.
\]
Then the following hold
\begin{itemize}
\item[(i)] For all $i,j\in\mathcal{I}$, 
\[
\lim_{x\to\infty}\frac{U{i,j}([0,x])}{x}=\frac{\pi_j}{m}.
\]
\item[(ii)] For $x>0$ and $i,j\in \mathcal{E}$, the weak limit of the overshoot $(\xi(\tau^+_a)-a, \Theta({\tau^+_a}))$ exists as $a$ goes to $\infty$, where $\tau^+_z=\inf\{t:\xi(t)\ge z\}$. More precisely
\[
\nu({\rm d} x, j):=\textrm{w}-\lim_{a\to\infty} \mathtt{P}_i\Big(\xi(\tau^+_a)-a\in {\rm d} x, \Theta({\tau^+_a})=j\Big),
\]
where $\textrm{w}-\lim$ denotes  weak limit of probability measures. \item[(iii)] Let us assume that for each $j\in \mathcal{I}$,   $g(\cdot,j)$ is direct Riemann integrable, then
\[
\lim_{z\to\infty}\sum_{j\in \mathcal{I}}\int_{\mathbb{R}} g(y-z,j)U_{i,j}({\rm d} y)=\frac{1}{m} \sum_{j\in \mathcal{I}}\pi_j\int_{\mathbb{R}}g(x,j){\rm d} x.
\]
\end{itemize}
\end{Thm}
Part (i) and (ii) is the continuous-time analogue of the Markov additive renewal theorem in \cite{Lalley84}. Both follow from Theorem 28 in \cite{DeDoKy}. Part (iii) follows from similar arguments as in \cite{Alsmeyer14}.

With the previous result in hand, we may now state the following result. 

\begin{Lem}\label{lemmrthm} Let $f:(0,\infty)\times\mathcal{I}\to \mathbb{R}$ be such that for each $j\in \mathcal{I}$, $y\mapsto e^{\alpha y}f(e^y, j)$ is directly Riemann integrable. Then
\[
\lim_{x\to 0+}V^{(0)}f(x,i)=\frac{1}{m}\sum_{j\in \mathcal{I}}\pi_j\int_0^\infty f(y,j)y^{\alpha-1}{\rm d}y
\]
\end{Lem}
\begin{proof} Observe, from the Lamperti-type transform of $(X,J)$ that the potential $V^{(0)}$ can be expressed in terms of the potential operator of $(\xi, \Theta)$ as follows 
\[
\begin{split}
V^{(0)}f(x, i)&=\mathtt{E}_{ i} \left[ \int_0^\infty f(e^{\xi(t)+\log x}, \Theta(t)) e^{\alpha (\xi(t)+\log x)}\mathrm{d} t \right]\\
&=\sum_{j\in\mathcal{I}}\int_{\mathbb{R}} f(e^{y+\log x}, j)e^{\alpha(y+\log x)} U_{i,j}({\rm d} y). 
\end{split}
\]
Hence from Theorem \ref{mrenw} part (iii), we deduce
\[
\lim_{x\to 0+}V^{(0)}f(x,i)=\frac{1}{m}\sum_{j\in \mathcal{I}}\pi_j\int_{-\infty}^\infty f(e^y,j)e^{\alpha y}{\rm d}y=\frac{1}{m}\sum_{j\in \mathcal{I}}\pi_j\int_{0}^\infty f(z,j)z^{\alpha -1}{\rm d}z,
\]
as expected.
\end{proof} 
Next, we study regularity properties of the resolvent $V^{(q)}$.
\begin{Lem}\label{lempot} For each $i\in \mathcal{I}$, we assume that  $f(\cdot, i):\mathbb{R}_+\to \mathbb{R}$  is a bounded continuous function. Then for every $q>0$, $V^{(q)}f(\cdot, i)$ is continuous and bounded on $(0,\infty)$. Moreover, if   $f(\cdot, i)$ has compact support, then the function $y\mapsto e^{\alpha y}V^{(q)}f(e^y, i)$ is directly Riemann integrable on $\mathbb{R}$.
\end{Lem}
\begin{proof} Since  the MAP $(\xi, \Theta)$ is a Feller process with c\`adl\`ag paths then, for each $i\in \mathcal{I}$ the map $x\mapsto \mathbbm{E}_{x, i}[f(X(t), J(t))]= \mathtt{E}_{ i} \left[   f(x e^{\xi(\varphi(tx^{-\alpha})}), \Theta(\varphi(tx^{-\alpha}))) \right] $ is continuous and bounded on $(0,\infty)$, for each $i\in \mathcal{I}$, and in particular the same holds for $x\mapsto V^{(q)}f(x,i)$, for $q>0$. 

For the  direct Riemann integrability property, and without loss of generality, we assume that $|f(x,i)|\le \mathds{1}_{[0,1]\times\mathcal{I}}(x,i)$. We also introduce the step function
\[
g(x, i):=\inf_{y\in[k-2, k-1]}e^{\alpha y}V^{(q)}\mathds{1}_{[0,1]\times\mathcal{I}}(e^y, i) \qquad \textrm{for every $x\in [k,k+1)$ and $k\in \mathbb{Z}$.}
\]
and observe from Lemma \ref{prop: wd}  that
\[
\begin{split}
\int_{-\infty}^\infty g(x,i){\rm d} x&\le \int_{-\infty}^\infty e^{\alpha y} V^{(q)}\mathds{1}_{[0,1]\times\mathcal{I}}(e^y, i){\rm d}y\\
&\le C\sum_{j\in \mathcal{I}} \pi_j\int_{0}^\infty  V^{(q)}\mathds{1}_{[0,1]\times\mathcal{I}}(x, i) x^{\alpha-1}{\rm d}x\\
&=C\sum_{k\in \mathcal{I}} \pi_k\int_{0}^1 V^{\natural,(q)}\mathds{1}_{[0,\infty)\times\mathcal{I}}(y, k) y^{\alpha-1}{\rm d}y<\infty,
\end{split}
\]
where $C=\max_{j\in \mathcal{I}}\frac{1}{\pi_j}$.
Next, for $y\le x$, we observe from the scaling property that 
\[
\begin{split}
|V^{(q)}f(x,i)|&\le V^{(q)}\mathds{1}_{[0,1]\times \mathcal{I}}(x,i)\\
&\le \int_0^\infty e^{-qt}\mathbb{P}_{x,i}(X(t)\le 1){\rm d}t\\
&= \int_0^\infty e^{-qt}\mathbb{P}_{1,i}(X(tx^{-\alpha})\le 1/x){\rm d}t\\
&\le \int_0^\infty e^{-qt}\mathbb{P}_{1,i}(X(tx^{-\alpha})\le 1/y){\rm d}t\\
&= \int_0^\infty e^{-qt}\mathbb{P}_{y,i}(X(t(x/y)^{-\alpha})\le 1){\rm d}t\\
&=\left(\frac{x}{y}\right)^\alpha V^{(q)}\mathds{1}_{[0,1]\times \mathcal{I}}(y,i).
\end{split}
\]
In other words, we get $e^{\alpha x}|V^{(q)}f(e^x,i)|\le e^{4\alpha} g(x, i)$, for every $x\in \mathbb{R}$. Since the function $x\mapsto e^{\alpha x}V^{(q)}f(e^x,i)$ is continuous, the direct Riemann integrability follows. \end{proof}
Finally, we present a tightness property for the distribution of self-similar Markov processes with types.
\begin{Lem}\label{tight}Let $T$ be a random time with an exponential distribution that is independent of the self-similar Markov process with types $(X, J)$. Then the familly of probability measures on $[0,\infty)\times\mathcal{I}$, $\{\mathbb{P}_{x,j}((X(T), J(T))\in \cdot), 0<x\le 1, j\in \mathcal{I}\}$  is tight.
\end{Lem}
\begin{proof} Let $\overline{X}$ be the supremum process of $X$, that is
\[
\overline{X}(t)= \sup_{u\le t} X(u), \qquad t\ge 0,
\]
and let $\sigma^+(\cdot)=\inf\{t\ge 0: \overline{X}(t)> \cdot\}$ be its right-continuous inverse.  For $x\in (0,1]$ and $k\ge 1$, it is clear that 
\[
 \mathbb{P}_{x,i}\Big(\overline{X}(T)\le k^3\Big)\le \mathbb{P}_{x,i}(X(T)\le k^3).
\]
On the other hand, using the Markov property (in the third line),
\[
\begin{split}
 \mathbb{P}_{x,i}\Big(\overline{X}(T)\le k^3\Big)&\ge \mathbb{P}_{x,i}\Big(\overline{X}(\sigma^+(k)+T)\le k^3\Big)\\
& \ge \mathbb{P}_{x,i}\Big(\overline{X}(\sigma^+(k)+T)\le k^3, X(\sigma^+(k))\le k^2\Big)\\
 & =\sum_{j\in \mathcal{I}}\int_{[k,k^2]}\mathbb{P}_{x, i}(X(\sigma^+(k))\in{\rm d} y, J(\sigma^+(k))=j)\\
 &\hspace{5cm}\times\mathbb{P}_{y, j}\Big(\overline{X}(T)\le k^3\Big)\\
 & \ge \mathbb{P}_{x, i}(X(\sigma^+(k))\le k^2)\inf_{y\in [k, k^2], j\in \mathcal{I}}\mathbb{P}_{y, j}\Big(\overline{X}(T)\le k^3\Big).\\
 \end{split}
\]
From the scaling property, for every $y\in [k,k^2]$,
\[
\begin{split}
\mathbb{P}_{y, j}\Big(\overline{X}(T)\le k^3\Big)&=\mathbb{P}_{1, j}\Big(\overline{X}(Ty^{-\alpha})\le k^3/y\Big)\\
&\ge \mathbb{P}_{1, j}\Big(\overline{X}(T)\le k\Big), 
\end{split}
\]
which is close to $1$, uniformly for $y\in[k, k^2]$, by taking $k$ large enough.

Moreover, recalling the notation $\tau_a^+$ of \cref{mrenw}, from the Lamperti-type transform and Theorem \ref{mrenw} we may deduce that
\[
\begin{split}
\mathbb{P}_{x, i}(X(\sigma^+(k))\le k^2)&=\sum_{j\in\mathcal{I}}\mathbb{P}_{x, i}(X_{\sigma^+(k)}\le k^2, J_{\sigma^{+}(k)}=j)\\
&=\sum_{j\in\mathcal{I}}\mathtt{P}_{ i}\Big(\xi(\tau^+_{\log (k/x)})\le \log(k^2/x), \Theta(\tau^+_{\log (k/x)})=j\Big)\xrightarrow[k\to\infty]{} 1.
\end{split}
\]
Putting all pieces together,  allow us to deduce our claim, with a choice of compact given by $[0,k^3]\times\Ical$.
\end{proof}
We are now ready to prove Theorem \ref{entrancelaw}.
\begin{proof}[Proof of Theorem \ref{entrancelaw}]
For each $i\in\mathcal{I}$, we let $f(\cdot, i):[0,\infty)\to\mathbb{R}$ be a continuous function with compact support and $q>0$. It is clear that $y\mapsto e^{\alpha y}f(e^y, i)$ is directly Riemann integrable and from Lemma \ref{lempot}, $y\mapsto e^{\alpha y}V^{(q)}f(e^y,i)$ is also directly Riemann integrable. Thus from Lemma \ref{lemmrthm}, we have
\[
\lim_{x\to 0+} V^{(0)}f(x,i)=\frac{1}{m}\sum_{j\in \mathcal{I}}\pi_j\int_0^\infty f(y,j)y^{\alpha -1}{\rm d}y 
\]
and 
\[
 \lim_{x\to 0+} V^{(0)}V^{(q)}f(x,i)=\frac{1}{m}\sum_{j\in \mathcal{I}}\pi_j\int_0^\infty V^{(q)}f(y,j)y^{\alpha -1}{\rm d}y.
\]
From the resolvent equation $V^{(q)}f(x,i)=V^{(0)}f(x,i)-qV^{(0)}V^{(q)}f(x,i)$, see for instance identity (2.7) in Chapter 1 in Ethier and Kurtz \cite{EthierKurtz}, Proposition \ref{prop: wd} and identity \eqref{eq: zeta lifetime}, we deduce
\[
\begin{split}
\lim_{x\to 0^+} V^{(q)}f(x,i)&=\frac{1}{m}\sum_{j\in \mathcal{I}}\pi_j\int_0^\infty \Big(f(y,j)-qV^{(q)}f(y,j)\Big)y^{\alpha -1}{\rm d}y\\
&=\frac{1}{m}\sum_{j\in \mathcal{I}}\pi_j\int_0^\infty f(y,j)\Big(1-qV^{\natural, (q)}\mathds{1}(y,j)\Big)y^{\alpha -1}{\rm d}y\\
&=\frac{1}{m}\sum_{j\in \mathcal{I}}\pi_j\int_0^\infty f(y,j)\mathbb{E}^\natural_{y,j}[e^{-q\zeta}]y^{\alpha -1}{\rm d}y\\
&=\frac{1}{m}\sum_{j\in \mathcal{I}}\pi_j\int_0^\infty f(y,j)\mathtt{E}^\natural_{j}\Big[e^{-qy^{\alpha}I(\alpha\xi)}\Big]y^{\alpha -1}{\rm d}y.
\end{split}
\]
Here we used the definition of $\overline{V}^{(q)}$ (see \eqref{eq: def of resolvents}) in the third line.
Using Lemma \ref{tight}, we observe that  
\[
\alpha m=\sum_{j\in \mathcal{I}}\pi_j\mathtt{E}^\natural_{j}\left[\frac{1}{I(\alpha\xi)}\right],
\]
and that for every continuous  $f(\cdot, i):[0,\infty)\to \mathbb{R}$ with compact support,
\[
\begin{split}
\lim_{x\to 0+} V^{(q)}f(x,i)&=\frac{1}{m}\sum_{j\in \mathcal{I}}\pi_j\int_0^\infty f(y,j)\mathtt{E}^\natural_{j}\Big[e^{-qy^{\alpha}I(\alpha\xi)}\Big]y^{\alpha -1}{\rm d}y\\
&=\frac{1}{m}\sum_{j\in \mathcal{I}}\pi_j\int_0^\infty \mathtt{E}^\natural_{j}\Big[ f\left(\left(\frac{u}{I(\alpha\xi)}\right)^{1/\alpha},j\right)\frac{1}{I(\alpha\xi)}\Big]e^{-qu}{\rm d}u.
\end{split}
\]
This completes the proof.
\end{proof}

\bibliography{biblio}

\begin{thebibliography}{KKPW14}

\bibitem[ACGZ17]{ACGZ}
L.~Alili, L.~Chaumont, P.~Graczyk, and T~\.{Z}ak.
\newblock Inversion, duality and {D}oob {$h$}-transforms for self-similar
  {M}arkov processes.
\newblock {\em Electronic Journal of Probability}, {\bf 22}:Paper No. 20, 18,
  2017.

\bibitem[Als94]{Alsmeyer94}
Gerold Alsmeyer.
\newblock On the {M}arkov renewal theorem.
\newblock {\em Stochastic Process. Appl.}, 50(1):37--56, 1994.

\bibitem[Als14]{Alsmeyer14}
Gerold Alsmeyer.
\newblock Quasistochastic matrices and {M}arkov renewal theory.
\newblock {\em J. Appl. Probab.}, 51A(Celebrating 50 Years of The Applied
  Probability Trust):359--376, 2014.

\bibitem[AMN78]{AMN}
K.~Athreya, D.~McDonald, and P.~Ney.
\newblock Limit theorems for semi-{M}arkov processes and renewal theory for
  {M}arkov chains.
\newblock {\em The {A}nnals of {P}robability}, pages 788--797, 1978.

\bibitem[AS20]{AD}
E.~Aïdékon and W.~Da Silva.
\newblock Growth-fragmentation process embedded in a planar {B}rownian
  excursion.
\newblock {\em arXiv:2005.06372}, 2020.
\newblock \url{https://arxiv.org/pdf/2005.06372.pdf}.

\bibitem[Asm08]{Asm}
S{\o}ren Asmussen.
\newblock {\em Applied probability and queues}, volume~51.
\newblock Springer Science \& Business Media, 2008.

\bibitem[AW21]{AliWoo}
Larbi Alili and David Woodford.
\newblock On the finiteness and tails of perpetuities under a {L}amperti-{K}iu
  {MAP}.
\newblock {\em J. Appl. Probab.}, 58(4):1086--1113, 2021.

\bibitem[BBCK18]{BBCK}
J.~Bertoin, T.~Budd, N.~Curien, and I.~Kortchemski.
\newblock Martingales in self-similar growth-fragmentations and their
  connections with random planar maps.
\newblock {\em Probab. Th. Rel. Fields}, {\bf 172}:663--724, 2018.

\bibitem[BCK18]{BCK}
Jean Bertoin, Nicolas Curien, and Igor Kortchemski.
\newblock Random planar maps and growth-fragmentations.
\newblock {\em The Annals of Probability}, {\bf 46}(1):207--260, 2018.

\bibitem[Ber06]{Ber06}
J.~Bertoin.
\newblock {\em Random fragmentation and coagulation processes}.
\newblock Cambridge Studies in Advanced Mathematics,{\bf 102}, Cambridge
  University Press, Cambridge, 2006.

\bibitem[Ber17]{Ber-GF}
J.~Bertoin.
\newblock Markovian growth-fragmentation processes.
\newblock {\em Bernoulli}, {\bf 23}:1082--1101, 2017.

\bibitem[Big77]{Big}
J.~D. Biggins.
\newblock Martingale convergence in the branching random walk.
\newblock {\em Journal of {A}pplied {P}robability}, {\bf 14}(1):25--37, 1977.

\bibitem[BLR21]{BLR}
Anita Behme, Alexander Lindner, and Jana Reker.
\newblock On the law of killed exponential functionals.
\newblock {\em Electron. J. Probab.}, 26:Paper No. 60, 35, 2021.

\bibitem[BS20]{BehSid}
Anita Behme and Apostolos Sideris.
\newblock Exponential functionals of {M}arkov additive processes.
\newblock {\em Electron. J. Probab.}, 25:Paper No. 37, 25, 2020.

\bibitem[BS21]{BarSav}
A.~Barker and M.~Savov.
\newblock Bivariate {B}ernstein-gamma functions and moments of exponential
  functionals of subordinators.
\newblock {\em Stochastic Process. Appl.}, 131:454--497, 2021.

\bibitem[Bud16]{Budd}
Timothy Budd.
\newblock The peeling process of infinite boltzmann planar maps.
\newblock {\em The {E}lectronic {J}ournal of {C}ombinatorics}, {\bf 23}, 2016.

\bibitem[BY02]{BY2002}
Jean Bertoin and Marc Yor.
\newblock The entrance laws of self-similar {M}arkov processes and exponential
  functionals of {L}\'{e}vy processes.
\newblock {\em Potential Anal.}, 17(4):389--400, 2002.

\bibitem[BY05]{BY}
J.~Bertoin and M.~Yor.
\newblock Exponential functionals of {L}{\'e}vy processes.
\newblock {\em Probability {S}urveys}, {\bf 2}, 2005.

\bibitem[{\c{C}}in75]{Cin}
Erhan {\c{C}}inlar.
\newblock Lévy systems of {M}arkov additive processes.
\newblock {\em Zeitschrift f{\"u}r {W}ahrscheinlichkeitstheorie und {V}erwandte
  {G}ebiete}, {\bf 31}(3):175--185, 1975.

\bibitem[CPR13]{CPR}
Lo{\"\i}c Chaumont, Henry Pant{\'i}, and V{\'i}ctor Rivero.
\newblock The {L}amperti representation of real-valued self-similar {M}arkov
  processes.
\newblock {\em Bernoulli}, 19(5B):2494--2523, 2013.

\bibitem[CPY97]{CPY}
P.~Carmona, F.~Petit, and M.~Yor.
\newblock On the distribution and asymptotic results for exponential
  functionals of {L}{\'e}vy processes.
\newblock {\em Exponential functionals and principal values related to
  {B}rownian motion}, pages 73--121, 1997.

\bibitem[Dad17]{Dadoun}
B.~Dadoun.
\newblock Asymptotics of self-similar growth-fragmentation processes.
\newblock {\em Electronic Journal of Probability}, {\bf 22}:no. 27, 30, 2017.

\bibitem[DDK17]{DeDoKy}
Steffen Dereich, Leif D\"{o}ring, and Andreas~E. Kyprianou.
\newblock Real self-similar processes started from the origin.
\newblock {\em Ann. Probab.}, 45(3):1952--2003, 2017.

\bibitem[EK86]{EthierKurtz}
Stewart~N. Ethier and Thomas~G. Kurtz.
\newblock {\em Markov processes}.
\newblock Wiley Series in Probability and Mathematical Statistics: Probability
  and Mathematical Statistics. John Wiley \& Sons, Inc., New York, 1986.
\newblock Characterization and convergence.

\bibitem[Gol91]{Gol}
C.~Goldie.
\newblock Implicit renewal theory and tails of solutions of random equations.
\newblock {\em The {A}nnals of {A}pplied {P}robability}, pages 126--166, 1991.

\bibitem[Gri74]{Grin}
A.K. Grincevicjus.
\newblock On the continuity of the distribution of a sum of dependent variables
  connected with independent walks on lines.
\newblock {\em Theory of {P}robability \& {I}ts {A}pplications}, {\bf
  19}(1):163--168, 1974.

\bibitem[Iva11]{Iva}
Ivanovs.
\newblock {\em One-sided Markov additive processes and related exit problems}.
\newblock BOXPress, 2011.

\bibitem[JOC12]{JO-C}
P.~R. Jelenkovi{\'c} and M.~Olvera-Cravioto.
\newblock Implicit renewal theory and power tails on trees.
\newblock {\em Advances in {A}pplied {P}robability}, {\bf 44}(2):528--561,
  2012.

\bibitem[Kes73]{Kes}
H.~Kesten.
\newblock Random difference equations and renewal theory for products of random
  matrices.
\newblock {\em {A}cta {M}ath.}, pages 207--248, 1973.

\bibitem[Kes74]{Kesten74}
Harry Kesten.
\newblock Renewal theory for functionals of a {M}arkov chain with general state
  space.
\newblock {\em Ann. Probability}, 2:355--386, 1974.

\bibitem[Kin61]{Kin}
J.F.C. Kingman.
\newblock A convexity property of positive matrices.
\newblock {\em The {Q}uarterly {J}ournal of {M}athematics}, {\bf
  12}(1):283--284, 1961.

\bibitem[KKPW14]{KKPW}
A.~Kuznetsov, A.~E. Kyprianou, J.~C. Pardo, and A.~R. Watson.
\newblock The hitting time of zero for a stable process.
\newblock {\em Electronic Journal of Probability}, {\bf 19}:no. 30, 26, 2014.

\bibitem[KP76]{Kah-Pey}
J.-P. Kahane and J.~Peyrière.
\newblock Sur certaines martingales de {B}enoit {M}andelbrot.
\newblock {\em Advances in {M}athematics}, {\bf 22}(2):131--145, 1976.

\bibitem[KP21]{KP}
A.E. Kyprianou and J.C. Pardo.
\newblock {\em Stable {L}évy processes via {L}amperti-type representations}.
\newblock Cambridge University Press, Cambridge, 2021.

\bibitem[KS01]{KS}
A.~E. Kyprianou and A.~R. Sani.
\newblock Martingale convergence and the functional equation in the multi-type
  branching random walk.
\newblock {\em Bernoulli}, pages 593--604, 2001.

\bibitem[Lal84a]{Lal}
S.~Lalley.
\newblock Conditional {M}arkov renewal theory i. {F}inite and denumerable state
  space.
\newblock {\em The {A}nnals of {P}robability}, pages 1113--1148, 1984.

\bibitem[Lal84b]{Lalley84}
S.~P. Lalley.
\newblock Conditional {M}arkov renewal theory. {I}. {F}inite and denumerable
  state space.
\newblock {\em Ann. Probab.}, 12(4):1113--1148, 1984.

\bibitem[Lam72]{Lam}
John Lamperti.
\newblock Semi-stable {M}arkov processes, {I}.
\newblock {\em Zeitschrift f{\"u}r Wahrscheinlichkeitstheorie und verwandte
  Gebiete}, 22(3):205--225, 1972.

\bibitem[LGR20]{LG-Rie}
Jean-Fran{\c{c}}ois Le~Gall and Armand Riera.
\newblock Growth-fragmentation processes in {B}rownian motion indexed by the
  {B}rownian tree.
\newblock {\em The {A}nnals of {P}robability}, {\bf 48}(4):1742--1784, 2020.

\bibitem[Liu00]{Liu}
Q.~Liu.
\newblock On generalized multiplicative cascades.
\newblock {\em Stochastic processes and their applications}, {\bf
  86}(2):263--286, 2000.

\bibitem[Lyo97]{Lyo}
R.~Lyons.
\newblock A simple path to {B}iggins’ martingale convergence for branching
  random walk.
\newblock In {\em Classical and modern branching processes}, pages 217--221.
  Springer, 1997.

\bibitem[Man74]{Man}
B.~Mandelbrot.
\newblock Multiplications al{\'e}atoires it{\'e}r{\'e}es et distributions
  invariantes par moyenne pond{\'e}r{\'e}e al{\'e}atoire.
\newblock {\em CR {A}cad. {S}ci. {P}aris}, {\bf 278}(289-292):355--358, 1974.

\bibitem[MSW20]{MSW}
Jason Miller, Scott Sheffield, and Wendelin Werner.
\newblock Simple {C}onformal {L}oop {E}nsembles on {L}iouville {Q}uantum
  {G}ravity.
\newblock {\em arXiv preprint arXiv:2002.05698}, 2020.

\bibitem[PR13]{PR}
Juan~Carlos Pardo and V{\'i}ctor Rivero.
\newblock Self-similar {M}arkov processes.
\newblock {\em Bol. Soc. Mat. Mexicana (3)}, 19(2):201--235, 2013.

\bibitem[PS18]{PatSav}
Pierre Patie and Mladen Savov.
\newblock Bernstein-gamma functions and exponential functionals of {L}\'{e}vy
  processes.
\newblock {\em Electron. J. Probab.}, 23:Paper No. 75, 101, 2018.

\bibitem[Riv12]{Riv}
V.~Rivero.
\newblock Tail asymptotics for exponential functionals of {L}{\'e}vy processes:
  the convolution equivalent case.
\newblock {\bf 48}(4):1081--1102, 2012.

\bibitem[Riv16]{Rivero16}
V\'{\i}ctor Rivero.
\newblock Entrance laws for positive self-similar {M}arkov processes.
\newblock In {\em Mathematical {C}ongress of the {A}mericas}, volume 656 of
  {\em Contemp. Math.}, pages 119--140. Amer. Math. Soc., Providence, RI, 2016.

\bibitem[Sil21]{DS}
W.~Da Silva.
\newblock Self-similar signed growth-fragmentations.
\newblock {\em arXiv:2101.02582}, 2021.
\newblock \url{https://arxiv.org/abs/2101.02582}.

\bibitem[SP23]{DP23}
W.~Da Silva and J.C. Pardo.
\newblock Spatial growth-fragmentations and excursions from hyperplanes.
\newblock {\em Work in progress}, 2023.

\bibitem[Ste18]{Ste}
Robin Stephenson.
\newblock On the exponential functional of {M}arkov additive processes, and
  applications to multi-type self-similar fragmentation processes and trees.
\newblock {\em ALEA, Lat. Am. J. Probab. Math. Stat.}, {\bf 15}:1257–1292,
  2018.

\end{thebibliography}
\bibliographystyle{alpha}

\end{document}